\documentclass[10pt]{amsart}

\usepackage{amssymb,amsmath,amscd,amsthm,xspace,color,tocvsec2, array,
tikz-cd, enumitem}
\usepackage[mathscr]{euscript}
\usepackage[breaklinks=true, hidelinks]{hyperref}
\usepackage[all, cmtip]{xy}
\usepackage[ampersand]{easylist}
\usepackage{stmaryrd}
\usepackage[myheadings]{fullpage}
\usepackage{csquotes}
\usepackage{mathtools}
\usepackage[skip=5pt, indent=10pt]{parskip}

\newcommand\cA{\mathscr{A}}
\newcommand\cB{\mathscr{B}}
\newcommand\cC{\mathscr{C}}
\newcommand\cD{\mathscr{D}}
\newcommand\cH{\mathscr{H}}
\newcommand\cO{\mathscr{O}}

\newcommand\cT{\mathscr{T}}

\newcommand\cY{\mathscr{Y}}
\newcommand\cU{\mathscr{U}}
\renewcommand{\leq}{\leqslant}
\renewcommand{\geq}{\geqslant}

\def\bA{\mathbf{A}}
\def\bB{\mathbf{B}}
\def\bC{\mathbf{C}}
\def\bF{\mathbf{F}}
\def\bG{\mathbf{G}}

\def\bI{\mathbf{I}}
\def\bJ{\mathbf{J}}
\def\bK{\mathbf{K}}
\def\bL{\mathbf{L}}
\def\bM{\mathbf{M}}
\def\bN{\mathbf{N}}
\def\bP{\mathbf{P}}
\def\bQ{\mathbf{Q}}
\def\bR{\mathbf{R}}
\def\bS{\mathbf{S}}
\def\bV{\mathbf{V}}
\def\bZ{\mathbf{Z}}

\def\sfu{\mathsf{u}}
\def\sfB{\mathsf{B}}
\def\sfC{\mathsf{C}}
\def\sfG{\mathsf{G}}
\def\sfH{\mathsf{H}}

\def\sfU{\mathsf{U}}
\def\sfV{\mathsf{V}}
\def\sfR{\mathsf{R}}
\def\sfT{\mathsf{T}}
\def\sfX{\mathsf{X}}
\def\sfY{\mathsf{Y}}
\def\sfZ{\mathsf{Z}}

\newcommand{\beq}{\begin{equation}}
\newcommand{\eeq}{\end{equation}}

\theoremstyle{plain}
\newtheorem{theorem}[subsubsection]{Theorem}
\newtheorem{proposition}[subsubsection]{Proposition}
\newtheorem{lemma}[subsubsection]{Lemma}

\newtheorem{conjecture}[subsubsection]{Conjecture}

\theoremstyle{definition}
\newtheorem{remark}[subsubsection]{Remark}
\theoremstyle{definition}
\newtheorem{definition}[subsubsection]{Definition}

\newcommand{\on}{\operatorname}
\DeclareMathOperator{\fnt}{f}
\DeclareMathOperator{\Rep}{Rep}
\DeclareMathOperator{\RHom}{RHom}
\DeclareMathOperator{\Hom}{Hom}

\DeclareMathOperator{\CT}{CT}
\DeclareMathOperator{\cpt}{c}
\DeclareMathOperator{\Perf}{Perf}
\DeclareMathOperator{\oblv}{oblv}
\DeclareMathOperator{\IndPerf}{IndPerf}
\DeclareMathOperator{\QCoh}{QCoh}
\DeclareMathOperator{\IndCoh}{IndCoh}

\DeclareMathOperator{\pt}{pt}

\DeclareMathOperator{\GL}{GL}
\DeclareMathOperator{\SL}{SL}
\DeclareMathOperator{\reg}{reg}

\DeclareMathOperator{\ind}{ind}
\DeclareMathOperator{\gr}{gr}
\DeclareMathOperator{\cone}{cone}

\DeclareMathOperator{\image}{im}

\DeclareMathOperator{\colim}{colim}

\DeclareMathOperator{\Ext}{Ext}

\DeclareMathOperator{\soc}{soc}

\DeclareMathOperator{\Spec}{Spec}
\DeclareMathOperator{\QSpec}{Spec^{rel}}
\DeclareMathOperator{\Free}{Free}
\DeclareMathOperator{\Frac}{Frac}

\DeclareMathOperator{\Shv}{Shv}
\DeclareMathOperator{\Mod}{Mod}
\DeclareMathOperator{\PGL}{PGL}
\DeclareMathOperator{\module}{mod-}

\DeclareMathOperator{\Bim}{Bim}
\DeclareMathOperator{\Perv}{Perv}

\DeclareMathOperator{\waki}{waki}
\DeclareMathOperator{\Fl}{Fl}

\DeclareMathOperator{\K}{K}
\DeclareMathOperator{\bnd}{b}

\DeclareMathOperator{\Z}{Z}
\DeclareMathOperator{\Gr}{Gr}
\DeclareMathOperator{\id}{id}

\DeclareMathOperator{\Sym}{Sym}

\DeclareMathOperator{\aff}{aff}

\DeclareMathOperator{\sct}{sc}
\DeclareMathOperator{\ad}{ad}

\DeclareMathOperator{\act}{act}

\DeclareMathOperator{\Transp}{Transp}

\DeclareMathOperator{\Bdry}{Bdry}
\DeclareMathOperator{\IC}{IC}
\DeclareMathOperator{\mix}{mix}

\DeclareMathOperator{\tilt}{tilt}
\DeclareMathOperator{\coact}{coact}

\newcommand{\bs}{\backslash}
\newcommand{\G}{\sfG}
\newcommand{\wG}{\widetilde{\G}}

\newcommand{\Hs}{\mathscr{H}}
\newcommand{\wHws}{ {}_\chi\mathscr{H}_\chi}
\newcommand{\wHs}{ {}_\chi\mathscr{H}}
\newcommand{\Hws}{\mathscr{H}_{\chi}}

\newcommand{\Hwfs}{\mathscr{H}^{\fnt}_\chi}
\newcommand{\Hfs}{\mathscr{H}^{\fnt}}
\newcommand{\wHfs}{{}_\chi\mathscr{H}^{\fnt}}
\newcommand{\wHwfs}{{}_\chi\mathscr{H}^{\fnt}_\chi}

\title[Universal Arkhipov--Bezrukavnikov]{The universal monodromic Arkhipov--Bezrukavnikov equivalence}
\author{Gurbir Dhillon and Jeremy Taylor}
\date{}

\begin{document}
\begin{abstract}
Working with integral coefficients, we identify universal monodromic Iwahori--Whittaker sheaves on the enhanced affine flag variety with ind-perfect complexes on the Grothendieck--Springer stack of the Langlands dual group. This extends a similar result for quasi-coherent sheaves on the Springer stack due to Arkhipov--Bezrukavnikov. We further prove a monoidal equivalence between bi-Iwahori--Whittaker sheaves and ind-perfect complexes on the adjoint quotient of the Langlands dual group. These results are used in the sequel to this paper to prove the tame local Betti geometric Langlands conjecture of Ben-Zvi--Nadler.

Our proof of fully faithfulness provides an alternative to the argument of Arkhipov--Bezrukavnikov. Namely, while they localize in unipotent directions, we localize in semi-simple directions, thereby reducing fully faithfulness to an order of vanishing calculation in semi-simple rank one and characteristic zero.
\end{abstract}
\maketitle
\setcounter{tocdepth}{1}

\vspace{-0.5cm}

\tableofcontents

\section{Introduction}

\subsection{Overview}
In the paper \cite{BZN}, Ben-Zvi--Nadler introduced the Betti geometric Langlands program, a version of geometric Langlands in which the coherent side concerns the usual character varieties of topological surfaces, i.e., the representation varieties of their fundamental groups.\footnote{In the case of surfaces without punctures, i.e., the global unramified case, their conjecture is now a theorem \cite{ABC}.}

Among their many interesting conjectures, the basic local assertion is an identification of two universal versions of the affine Hecke category; this statement in fact had appeared in their earlier work \cite{BZN07}. Here, in terms of character varieties, local means near a given puncture in the surface, and universal means that one does not fix the monodromy of local systems around the puncture, but rather allows them to vary through all conjugacy classes in the group.

As they emphasized in their initial paper, when one restricts from all conjugacy classes to unipotent ones, their local conjecture was already a celebrated theorem of Bezrukavnikov \cite{B}. This result of Bezrukavnikov, in turn, built on an earlier work joint with Arkhipov \cite{AB}.

The goal of this paper and its sequel is to supply a proof of Ben-Zvi--Nadler's  local conjecture, i.e., tame local Betti geometric Langlands. Accordingly, in this first paper, we provide the deformation of the results of \cite{AB} over all conjugacy classes, and in its sequel, we provide the deformation of \cite{B}.  

A direct deformation of the arguments of Arkhipov and Bezrukavnikov seems to present nontrivial challenges at some key steps, cf. Remark \ref{r:brr} below. Partly for this reason, we have chosen to provide some alternative arguments at various points; we hope these provide a useful supplement to their perspective.

In the remainder of this introduction, we review for context the statement of tame local Betti geometric Langlands in Section \ref{s:conj}, provide a precise statement of our results in Section \ref{ss:results}, and discuss some aspects of the argument in Section \ref{ss:idea}.

\subsection{Statement of tame local Betti geometric Langlands}
\label{s:conj}

It is both helpful as context and requires less technical input to formulate the deformation of Bezrukavnikov's equivalence, rather than that of Arkhipov--Bezrukavnikov, so this is where we begin.

 In what follows, let $G$ be a pinned complex reductive group. Write $\sfG$ for the Langlands dual group defined over the integers.

On the constructible side, we denote by $\Fl$ the enhanced affine flag variety of $G$, viewed as an ind-analytic variety over the complex numbers. As such, we may consider the derived category of weakly Iwahori--constructible sheaves of $\bZ$-modules on $\Fl$ in the analytic topology. We denote this category, which is naturally monoidal under convolution, by $\cH \coloneqq \Shv_{(I)}(\Fl)$.

\begin{remark} One can get a good feel for this category, in particular its fine dependence on the chosen sheaf theory, simply by considering the monoidal unit. This object, for the usual reasons, is supported on the trivial Iwahori orbit. The full subcategory of $\cH$ supported on this orbit is canonically equivalent to the category of all local systems on the abstract Cartan $T$. Under the tautological identification of the group algebra $R := \bZ[\pi_1(T)]$  with the ring of functions on the dual torus $\sfT$,
we may equivalently think of this category as all quasi-coherent sheaves on $\sfT$. Under this correspondence, the monoidal unit corresponds to the structure sheaf of $\sfT$, and individual character sheaves on $T$ correspond to the skyscraper sheaves at closed points of $\sfT$. In this sense, the monoidal unit is the direct integral of all the character sheaves on $T$, and is a prototypical object with universal monodromy.  \end{remark}

Allowing arbitrary monodromy along the $T$-orbits corresponds under Koszul duality to replacing $\sfT$-equivariant cohomology by $\sfT$-equivariant K-theory as is studied using Soergel bimodules over $R$ in \cite{E, EE}.

On the dual side, we write $\sfB$ for the Borel subgroup of $\sfG$, and consider the tautological fiber product of adjoint quotient stacks $\sfB / \sfB \times_{\sfG / \sfG} \sfB / \sfB.$ Its category of ind-coherent sheaves, which we denote by
$$\IndCoh(\sfB/\sfB \times_{\sfG/\sfG} \sfB/\sfB)$$ is naturally monoidal under convolution.

Having introduced notation, let us state Ben-Zvi--Nadler's local conjecture, and explain its relation to Bezrukavnikov's equivalence. The conjecture reads as follows.

\begin{conjecture}[\cite{BZN07, BZN}]\label{TameBetti}
There is a monoidal equivalence $$\cH \simeq \IndCoh(\sfB/\sfB \times_{\sfG/\sfG} \sfB/\sfB).$$
\end{conjecture}

 To explain the connection with Bezrukavnikov's equivalence, we need to first recall that both categories are naturally linear over $R \otimes R$. On the constructible side, the linearity over $R \otimes R$ is afforded by the monodromy of sheaves along left and right Iwahori cosets. On the coherent side, the linearity over $R \otimes R$ comes from the two projections to $\sfB/\sfB$, postcomposed with the tautological map to the invariant theory quotient $\sfB / \sfB \rightarrow \sfB /\!\!/ \sfB \simeq \sfT$.
 We recall this again has an interpretation in terms of monodromy. Namely, the Steinberg stack is the moduli space of $\sfG$ local systems on a cylinder with $\sfB$ reductions around the two punctures. Under this identification, the maps to $\sfT$ simply record the semi-simplified monodromy around each of the two punctures.

Let us specialize to the formal completion of the identity in $\sfT \times \sfT$. Then the constructible side becomes pro-unipotent version of the affine Hecke category. The matching of these two specializations as monoidal categories is precisely Bezrukavnikov's equivalence \cite{B}, which was extended to modular coefficients in \cite{BRR, BR22, BR24} (with the exception of some small primes).

In \cite{DT2} we prove Conjecture \ref{TameBetti} using the results of the present paper, which we now turn to.

\subsection{Statement of results}
\label{ss:results}

Our two main results are  Theorems \ref{t:main1} and \ref{t:main2} below. Before giving their formal statements, we would like to first recall some relevant context.

At a high level, the argument of Bezrukavnikov establishing the main result of \cite{B} parallels an earlier argument of Kazhdan--Lusztig at the function theoretic level, identifying two Langlands dual realizations of the affine Hecke algebra \cite{KL}.

We recall that the basic pattern of both arguments is visible in the following toy model. Consider a map of finite sets $X \rightarrow Y$. If we pass to the algebras of functions, pullback defines a map of algebras $\on{Fun}(Y) \rightarrow \on{Fun}(X)$, and one has $$\on{End}_{\on{Fun}(Y)}(\on{Fun}(X)) \simeq \on{Fun}(X \times_Y X).$$ That is, the convolution algebra $\on{Fun}(X \times_Y X)$ emerges from the more basic action of $\on{Fun}(Y)$ on $\on{Fun}(X)$. We emphasize that similar mechanisms, e.g. the action of a group on its flag variety, produce the appearances of Hecke algebras and categories throughout representation theory. In any case, given an algebra $A$, the datum of a map of algebras $\phi: A \rightarrow \on{Fun}(X \times_Y X)$ is the same as an action of $A$ on $\on{Fun}(X)$ commuting with the action of $\on{Fun}(Y)$, and for $\phi$ to be an isomorphism is the assertion that the action is faithful, with image the full centralizer of the action of $\on{Fun}(Y)$. That is, in general, to prove an isomorphism of algebras $A_1 \simeq A_2$, it is of course enough to realize them as the same set of operators acting on the same vector space. However, when one of the algebras is presented as functions on a fiber product $\on{Fun}(X \times_Y X)$, it admits a preferred choice of representation, namely $\on{Fun}(X)$, and one additionally has an explicit description of the centralizer of the action, namely the image of $\on{Fun}(Y)$ in $\on{Fun}(X \times X)$.

When we replace $X \rightarrow Y$ as above with $\mathsf{U} / \sfB \rightarrow \sfG / \sfG$, and functions with Grothendieck groups of coherent sheaves, Kazhdan--Lusztig essentially found that $K_0(\mathsf{U} / \sfB)$ is naturally identified with the underlying vector space of the antispherical representation of the affine Hecke algebra, i.e., Iwahori--Whittaker functions on affine flags. Moreover, under this identification, $K_0(\sfG / \sfG)$ acted as the center of the affine Hecke algebra, which then yielded the identification of the $K$-theory of the Steinberg with the affine Hecke algebra itself.\footnote{In truth, we recall one needs to consider a further $\bG_m$-grading, to account for the parameter $q$ in the affine Hecke algebra.} However, the double centralizer property did not hold literally in their setting, which complicates the argument. As we will see below, these complications disappear at the sheaf-theoretic level.

When we pass from Grothendieck groups of coherent sheaves to the categories of coherent sheaves themselves, it is again true that given a map $X \rightarrow Y$ between reasonable stacks over a field of characteristic 0, one has
$$\on{End}_{\QCoh(Y)}(\QCoh(X)) \simeq \QCoh(X \times_Y X),$$as shown by Ben-Zvi--Francis--Nadler \cite{BFN10}. In particular, given a monoidal category $\mathscr{A}$, the datum of a monoidal functor $\mathscr{A} \rightarrow \QCoh(X \times_Y X)$ is the same as an action of $\mathscr{A}$ on $\on{QCoh}(X)$ commuting with the action of $\on{QCoh}(Y)$.
With this in mind, to identify coherent sheaves on Steinberg with the affine Hecke category, one should in particular show that the natural categorifications of the antispherical representation on the constructible and coherent sides, namely Iwahori--Whittaker sheaves on affine flags and $\IndPerf(\sfU / \sfB)$, respectively, are still canonically equivalent. I.e., one again wants to identify both monoidal categories as the `same' endofunctors of the `same' category. The assertion that the two natural modules are still the `same' is the main result of \cite{AB}.

In the universal case, the preceding discussion suggests we should replace $\sfU/\sfB$ with $\sfB/\sfB$, and match it with a suitably defined category of universal Iwahori--Whittaker sheaves on $\Fl$. The latter is slightly delicate to set up, but we do so and denote it by ${}_{\chi}\cH.$ The first task of this paper is then to prove the following.  

\begin{theorem} \label{t:main1} There is a canonical equivalence of categories
$${}_\chi \cH \simeq \IndPerf(\sfB/\sfB).$$    
\end{theorem}

To motivate our second main result, note that to obtain from Theorem \ref{t:main1} a monoidal functor
$$\cH \rightarrow \IndPerf(\sfB / \sfB \times_{\sfG / \sfG} \sfB / \sfB),$$
 it remains to argue why the action of $\cH$ on $\IndPerf(\sfB / \sfB)$ commutes with the action of $\IndPerf(\sfG / \sfG)$.

\begin{remark} We note this direct naive approach to producing the functor between the two sides\footnote{Strictly speaking, of course we are producing the coherent side up to renormalization issues, i.e., obtaining the further projection from $\IndCoh$ to $\IndPerf$.} of Conjecture \ref{TameBetti}, which is the one used in \cite{DT2}, mildly differs from the strategy of \cite{B}. But, we should emphasize, it differs even less mildly from unpublished approaches of Gaitsgory--Lurie, Ben-Zvi--Nadler, and Arinkin--Bezrukavnikov to derived geometric Satake; the first of these approaches was recently realized by Campbell--Raskin \cite{CR}.  
\end{remark}

To produce a datum of commutativity between the actions of $\cH$ and $\IndPerf(\sfG / \sfG)$, we should look for a constructible realization of the action of $\IndPerf(\sfG / \sfG)$ on ${}_\chi \cH$. In fact, we show it is simply the convolution action of bi-Iwahori--Whittaker sheaves on the loop group of $G$, which we denote by $_{\chi}\cH_{\chi}$. We recall this is the most obvious candidate, i.e., is a monoidal category acting on ${}_\chi \cH$ by convolution on the `other side' from the affine Hecke category. Moreover, this matching is actually forced by more general expectations in local geometric Langlands.

At all events, our second basic result in this paper is the following.

\begin{theorem} \label{t:main2} There is an equivalence of monoidal categories
$${}_{\chi}\cH_\chi \simeq \IndPerf(\sfG/\sfG),$$
    compatible with their actions on both sides of Theorem \ref{t:main1}.
\end{theorem}

In the case of pro-unipotent monodromy, a version of Theorem \ref{t:main2} was first observed by Bezrukavnikov \cite{BezNilp}. Slightly more carefully, his work was not explicitly concerned with the monoidal structure, but rather matched the perverse t-structure on the constructible side with the perverse coherent t-structure on the coherent side, and developed implications for two-sided cells in the affine Hecke category. An explicitly monoidal equivalence appeared later in the work of Chen and the first named author \cite{CD}, building on unpublished ideas of Arinkin--Bezrukavnikov.

Given the two results of this paper, in \cite{DT2} we essentially obtain Conjecture \ref{TameBetti} by a double centralizer argument.

\subsection{Structure of the argument}
\label{ss:idea}
We now discuss the proofs of Theorems \ref{t:main1} and \ref{t:main2}, with a view towards to what we leave undisturbed from \cite{AB}, and we have opted to change.

The argument of \cite{AB} roughly consists of two steps. In the first step, they construct a functor from the coherent side to the constructible side, and in the second step they prove that it is an equivalence. In short, there are very good reasons not to disturb the first step, so we do not. However the second step can be modified, and we choose to give a different argument.

We now recall, in slightly more detail, the basic strategy of \cite{AB}, stated {\em mutatis mutandis} in the universal setting. On the coherent side, the composition
$$\IndPerf(\sfB/\sfB) \xrightarrow{\Delta_*}  \IndCoh(\sfB/\sfB \times_{\sfG/\sfG} \sfB/\sfB) \xrightarrow{\pi_*} \IndPerf(\sfB/\sfB)$$
is tautologically equivalent to the identity, where $\Delta$ and $\pi$ respectively denote the diagonal and second projection. Note that $\Delta_*$ is monoidal, with respect to the usual tensor product on the source and convolution on the target, and similarly $\pi_*$ can be interpreted as convolution with the structure sheaf of $\sfB/\sfB$. To orient the reader, we emphasize that the basic pattern here follows its analogue in the toy model of $\on{Fun}(X \times_Y X)$.

From the identification of the Grothendieck groups with the antispherical representation, we expect the structure sheaf of $\sfB/\sfB$ should be exchanged with the `simplest' object on the construcible side, namely the Iwahori--Whittaker sheaf with minimal support. On the constructible side, convolution with this object amounts to performing Whittaker averaging $\cH \rightarrow {}_\chi \cH.$ Therefore assuming Conjecture \ref{TameBetti}, we can rewrite this as
$$\IndPerf(\sfB/\sfB) \xrightarrow{\Delta_*} \IndCoh(\sfB/\sfB \times_{\sfG/\sfG} \sfB/\sfB) \simeq \cH \xrightarrow{} {}_\chi \cH.$$
The first step of \cite{AB} is the direct construction of the monoidal functor $\IndPerf(\sfB/\sfB) \rightarrow \cH,$ and the second step is verifying that the composite $\IndPerf(\sfB/\sfB) \rightarrow {}_\chi \cH$ is indeed an equivalence.

\subsection{Structure of the argument: part \ref{Part1}}\phantom{hi}

\subsubsection*{Arkhipov--Bezrukavnikov's construction} Let us recall the idea of part \ref{Part1} of the argument. Ignoring at first the monodromy of the local systems, one wants to construct a monoidal functor $\on{Rep}(\sfB) \rightarrow \cH$. This is difficult essentially due to the complexity of $\Rep(\sfB)$. The nearby categories $\Rep(\sfG)$ and $\Rep(\sfT)$ are somewhat simpler, and one may use them via the following maneuver, sometimes called {\em Drinfeld--Pl\"{u}cker formalism}. Namely, one considers the usual parabolic induction diagram $\sfT \leftarrow \sfB \rightarrow \sfG$, and notes that in terms of this $$\pt/\sfB \simeq
 \sfG \bs (\sfG / \sfU) / \sfT,$$ compatibly with maps to $\pt/\sfG$ and $\pt / \sfT$. Let us pretend that $\sfG/\sfU$ is affine, so that $\Rep(\sfB)$ is approximately the same as $\Rep(\sfB)' := \cO(\sfG/\sfU)\on{-mod}(\Rep(\sfG \times \sfT))$. Given the simple structure of $\cO(\sfG/\sfU)$ as an algebra, as dictated by the Peter--Weyl theorem, a monoidal functor out of  $\Rep(\sfB)'$ is the same data as a monoidal functor from $\Rep(\sfG \times \sfT)$, along with appropriate highest weight arrows between the objects of $\Rep(\sfG)$ and $\Rep(\sfT)$ indexed by the same dominant weights.

In the setting of \cite{AB}, one already has Gaitsgory's functor $\Rep(\sfG) \rightarrow \cH$, constructed via geometric Satake and the factorization action of the affine Grassmannian on affine flags, i.e., via nearby cycles. One similarly already has a functor $\Rep(\sfT) \rightarrow \cH$, essentially obtained from the torus case by parabolic induction; the essential image are the Wakimoto sheaves. As \cite{AB} observe, one has natural highest weight arrows of the required form, essentially for support reasons, and this gives the desired functor from $\Rep(\sfB)$; the monodromy of nearby cycles then provides the desired lift to $\IndPerf(\sfB / \sfB)$.

\begin{remark} Some of the fundamental ideas in the above argument fit into subsequent language as follows: for an algebraic group $\mathsf{H}$, an $E_2$-action of the symmetric monoidal category $\Rep(\mathsf{H})$ is the same data as a usual, i.e., $E_1$-, action of its integral over the circle, namely
$\IndPerf(\mathsf{H}/\mathsf{H})$.
\end{remark}

More truthfully, this argument is complicated by the non-affineness of $\sfG/\sfU$, and circumventing this is one of the technical feats of \cite{AB}. Their basic idea is as follows. While $\sfG / \sfU$ is not affine, it is not so far from it. Namely, the map to the affinization $\sfG / \sfU \rightarrow (\sfG / \sfU)^{\aff}$ is an open embedding, and the argument so far has actually produced a functor $F$ out of \[\Rep(\sfB)' \simeq \IndPerf( \sfG \bs (\sfG / \sfU)^{\aff} / \sfT).\]
To argue $F$ really does factor through restriction to $\IndPerf(\sfG \bs (\sfG / \sfU) / \sfT)$, one must further check that $F$ kills all objects supported on the irrelevant boundary of $(\sfG / \sfU)^{\aff}$, i.e., those objects whose restrictions to the open orbit vanish. By equivariance, this is the same as asking that the fiber at the identity coset in $(\sfG / \sfU)^{\aff}$ should vanish. The latter fiber retains an action of $\sfT$, and the resulting functor $\on{Rep}(\sfB) \rightarrow \on{Rep}(\sfT)$ is given by semi-simplification, i.e., passing to the associated graded representation. For this reason, \cite{AB} show that highest weight arrows out of central sheaves are the last step in Wakimoto filtrations. They then inspect the functor of taking the associated graded and by a Tannakian argument identify it with the above fiber functor, which completes the first step of the argument.

\subsubsection*{Applications} In applications, the form of Arkhipov--Bezrukavnikov's approach to part \ref{Part1} of the argument is crucial. This is notably the case in the work of Frenkel--Gaitsgory on localization theory for affine Lie algebras at critical level, essentially due to the similar use of factorization, i.e. vertex operator algebraic, methods in that theory \cite{FG}. In extensions of their work to treat general highest weight modules, which will be performed elsewhere, one needs again the compatibility with factorization; see e.g. the interesting work of F\ae rgeman on the motivicity of rigid local systems for recent uses of this expected extension \cite{Fae}.

Relatedly, one expects similar nearby cycles constructions for affine Hecke categories of Bernstein blocks of arbitrary depth, and for local global compatibilities, as in the work of Genestier--Lafforgue in arithmetic \cite{GL}, it will likely be important that coherent descriptions of the blocks are compatible with the action of $\IndPerf(\sfG / \sfG)$.

\subsubsection*{Universal monodromic deformation} The practical consequence for us of the preceding discussion is that the first step of \cite{AB}, i.e. the construction of the functor, should not be modified. Accordingly, in Part \ref{Part1} of this paper we perform the necessary checks that this indeed works, i.e. deforms over all conjugacy classes in the desired way.

Here, most of the arguments and constructions of \cite{AB}, as well as those of Gaitsgory \cite{GZ}, essentially deform straightforwardly to universal monodromy. We now briefly describe the two closely related exceptions. The first is the verification of the Pl\"{u}cker relations, performed in Section \ref{s:Plucker}. The problem here is that writing down a canonical highest weight arrow in the universal setting is slightly more delicate, because loop rotation no longer fixes the standard base points where one rigidifies the sheaves. The second exception is checking that the associated graded of the Wakimoto filtrations on central sheaves takes the expected form, performed in Section \ref{Fiberfunctor}. The problem here is that, after running the Tannakian formalism, it remains to further identify a certain $\sfG$-bundle over $\sfT / \sfT$. In Arkhipov-Bezrukavnikov's non-monodromic setting, the analogous $\sfG$-bundle is over $\pt / \sfT$, so when the coefficient field is algebraically closed the problem is vacuous.

\subsubsection*{Extension to integral coefficients}
Extending the Arkhipov--Bezrukavnikov functor to integral coefficients introduces an additional obstacle which we will now discuss. For $G = \GL(n)$ and modular field coefficients (with the exception of some small primes), Achar and Riche have already overcome this obstacle in \cite{AR}. Moreover Bezrukavnikov and Riche extended Bezrukavnikov's equivalence to modular field coefficients (with the exception of some small primes) using different methods in \cite{BR24}.

Arkhipov and Bezrukavnikov describe the coherent side of their equivalence as a quotient category: equivariant modules over a certain multi-graded ring modulo complexes supported on the irrelevant boundary.
They first construct a functor out of the additive category of free modules over this multi-graded ring. Then they pass to bounded homotopy categories and prove that complexes supported on the irrelevant boundary are killed.

For integral coefficients, it is better to only consider free modules that are pulled back from tilting representations.
But the bounded homotopy category of such free modules no longer recovers the category of perfect complexes. Indeed, if the multi-graded ring does not admit a costandard filtration, then there may exist higher Exts between the free modules.

To overcome this obstacle, we prove that all \textit{sufficiently dominant} weight spaces of the multi-graded ring admit costandard filtrations. Therefore there exists a submonoid containing all sufficiently dominant weights, such that the corresponding multi-graded subring admits a costandard filtration. Discarding these lower weight spaces does not change the relevant open subscheme.

Therefore we are able to extend the Arkhipov--Bezrukavnikov functor to integral coefficients by restricting to this subring before passing to homotopy categories. This extends results of \cite{AB, AR} uniformly to all reductive groups and coefficient fields.

\subsection{Structure of the argument: part \ref{Part2}}
With the monoidal functor $\IndPerf(\sfB / \sfB) \rightarrow \cH$ in hand, it remains to do the following two things.

\begin{enumerate}[label=(\roman*)]

\item \label{StructureStep1} First, one needs to construct an appropriate Iwahori--Whittaker category ${}_\chi \cH$, equipped with an action of $\cH$ and an averaging functor $\cH \rightarrow {}_\chi \cH$. I.e., we need the analogue of the antispherical representation along with its canonical generator.

\item \label{StructureStep2} Second, one needs to show the composite map $\IndPerf(\sfB/\sfB) \rightarrow {}_\chi \cH$ is an equivalence.

\end{enumerate}

Let us describe our approach to both points.

\subsubsection*{Whittaker sheaves} In our estimation, step \ref{StructureStep1} is deceptively subtle, and may well be why the equivalences proven in this paper and its sequel were not obtained earlier. The problem is as follows. At a fixed semi-simplified monodromy, one has many standard recipes for defining Whittaker categories. Namely, for $\ell$-adic sheaves, one may work with an affine flag variety over a field of positive characteristic, and use the Artin--Schreier sheaf. For D-modules, one may use the exponential D-module. For Betti sheaves at a fixed monodromy, Gaitsgory introduced a version of the Kirillov model, familiar from the representation theory of $GL_2(\mathbb{F}_q)$, which has the desired behavior \cite{GWhit}.

When we work with Betti sheaves with varying monodromy, none of the above approaches work: $\ell$-adic, de Rham, and Betti families of local systems give genuinely different moduli spaces, which rules out the first two approaches. For the remaining Kirillov model, one needs to impose a certain $\mathbf{C}^\times$-equivariance in its construction, which is incompatible with universal monodromy. Constructing the desired universal Whittaker model is therefore one of the basic problems one must confront here.

Fortunately, the solution was essentially supplied in \cite{LNY, IY, T}, where the authors defined the Whittaker functional on universal monodromic Hecke categories in terms of vanishing cycles, and proved that it is monoidal; such a microlocal approach to Whittaker functionals appears also in \cite{N06, NT, FR}. We use monoidality of the Whittaker functional to construct the Whittaker module for the universal finite Hecke category. The desired category ${}_{\chi}\cH$ is then obtained by induction, exactly as one defines the antispherical representation of the affine Hecke algebra.

\subsubsection*{Fully faithfulness} Let us now turn to step \ref{StructureStep2}, i.e., arguing the obtained functor is an equivalence.  Here, it is helpful to recall two extremely influential ideas of Soergel, first deployed together in the Category $\mathscr{O}$ \cite{S90}. The first idea, known as Soergel's Struktursatz, is that to calculate homomorphisms between tilting objects or their variants, the answers are unchanged after passing to the quotient by all objects whose singular support does not touch the regular nilpotent orbit. The second idea, known as {Soergel's deformation philosophy}, is essentially that the answers will also be flat in the highest weight, so that one may perform calculations after deforming to nearby simpler categories.

In some sense, \cite{AB} apply Soergel's first idea to prove step \ref{StructureStep2}, whereas we apply Soergel's second idea. That is, they localize in the nilpotent directions, whereas we localize in the semi-simple ones. Let us say this all more precisely, beginning with the approach of \cite{AB}.

\subsubsection*{Localization in the nilpotent directions} \label{ss:RegCent}Momentarily, suppose one already knows an equivalence between $\IndCoh_{\on{nilp}}(\sfU / \sfB \times_{\sfG/\sfG} \sfU / \sfB)$ and the affine Hecke category. Within $\sfU/\sfB$, one has the unique open orbit of regular unipotent elements $\sfU^{\on{reg}}/\sfB \simeq \pt / \mathsf{Z}$, where $\mathsf{Z}$ is the centralizer of any particular regular unipotent element $e$. It follows that one has a monoidal localization $$\IndCoh_{\on{nilp}}(\sfU/\sfB  \times_{\sfG/\sfG} \sfU/\sfB) \rightarrow \IndCoh_{\on{nilp}}( \pt / \sfZ \times_{\sfG/\sfG} \pt / \sfZ).$$
The latter open substack has much simpler geometry - as the centralizer of $e$ in $\sfG$ is again given by $\mathsf{Z}$, the latter fiber product is essentially $\on{\pt / \sfZ}$, up to a derived nil-thickening arising from the self-intersection.\footnote{In fact, the derived thickening comes  exactly from the semi-simple deformation directions we are concerned with.}

To reduce calculations to inspecting this regular locus, and in particular the fully faithful embedding from $\on{Rep}(\sfZ)$ to the Iwahori--Whittaker category, one would like to explicitly identify this simple quotient in constructible terms. This was done, among other things, by Bezrukavnikov in the work \cite{BI}, motivated by work of Lusztig and Vogan, particularly the work of Lusztig on two-sided cells in the affine Hecke category, where it corresponds to the smallest two sided cell \cite{Lus89, V}. We recall that these considerations also have a direct analogue in arithmetic, namely how the center of the affine Hecke algebra acts on subquotients of unramified principal series, with the principal orbit corresponding to the character of the Steinberg representation.

This was the localization in nilpotent directions used by \cite{AB}. Namely, they prove fully faithfulness via further restriction to this microlocalization, which they identify with $\on{Rep}(\sfZ)$ using a Tannakian argument.

\begin{remark}\label{r:brr} Extending the strategy of \cite{AB} to the universal monodromic case, i.e., directly proving an equivalence between $\sfB^{\on{reg}}/\sfB$ with a suitable microlocalization of $\cH$, does not appear to be straightforward. The obstruction is explained by Bezrukavnikov--Riche in the discussion following Theorem 1.3 of \cite{BR22}. Namely the general result, Proposition 1 of \cite{BI}, that was used in \cite{AB} to construct the fiber functor out of the microlocalized category, does not appear to work in families.  We expect, however, that it may be possible to factor the fiber functor through microlocalization using an affine nondegenerate Whittaker model as in \cite{DLYZ}.
\end{remark}

In contrast to the approach of \cite{AB}, the approach we will describe below does not seem to have a direct analogue at the function theoretic level, as characters of an Iwahori form a discrete group, i.e., no longer have moduli.

\subsubsection*{Localization in the semi-simple directions}

In the present work, we will perform calculations following Soergel's second idea, i.e., the deformation philosophy. Namely, we first check that various spaces of homomorphisms on both sides of the equivalence are vector bundles over $\sfT$, i.e., do not jump as the semi-simplified monodromy varies.
On the constructible side, this is a consequence of the fact that for costandard representations of $\sfG$ the corresoping Whittaker averaged central sheaves admit costandard filtrations \cite{AB, BRR, T26}.
As such, to verify fully faithfulness, by Hartogs' lemma it is enough to work on an open subscheme of $\sfT$ whose complement is codimension two.

There are two different types of hyperplane in $\sfT$: the base change to characteristic $p$ and the root hyperplanes. By Hartogs' lemma we may throw away the intersections of these hyperplanes, after which we reduce to working locally around each hyperplane separately.

Locally around a prime $p$, we reduce to the case of $\sfG$ being semi-simple rank 0 over a field of characteristic $p$. Then all filtrations split and it is easy to see that the functor is fully faithful.

Locally around a root hyperplane, we reduce to the case of $\sfG$ being semi-simple rank 1 over a field of characteristic 0. In this rank 1 case, we then give a direct analysis of the relevant spaces of homomorphisms. The main content is a certain order of vanishing calculation, involving the associated graded of the Wakimoto filtration on central sheaves and its coherent analogue. On the coherent side, we perform the order of vanishing calculation in Section \ref{Rank1Spectral}. On the constructible side, we perform the order of vanishing calculation in Section \ref{s:loccentr} using the weight filtration.

These calculations seem to be new, are particular to the affine case, and in our view are interesting in their own right. With that said, we also recall that a similar idea, namely controlling homomorphisms between tilting objects in monodromic Hecke categories via the associated graded, appears in earlier work of Bezrukavnikov--Riche \cite{BR}.

\subsubsection*{The bi-Whittaker equivalence} Finally, let us briefly comment on the argument for Theorem \ref{t:main2}, i.e., the equivalence
$$_{\chi}\cH_{\chi} \simeq \IndPerf(\sfG / \sfG).$$
We construct a monoidal functor from $\IndPerf(\sfG / \sfG)$ to $_{\chi}\cH_{\chi}$ as follows. Gaitsgory's central sheaves already provide a monoidal functor from $\IndPerf(\sfG/\sfG)$ to the center of $\cH$, and in particular to the equivariant endomorphisms of any $\cH$-module. Specializing to the Whittaker module ${}_\chi \cH$, we obtain the desired monoidal functor \[\IndPerf(\sfG/\sfG) \rightarrow \Hom_{\cH}({}_\chi \cH, {}_\chi \cH) \simeq {}_{\chi}\cH_{\chi}.\]
That this is an equivalence of categories then follows from Theorem \ref{t:main1} essentially by faithfully flat descent along $\wG^{\on{aff}} \rightarrow \sfG$, where $\wG^{\on{aff}}$ denotes the affinization of $\wG$.

\begin{remark}
If the derived subgroup of $\sfG$ is not simply connected, then we have the following minor complications.

First, the description of the affinization of the Grothendieck--Springer scheme $\wG^{\on{aff}}$ needs to be modified as follows. Namely, if the derived subgroup is simply connected, then the natural map $\wG^{\on{aff}} \rightarrow \sfG \times_{\sfG/\!\!/ \sfG} \sfT$ is an isomorphism. In general, this is not true, and instead one has the following. If we choose a finite isogeny $1 \rightarrow \sfZ \rightarrow  \sfG^{\on{sc}} \rightarrow \sfG \rightarrow 1,$ wherein $\sfG^{\on{sc}}$ has simply connected derived subgroup, then $$\wG^{\on{aff}} \simeq \sfG \underset{(\sfG^{\on{sc}}/\!\!/ \sfG^{\on{sc}}) / \sfZ} \times \sfT.$$
Relatedly the finite bi-Whittaker category is equivalent to quasi-coherent sheaves on the stack $(\sfG^{\on{sc}}/\!\!/ \sfG^{\on{sc}}) / \sfZ$.

Second, as observed in \cite{FS}, if the derived subgroup is not simply connected then the category of perfect complexes on $\sfG/\sfG$ is no longer generated under finite colimits by free modules pulled back from representations of $\sfG$.
\end{remark}

\subsection{Organization of the paper}\phantom{hi}
Let us briefly indicate the contents of the subsequent sections.

\begin{itemize}
\item[-] In Section \ref{Notation}, we collect some mostly standard notation.
 \end{itemize}
 
\noindent Part \ref{Part1} concerns the construction of the Arkhipov--Bezrukavnikov functor and its contents are as follows.

\begin{itemize}
\item[-] In Section \ref{UniversalSheaves} we gather some basic results about the universal monodromic affine Hecke category.

\item[-] In Section \ref{s:Wakimoto}, we discuss the universal monodromic Wakimoto sheaves.

\item[-] In Section \ref{CentralSheaves}, we discuss the universal monodromic version of Gaitsgory's construction of central sheaves via nearby cycles.

\item[-] In Section \ref{s:Plucker}, we produce the highest weight arrows between the central sheaves and Wakimoto sheaves. Then we check that they satisfy the Pl\"ucker relations.

\item[-] In Section \ref{Fiberfunctor}, we give a Tannakian description of the associated graded of the Wakimoto filtration on central sheaves. This involves   identifying a certain $\sfG$-torsor on $\sfT/\sfT$.

\item[-] In Section \ref{ConstructionFunctorI}, we perform the first step in the construction of the Arkhipov--Bezrukavnikov functor. Namely we construct a functor out of the additive category of free equivariant modules over a certain multi-graded ring $\cO(\sfX)$.

\item[-] In Section \ref{SectionMulti}, we collect some preliminaries on multi-graded rings. We prove a multi-graded version of Serre vanishing and we describe the category of perfect complexes on the relevant open subscheme as a quotient. Moreover we show that passing to a submonoid that contains all sufficiently dominant coweights does not change the relevant open subscheme.

\item[-] In Section \ref{ConstructionFunctorII}, we finish the construction of the Arkhipov--Bezrukavnikov functor. We use that all sufficiently dominant weight spaces of $\cO(\sfX)$ admit dual Weyl filtrations and that restricting to a submonoid containing all sufficiently dominant weight does not change the relevant open subscheme. 
\end{itemize}

\noindent Part \ref{Part2} concerns the proof of the Arkhipov--Bezrukavnikov equivalence and its contents are as follows.

\begin{itemize}
\item[-] In Section \ref{s:specside}, we study the global sections of certain vector bundles on the Grothendieck--Springer scheme, and we perform the order of vanishing calculation on the coherent side.

\item[-] In Section \ref{WhittakerSheaves}, we define the category of universal monodromic Whittaker sheaves.

\item[-] In Section \ref{OrderVanishingSection}, we perform an order of vanishing calculation on the constructible side using the weight filtration.

\item[-] In Section \ref{s:loccentr}, we show that after localizing away from all but one wall the central functor factors through restriction to a subgroup that has semi-simple rank 1.

\item[-] In Section \ref{s:whiteq}, we use Hartogs' lemma to prove our first main Theorem \ref{t:main1}, the Whittaker equivalence.
\end{itemize}
 
\noindent Part \ref{Part3} concerns the proof of the bi-Whittaker equivalence and its contents are as follows.

\begin{itemize}
\item[-] In Section \ref{BiWhitPrelim} collect some preliminaries about the finite bi-Whittaker category, centrality and the generation of perfect complexes by induced bundles.
\item[-] In Section \ref{BiWhitFunctorSection} we use centrality to construct the bi-Whittaker functor. The main difficulty concerns groups with disconnected center, for which we study the fundamental group of Gaitsgory's family.
\item[-] In Section \ref{BiWhitEquivSection} we prove our second main Theorem \ref{t:main2}, the bi-Whittaker equivalence.
\end{itemize}

\subsection{Acknowledgments} It is a pleasure to thank Pramod Achar, Johannes Anschütz, Dima Arinkin,  David Ben-Zvi, Roman Bezrukavnikov, Justin Campbell, Harrison Chen, Joakim F\ae rgeman, Dennis Gaitsgory, Dmitry Kubrak, Yau Wing Li, João Lourenço, Mark Macerato, Sam Raskin, Simon Riche, Jack Sempliner, Zhiyou Wu, David Yang, Ruotao Yang, Jize Yu, and Xinwen Zhu for helpful discussions. We thank David Nadler and Zhiwei Yun for illuminating conversations regarding nearby cycles and the Kirillov model respectively.

We are especially grateful to Xinyu Li and Jiahao Niu for many discussions about this work and for some corrections to our draft. We also thank them for sharing with us a different approach to the Arkhipov--Bezrukavnikov equivalence in positive characteristic that they developed independently.

ChatGPT was also used for checking and proofreading and proposed the arguments in Section \ref{GenerationInduced}. Apart for the proof of Proposition \ref{IndPerfGsc}, AI did not contribute to any other arguments in this paper.

Part of this research was performed while the author was visiting SLMath, which is supported by the NSF grant DMS-2424139. G.D. was supported by an NSF Postdoctoral Fellowship under grant No. 2103387. J.T. was partially supported by NSF grant DMS-1646385.

\section{Notation}\label{Notation}

\subsection{Reductive groups}
Let $G$ be a pinned complex reductive group. Let $U \subset B$ be the unipotent radical of the Borel. Write $\Lambda$ for the coweight lattice of the maximal torus $T \subset G$. Let $\Lambda^+$ be the dominant coweights, and $\Lambda^{++}$ be the strictly dominant coweights. Let $W^{\fnt}$ denote the finite Weyl group.

Let $\sfG$ be the Langlands dual group defined over the integers. Let $\sfU \subset \sfB$ the unipotent radical of the \textit{negative} Borel.\footnote{We choose the negative Borel so that the line bundle $\sfG \times_{\sfB} \bZ_{\lambda} \simeq \cO(\lambda)$ is ample when $\lambda \in \Lambda^{++}$ is strictly dominant.} Let $\widetilde{\sfG} \coloneqq \sfG \times^{\sfB} \sfB$ be the Grothendieck--Springer scheme. Let $(\sfG/\sfU)^{\aff} \coloneqq \Spec(\cO(\sfG/\sfU))$ be the affine closure of base affine space. Write $R \coloneqq \bZ[\Lambda]$ for the ring of functions on the maximal torus $\sfT \subset \sfG$.

Let $Z \subset G$ be the center of the derived subgroup of $G$. This is a finite central subgroup and the quotient $G^{\ad} \coloneqq G/Z$ is the product of a torus and an adjoint reductive group. Let $\sfG^{\sct}$ be the Langlands dual group to $G^{\ad}$. Then $\sfG^{\sct}$ is the product of a torus and a simply connected reductive group, and $\sfG = \sfG^{\sct}/\sfZ$ where $\sfZ \coloneqq \Hom(Z, \bG_m)$ is a finite central subgroup of $\sfG^{\sct}$.

If the derived subgroup of $\sfG$ is not simply connected, then the na\"ive characteristic polynomial map should be replaced by \[\sfG = \sfG^{\sct}/\sfZ \rightarrow \sfC \coloneqq (\sfG^{\sct}/\!\!/\sfG^{\sct})/\sfZ.\]

\subsection{Loop groups}
Let $\mathring{I}$ be the pro-unipotent radical of the Iwahori subgroup $I \subset G[\![t]\!]$. The enhanced affine flag variety $\Fl \coloneqq G(\!(t)\!)/\mathring{I}$ is a $G/U$-torsor over the affine Grassmannian $\Gr \coloneqq G(\!(t)\!)/G[\![t]\!]$.\footnote{Beware that the enhanced affine flag variety is usually denoted with a tilde in the literature.}

Let $W \coloneqq W^{\fnt} \ltimes \Lambda$ be the extended affine Weyl group. Let $\ell(w)$ denote the length of an element $w \in W$. Let $\sfR$ be the set of affine coroots and $\sfR^+$ be the set of positive affine coroots.

\subsection{Weakly constructible sheaves}
All derived categories are interpreted in the DG categorical sense. If $C$ is a DG-category, we write $C^{\cpt}$ for the full subcategory of compact objects.

Suppose $H$ is a complex algebraic group. Let $Y$ be a complex $H$-variety.

Let $\Shv(Y)$ be the category of all sheaves (in the sense of \cite{L09}) on $Y$ equipped with its analytic topology valued in the unbounded derived category of $\bZ$-modules.

Let $\Shv_{(H)}(Y)$ be the full subcategory of weakly $H$-constructible sheaves. Weakly $H$-constructible means locally constant along the $H$-orbits. We impose no finiteness conditions on stalks and allow unbounded complexes.

Now let $Y = \colim Y_i$ be an ind-complex $H$-variety.

Let $\Shv_{(H)}(Y) \coloneqq \lim \Shv_{(H)}(Y_i)$ be the limit under $!$-pullback.

Let $\Perv_{(H)}(Y)$ be the abelian subcategory of perverse sheaves. We allow perverse sheaves to have infinite dimensional stalks, but require that they are compact in $\Shv_{(H)}(Y)$.

Let $\cH \coloneqq \Shv_{(I)}(\Fl)$ be the universal monodromic affine Hecke category, equipped with the convolution monoidal structure.

By realizing $\cH$ as the ind-completion of the bounded homotopy category of tilting sheaves, it acquires the structure of a monoidal category in the infinity categorical sense.

\subsection{Coherent sheaves}
Suppose $\sfH$ is an algebraic group acting on a scheme $\sfY$, both defined over the integers. We use the following notation for some associated categories of sheaves.

\begin{itemize}
\item[-] Let $\Rep(\sfH)$ be the additive category of finite dimensional $\sfH$-modules that are free over the integers.

\item[-] Let $\QCoh_{\sfH}(\sfY)$ be the unbounded derived category of $\sfH$-equivariant quasi-coherent sheaves on $\sfY$ as defined in \cite{Z25}.

\item[-] Let $\Free_{\sfH}(\sfY)$ be the additive (underived) category of $\sfH$-equivariant coherent sheaves that are pulled back from $\Rep(\sfH)$.

\item[-] If $\sfH$ is reductive, let $\Free_{\sfH}^{\tilt}(\sfY)$ be the additive (underived) category of $\sfH$-equivariant coherent sheaves that are pulled back from finite dimensional tilting representations of $\sfH$.

\item[-] Let $\Perf_{\sfH}^{\ind}(\sfY) \subset \QCoh_{\sfH}(\sfY)$ be the full subcategory generated under retracts and finite colimits by $\Free_{\sfH}(\sfY)$.

\item[-] Let $\Perf_{\sfH}(\sfY)$ be the derived category of perfect complexes.

\item[-] Let $\IndPerf_{\sfH}(\sfY)$ be the ind-completion of $\Perf_{\sfH}(\sfY)$.
\end{itemize}

{\large \part{Construction of the functor}\label{Part1}}
In this first part of the paper we construct the monoidal Arkhipov--Bezrukavnikov functor \[F: \IndPerf_{\sfG}(\widetilde{\sfG}) \rightarrow \cH\] in our universal monodromic integral coefficient setting.
This result appears in Section \ref{ConstructionFunctorI}, and we will now indicate the steps used in the construction.

After collecting some preliminaries about the affine Hecke category in Sections \ref{UniversalSheaves} and \ref{s:Wakimoto}, the starting point is a universal monodromic version of Gaitsgory's central functor \[Z: \Rep(\sfG) \rightarrow \Perv_{(I)}^{\waki}(\Fl).\] This functor sends representations of $\sfG$ to Wakimoto filtered perverse sheaves. It is constructed in Section \ref{CentralSheaves} using nearby cycles from the affine Grassmannian to the affine flag variety, exactly as in the original work \cite{GZ}.

In Section \ref{s:Plucker} we provide the extra data needed to extend the central functor to $\Rep(\sfB)$, that is the highest weight arrows from central to Wakimoto sheaves. Checking the relevant Pl\"ucker relations between these highest weight arrows is more delicate in the universal monodromic setting, because a certain trivialization of the relevant stratum in Gaitsgory's family does not extend globally.

Using monodromy of nearby cycles and highest weight arrows, we upgrade the central functor to a monoidal functor \[F: \Free_{\sfG \times \sfT}(\sfG \times (\sfG/\sfU)^{\aff} \times \sfT) \rightarrow \Perv_{(I)}^{\waki}(\Fl)\] such that
\begin{enumerate}
\item[($\star$)] \label{grFRestrict} after taking associated graded for the Wakimoto filtration, $\gr F$ is isomorphic to restriction along $\sfT \rightarrow \sfG \times (\sfG/\sfU)^{\aff} \times \sfT$ sending $t \mapsto (t, 1, t)$.
\end{enumerate}
The proof of $(\star)$ requires additional arguments not necessary in the setting of \cite{AB}, and we defer a discussion of this to the final paragraph of this section.

Assuming $(\star)$, to get a functor out of $\Perf_{\sfG}(\widetilde{\sfG})$ there are two remaining steps, which are performed in Sections \ref{ConstructionFunctorI} and \ref{ConstructionFunctorII}.
\begin{enumerate}[label=(\roman*)]
\item \label{Factor1}First we factor $F$ through a closed subscheme $\sfX \subset \sfG \times (\sfG/\sfU)^{\aff} \times \sfT$ cut out by equations coming from the compatibility between monodromy of nearby cycles and highest weight arrows.
\item \label{Factor2} Second we factor $F$ through the relevant open subscheme $\sfV \subset \sfX$, a $\sfT$-torsor over $\widetilde{\sfG}$. Here we use ($\star$) to show that, after passing to homotopy categories, $F$ kills complexes supported on the irrelevant boundary.
\end{enumerate}
The arguments for \ref{Factor1} and \ref{Factor2} are similar to ones that appeared in \cite{AB, AR}.

However working over the integers, we face the following additional obstacle. The ring of functions on $\sfX$ probably does not admit a dual Weyl filtration. For this reason, there may be higher Exts between free coherent sheaves on $\sfX$ that are pulled back from tilting representations. Therefore we can no longer simply realize $\Perf_{\sfG \times \sfT}(\sfX)$ as the bounded homotopy category of such vector bundles.

Instead we will roughly argue that all \textit{sufficiently dominant} weight spaces of $\cO(\sfX)$ admit dual Weyl filtrations. We then restrict to a subring obtained by discarding all weight spaces that do not admit dual Weyl filtrations. By a general result about multi-graded rings proved in Section \ref{SectionMulti}, this does not change the relevant open subscheme. Therefore, by first restricting to a subring of $\cO(\sfX)$ then passing to bounded homotopy categories and proving that complexes supported on the irrelevant boundary are killed, we obtain the desired functor out of $\Perf_{\sfG}(\widetilde{\sfG})$.

Finally, let us comment on the assertion $(\star)$, whose proof appears in Sections \ref{Fiberfunctor} and \ref{ConstructionFunctorI}. The proof crucially uses that $\gr Z: \Rep(\sfG) \rightarrow \Free_{\sfT}(\sfT)$ is \textit{monoidally} isomorphic to restriction to the maximal torus.\footnote{It is possible to construct such an isomorphism of plain functors using Braden's theorem, but it does not appear straightforward to prove directly that this isomorphism intertwines the monoidal structures.}
Following Arkhipov--Bezrukavnikov, $\gr Z$ is a tensor functor and therefore it identifies under Tannakian formalism with pullback along some map $\sfT/\sfT \rightarrow \pt/\sfG$.
Over the integers, it is not immediate how to identify this $\sfG$-torsor on $\sfT/\sfT$, unlike in the setting of \cite{AB}. In Section \ref{IDTorsor}, we use the Pl\"ucker relations to construct a $\sfB$-reduction of the $\sfG$-torsor. We then observe there is a unique isomorphism class of $\sfB$-torsor on $\sfT/\sfT$ whose associated graded $\sfT$-torsor is isomorphic to the pullback of the universal $\sfT$-torsor from $\pt / \sfT$, thereby allowing us to identify the $\sfG$-torsor.

\section{Universal monodromic sheaves}\label{UniversalSheaves}
Here we collect some basic results about universal monodromic sheaves on the enhanced affine flag variety. We recall the usual convolution formulas, and introduce the notion of universal perversity which combines perversity with flatness.

\subsection{Standard and costandard extensions}
If $w \in W \coloneqq W^{\fnt} \ltimes \Lambda$, let $j_w:\Fl_w \coloneqq I\dot{w} \hookrightarrow \Fl$ be the Iwahori orbit through a lift $\dot{w} \in G(\!(t)\!)$.
Let $R_{\Fl_{w}}$ be the perverse universal local system on $\Fl_w \simeq T \times \bC^{\ell(w)}$, i.e., the regular representation of the fundamental group. Define the universal standard and costandard sheaves by \[\Delta_w \coloneqq j_{w!} R_{\Fl_w}, \quad \nabla_w \coloneqq j_{w*} R_{\Fl_w} \; \in \; \Perv_{(I)}(\Fl).\]
Without having chosen a base point, these sheaves are defined uniquely, but only up to non-canonical isomorphism. After choosing a lift $\dot{w} \in G(\!(t)\!)$, they become unique up to canonical isomorphism, and we denote these canonical lifts by $R_{\Fl_{\dot{w}}}, \Delta_{\dot{w}}, \nabla_{\dot{w}}$.

We prove in Lemma \ref{UPervW} below that these universal standard and costandard sheaves are indeed perverse.
By adjunction \beq \label{StandardtoCostandard}\RHom(\Delta_{\dot{w}}, \nabla_{\dot{w}}) \simeq R \quad \text{and} \quad \RHom(\Delta_w, \nabla_v) \simeq 0 \quad \text{ for } \quad w \neq v.\eeq

Let $s$ be a simple reflection and $\alpha$ the corresponding simple coroot. As explained in \cite{T}, there is a short exact sequence of perverse sheaves \beq \label{DeltasNablas} 0 \rightarrow \Delta_s \rightarrow \nabla_s \rightarrow \nabla_1/(e^{\alpha} - 1) \rightarrow 0.\eeq

If $A, B \in \cH$ define the convolution \[A * B \coloneqq m_! (A \widetilde{\boxtimes} B)[\dim T] \simeq m_* (A \widetilde{\boxtimes} B) \quad \text{where} \quad m: \Fl \widetilde{\times} \Fl \coloneqq G(\!(t)\!) \times^{\mathring{I}} \Fl \rightarrow \Fl.\]

\begin{proposition}\label{Convolution}
Let $v, w \in W$.
\begin{enumerate}[label=(\alph*)]
\item \label{Convolution1} There are canonical isomorphisms $\Delta_{\dot{v}} * \nabla_{\dot{v}^{-1}} \simeq \Delta_1$.

\item \label{Convolution2} If $\ell(vw) = \ell(v) + \ell(w)$, there are canonical isomorphisms $\Delta_{\dot{v}} * \Delta_{\dot{w}} \simeq \Delta_{\dot{v}\dot{w}}$ and $\nabla_{\dot{v}} * \nabla_{\dot{w}} \simeq \nabla_{\dot{v}\dot{w}}$.
\end{enumerate}
\end{proposition}
\begin{proof}
Part \ref{Convolution1} is similar to \cite{T}. Part \ref{Convolution2} follows by commutativity of
\[\begin{tikzcd} \Fl_v \widetilde{\times} \Fl_w \arrow[r] \arrow[d] & \Fl_{vw} \arrow[d] \\
T \times T \arrow[r] & T \end{tikzcd} \qquad \begin{tikzcd} (\dot{v}, \dot{w}) \arrow[r, mapsto] \arrow[d, mapsto] & \dot{v}\dot{w} \arrow[d, mapsto] \\
(1, 1) \arrow[r, mapsto] & 1. \end{tikzcd}\]
The downward maps were induced by the choices of lifts $\dot{v}, \dot{w} \in G(\!(t)\!)$.
\end{proof}

\begin{remark}\label{CanonicalWhittaker} Using the Whittaker module, to be defined in Section \ref{WhittakerSheaves}, it is possible to canonically define those standard and costandard sheaves that are indexed by the finite Weyl group \cite{LY}. Namely if $v \in W^{\fnt}$, then $\Delta_v$ and $\nabla_v$ can be canonically defined by the condition that there are canonical isomorphisms ${}^{\chi}\nabla_1 * \Delta_v \simeq {}^{\chi}\nabla_1 * \nabla_v \simeq {}^{\chi}\nabla_1$ in ${}_{\chi}\cH^{\fnt}$. Thereby, for $v, w \in W^{\fnt}$, the isomorphisms in Proposition \ref{Convolution} become canonical independently of the chosen lifts.
\end{remark}

\subsection{Universal perversity}\label{UniversalPerversity}
First we define some (not stable) subcategories of $\cH$.
\begin{enumerate}
\item[-] Let $\langle \Delta[\geq 0] \rangle$ be generated under extensions by  $\Delta_v[i]$ for $v \in W$ and $i \geq 0$.
\item[-] Let $\langle \nabla[\geq 0] \rangle$ be generated under extensions by  $\nabla_v[i]$ for $v \in W$ and $i \geq 0$.
\item[-] Let $\langle \Delta[\leq 0] \rangle$ be generated under extensions by  $\Delta_v[i]$ for $v \in W$ and $i \leq 0$.
\item[-] Let $\langle \nabla[\leq 0] \rangle$ be generated under extensions by  $\nabla_v[i]$ for $v \in W$ and $i \leq 0$.
\end{enumerate}

\begin{lemma}\label{RadonExactness}
If $w \in W$ then
\begin{enumerate}[label=(\alph*)]
\item convolution by $\nabla_w$ preserves $\langle \Delta[\geq 0] \rangle$ and $\langle \nabla[\geq 0] \rangle$,
\item convolution by $\Delta_w$ preserves $\langle \nabla[\leq 0] \rangle$ and $\langle \Delta[\leq 0] \rangle$.
\end{enumerate}
\end{lemma}
\begin{proof}
Let $s, v \in W$ such that $s$ is a simple reflection.
\begin{enumerate}
\item[-] If $vs < v$ then $\Delta_v * \nabla_s \simeq \Delta_{vs}$.
\item[-] If $v < vs$ then by \eqref{DeltasNablas} there is a triangle $\Delta_{vs} \rightarrow \Delta_v * \nabla_s  \rightarrow \Delta_v/(e^{\alpha} - 1).$
\end{enumerate}
In both cases $\Delta_v * \nabla_s \in \langle \Delta[\geq 0] \rangle$. Therefore induction on $\ell(w)$ shows that $\langle \Delta[\geq 0] \rangle$ is preserved by all $-* \nabla_w$. The other claims are similar.
\end{proof}

\begin{definition}
We call $A \in \cH$ universally perverse if it satisfies the following equivalent conditions.  
\end{definition}

\begin{proposition}\label{UPerv}
For $A \in \cH$ the following are equivalent.
\begin{enumerate}[label=(\alph*)]
\item \label{UPerv1} $M \otimes_R A \in \Perv_{(I)}(\Fl)$ is perverse, for any finitely generated $R$-module $M$.
\item $A \in \langle \Delta[\geq 0] \rangle \cap \langle \nabla[\leq 0] \rangle$.
\end{enumerate}
\end{proposition}
\begin{proof}
Suppose that $A$ satisfies \ref{UPerv1}. Since $A$ is perverse, it belongs to $\langle \Delta[\geq 0] \rangle$. Moreover $j_w^! A$ has Tor-amplitude $\geq -\ell(w)$ as a complex of $R$-modules. Therefore $j_w^! A$ can be represented by a finite complex of finitely generated flat $R$-modules in degrees $\geq -\ell(w)$ by \cite[\href{https://stacks.math.columbia.edu/tag/0651}{0651}]{Stacks}. The Laurent polynomial Quillen--Suslin theorem \cite{Sw} says that every finitely generated flat $R$-module is free.\footnote{The Laurent polynomial Quillen--Suslin theorem is not really essential for this paper, but it mildly simplifies the notation by allowing us to avoid tensoring with finitely generated flat $R$-modules.} Therefore $A \in \langle \nabla[\leq 0] \rangle$ by the Cousin filtration.
\end{proof}

\begin{lemma}\label{UPervW}
For all $v, w \in W$, the sheaf $\Delta_w * \nabla_v$ is universally perverse.
\end{lemma}

\begin{proof}
Lemma \ref{RadonExactness} implies that $\Delta_w * \nabla_v$ is contained in both $\langle \Delta[\geq 0] \rangle * \nabla_v \subset \langle \Delta[\geq 0] \rangle$ and $\Delta_w * \langle \nabla[\leq 0] \rangle \subset \langle \nabla[\leq 0] \rangle$.
\end{proof}

\subsection{Loop rotation}
Let $\bC^{\times}$ act on $\Fl$ by loop rotation $z \cdot t^{\lambda} \coloneqq \lambda(z) t^{\lambda}$. Every sheaf in $\cH$ is locally constant along the loop rotation orbits. Monodromy makes $\cH$ an $(R \otimes R)[m^{\pm 1}]$-linear category, where $m$ acts by monodromy around the loop rotation orbits.
The following is similar to Lemmas 6.6 and 6.7 of \cite{BR22}.

\begin{lemma}\label{Hom0}
Let $v = xe^{\lambda}$ and $w = y e^{\mu}$ in $W$, where $x, y \in W^{\fnt}$ and $\lambda, \mu \in \Lambda$.
\begin{enumerate}[label=(\alph*)]
\item \label{Hom01} The monodromy actions on $\Delta_{v}$ and $\nabla_{v}$ factor through \[(R \otimes R)[m^{\pm 1}]/(x(r) \otimes 1 - 1 \otimes r,\; 1 \otimes e^{\lambda} - m).\]
\item \label{Hom02} If $v \neq w$ then, \[\Hom(\Delta_{v}, \Delta_{w}) \simeq \Hom(\Delta_{v}, \nabla_{w}) \simeq \Hom(\nabla_{v}, \Delta_{w}) \simeq \Hom(\nabla_{v}, \nabla_{w}) \simeq 0.\]
\end{enumerate}
\end{lemma}
\begin{proof}
For part \ref{Hom01}, the left and right $T$-actions on $\mathring{I} \backslash \Fl_{xe^{\lambda}}$ differ by $x$. Moreover loop rotation acts on $\mathring{I} \backslash \Fl_{xe^{\lambda}}$ by right multiplication via $\lambda: \bC^{\times} \rightarrow T$.

For part \ref{Hom02}, let $\alpha \in \Hom(\Delta_v, \Delta_w)$. Observe that for $r \neq 0 \in R$, the monodromy action of $1 \otimes r$ on the image $\image \alpha$ is injective.  Indeed $\image \alpha$ is a perverse subsheaf of $\Delta_w$, and $1 \otimes r$ acts injectively on $\Delta_w$ by Proposition \ref{UPervW}.

Since loop rotation monodromy acts by $1 \otimes e^{\lambda}$ on $\Delta_v$ and by $1 \otimes e^{\mu}$ on $\Delta_w$, it follows that $1 \otimes (e^{\lambda} - e^{\mu})$ acts by zero on $\image \alpha$. The above observation therefore implies $\lambda = \mu$. If $s \in R$ then $1 \otimes (x(s) - y(s))$ acts by zero on $\image \alpha$. The above observation implies $x(s) - y(s) = 0$ for all $s \in R$, hence $x = y$.
\end{proof}

\section{Wakimoto sheaves}\label{s:Wakimoto}
Here we define universal monodromic Wakimoto sheaves $W_{\lambda}$ on the enhanced affine flag variety. They correspond to vector bundles on the diagonal of the Steinberg stack, that are pulled back from characters of $\sfT$. On the lattice subgroup of the affine Weyl group, the length function is only additive on the dominant cone. For this reason, the definition of $W_{\lambda}$ involves expressing $\lambda$ as a dominant weight plus an anti-dominant weight.

\subsection{Canonical lifts of the lattice elements}
The moduli of $\sfB$-local systems on the cylinder is only identified with $\sfB/\sfB$ after choosing a coordinate. Accordingly, to canonically define Wakimoto sheaves we needed to choose a coordinate on the disc.

The choice of coordinate yields canonical lifts $\dot{\lambda} \coloneqq t^{\lambda} \in G(\!(t)\!)$ of the lattice elements of the affine Weyl group. These lifts are multiplicative, in the sense that $t^{\lambda} t^{\mu} = t^{\lambda + \mu}$. For $\lambda, \mu \in \Lambda^+$ dominant, Proposition \ref{Convolution} gives canonical isomorphisms \beq \label{LambdaConvolve} \Delta_{\dot{\lambda}}*\Delta_{\dot{\mu}} \simeq \Delta_{\dot{\lambda} + \dot{\mu}} \quad \text{and} \quad \nabla_{\dot{\lambda}}*\nabla_{\dot{\mu}} \simeq \nabla_{\dot{\lambda} + \dot{\mu}},\eeq
satisfying the associativity identity.

\subsection{Definition of Wakimoto sheaves}
The above equation \eqref{LambdaConvolve} gives a canonical monoidal functor \[W: \Rep(\sfT) \rightarrow \cH, \qquad \bZ_{\lambda} \mapsto W_{\lambda}.\] Uniquely defined by $W_{\lambda} \simeq \nabla_{\dot{\lambda}}$ for $\lambda \in \Lambda^+$.
Writing $\lambda = \mu - \eta$ for $\mu, \eta \in \Lambda^+$, there is a canonical isomorphism $W_{\lambda} \simeq \nabla_{\dot{\mu}} * \Delta_{-\dot{\eta}}$.
Lemma \ref{UPervW} shows that Wakimoto sheaves are universally perverse.

\begin{proposition}\label{WProperties}
Wakimoto sheaves satisfy the following properties.
\begin{enumerate}[label=(\alph*)]
\item \label{LoopW} The monodromy action on $W_{\lambda}$ factors through \[(R \otimes R)[m^{\pm 1}]/(r \otimes 1 - 1 \otimes r,\; 1 \otimes e^{\lambda} - m).\]
\item \label{HomW} If $\mu \neq \lambda$ then $\Hom(W_{\lambda}, W_{\mu}) \simeq 0$.
\item \label{RHomW} If $\mu \nleq \lambda$ then $\RHom(W_{\lambda}, W_{\mu}) \simeq 0$.
\end{enumerate}
\end{proposition}
\begin{proof}
The proof is similar to that of Lemma \ref{Hom0}.
\end{proof}

\subsection{Wakimoto filtered perverse sheaves}
Let $\Perv_{(I)}^{\waki}(\Fl)$ be the full additive subcategory of perverse sheaves that admit a filtration with graded pieces $W_{\lambda}$. Such a filtration is necessarily unique. Note that $\Perv_{(I)}^{\waki}(\Fl)$ is closed under convolution, unlike the larger category $\Perv_{(I)}(\Fl)$.
The following is similar to Corollary 6.3 of \cite{BR}.

\begin{proposition}\label{grFaithful}
Associated graded gives a faithful monoidal functor \[\gr: \Perv_{(I)}^{\waki}(\Fl) \rightarrow \Free_{\sfT}(\sfT).\]
\end{proposition}
\begin{proof}
Proposition \ref{WProperties} says $\Hom(W_{\lambda}, W_{\mu}) \simeq 0$ for $\mu \not\leq \lambda$. Therefore any map of Wakimoto filtered perverse sheaves preserves the Wakimoto filtration,  hence $\gr$ is functorial. Moreover Proposition \ref{WProperties} says $\Hom(W_{\lambda}, W_{\mu}) \simeq 0$ for $\mu < \lambda$, hence $\gr$ is faithful. Moreover $\gr$ is monoidal by Lemma 16 of \cite{AB}.
\end{proof}

\section{Central sheaves}\label{CentralSheaves}
Gaitsgory constructed a deformation from the affine Grassmannian to the affine flag variety, and Gaitsgory's functor is defined by nearby cycles for this family.
The resulting central sheaves correspond to vector bundles on the diagonal of the Steinberg stack, that are pulled back from representations of $\sfG$.

Closely following \cite{GZ}, we verify that universal monodromic central sheaves are convolution exact and Gaitsgory's functor is central. Moreover we check that loop rotation monodromy induces a tensor automorphism, providing the extra data needed to lift Gaitsgory's functor to $\sfG/\sfG$. Finally we prove that central sheaves admit universal Wakimoto filtrations following \cite{AB}.

\subsection{Gaitsgory's central functor} Gaitsgory's central functor is defined by nearby cycles from the affine Grassmannian to the affine flag variety.

Denote the geometric Satake funtor \cite{MV} by \[S: \Rep(\sfG) \rightarrow \Perv_{G[\![t]\!]}(\Gr).\]

Let $\Fl_{\bC}$ be the moduli of: a $G$-bundle $E$ on $\bC$, a $U$-reduction of $E|_0$, a point $z \in \bC$, and a trivialization of $E|_{\bC - z}$. The choice of coordinate gives a trivialization $\Fl_{\bC^{\times}} \simeq \Gr \times G/U \times \bC^{\times}$ away from $z = 0$.

Gaitsgory's functor is nearby cycles in this family \[Z:\Rep(\sfG) \rightarrow \Shv(\Fl), \qquad A \mapsto \psi (S(A) \boxtimes R \boxtimes \bZ_{\bC^{\times}}).\] Here $R \in \Shv_{(B)}(G/U)$ is the monoidal unit. Let $\mu_A: Z(A) \rightarrow Z(A)$ denote monodromy of nearby cycles.

\subsection{Properties of the central functor}
In the proofs that follow, we will need several variations on Gaitsgory's family \cite{GZ}, that also appear in Section 3.5 of \cite{B}. The following table summarizes their fibers.

\begin{center}
\begin{tabular}{ l l l }
Moduli space & Nonzero fiber & Zero fiber \\
$\Fl_{\bC}$ & $\Gr \times G/U$ & $\Fl$ \\
$\Fl'_{\bC}$ & $\Gr \times \Fl$ & $\Fl$ \\
$(\Fl \widetilde{\times} \Fl)_{\bC}$ & $\Gr \widetilde{\times} \Gr \times G/U \times G/U$ & $\Fl \widetilde{\times} \Fl$ \\
$\Fl_{\bC} \widetilde{\times} \Fl$ & $\Gr \times G/U \times \Fl $ & $\Fl \widetilde{\times} \Fl$ \\
\end{tabular}
\end{center}

\begin{proposition}\label{ZProperties}
Gaitsgory's central functor satisfies the following properties.
\begin{enumerate}[label=(\alph*)]
\item \label{ZI} Central sheaves are weakly $I$-constructible.
\item \label{ZMonoidal} $Z$ admits a monoidal structure.

\item \label{ZMon} Monodromy of nearby cycles is a tensor automorphism of $Z$.

\item \label{ZLoopMon} Loop rotation monodromy equals nearby cycles monodromy.

\item \label{ZPerv} Central sheaves are universally perverse.
\end{enumerate}
\end{proposition}
\begin{proof}
Part \ref{ZI} is similar to Section 2.3 of \cite{GZ}. Let $I_{\bC} \coloneqq G[t] \times_G B$, where $G[t] \rightarrow G$ is evaluation at $0$. Gaitsgory's functor takes values in weakly $I$-constructible sheaves because
\begin{enumerate}
\item[-] $S(A) \boxtimes R \boxtimes \bZ_{\bC^{\times}}$ is weakly constructible for $I_{\bC}$ acting on $\Fl_{\bC^{\times}}$ by changing the trivialization of $E|_{\bC - z}$,
\item[-] and $Z(A)$ is supported on a subvariety of $\Fl$ on which $I$ acts through a finite dimensional quotient, that is surjected onto by $I_{\bC}$.
\end{enumerate}

The proof of monoidality will use a family that deforms convolution for affine flags to convolution for the affine Grassmannian.

Let $(\Fl \widetilde{\times} \Fl)_{\bC}$ be the moduli of: $G$-bundles $E, E'$ on $\bC$, $U$ reductions of $E|_0 $ and $E'|_0$, a point $z \in \bC$, a trivialization of $E'|_{\bC - z}$, and an isomorphism $E|_{\bC - z} \simeq E'|_{\bC - z}$. The convolution map $m: (\Fl \widetilde{\times} \Fl)_{\bC} \rightarrow \Fl_{\bC}$ forgets $E'$.

Parts \ref{ZMonoidal} and \ref{ZMon} are similar to Theorems 1(c) and 2 of \cite{GZ}. Namely Lemma \ref{TildeBoxtimes} and Remark \ref{AlmostProper} give an isomorphism \beq\label{MonoidalityIso} Z(A) * Z(B) \simeq m_* \psi(S(A) \widetilde{\boxtimes} S(B) \boxtimes R \boxtimes R \boxtimes \bZ_{\bC^{\times}}) \simeq \psi m_* (S(A) \widetilde{\boxtimes} S(B) \boxtimes R \boxtimes R \boxtimes \bZ_{\bC^{\times}}) \simeq Z(A \otimes B),\eeq such that monodromy acts by $\mu_{A*B} = \mu_A * \mu_B$.

Associativity can be proved using a triple convolution version of Gaitsgory's family. Since nearby cycles is exact for the perverse t-structure,\footnote{See Corollary 10.3.13 of \cite{KS} for a proof in the weakly constructible analytic setting.} we need not check higher coherences.

Part \ref{ZLoopMon} is similar to Theorem 7.8(6) of \cite{BR22}. As in Lemma 4.3 of \cite{BR22}, inverse loop rotation gives a $\bC^{\times}$-action on $\Fl_{\bC}$ such that
\begin{enumerate}
\item[-] $S(A) \boxtimes R \boxtimes \bZ_{\bC^{\times}}$ is $\bC^{\times}$-equivariant,
\item[-] $\Fl_{\bC} \rightarrow \bC$ is $\bC^{\times}$-equivariant for the scaling action on $\bC$,
\item[-] and $\bC^{\times}$ acts on the zero fiber $\Fl_0 = \Fl$ by inverse loop rotation.
\end{enumerate}
Therefore claim 2 of \cite{AB} implies $\mu = m$.

Now we prove part \ref{ZPerv}. Since $Z$ is monoidal it preserves dualizability, hence also compactness by Lemma \ref{Compact}. We already observed that nearby cycles is exact for the perverse t-structure. Therefore  $Z(A) \otimes_R M \simeq \psi(S(A) \boxtimes M \boxtimes \bZ_{\bC^{\times}})$ is perverse, for any finitely generated $R$-module $M$.
\end{proof}

\begin{remark}\label{AlmostProper}
We used that nearby cycles commutes with pushforward along the convolution map (compatibly with monodromy). Although $m$ is not proper, it factors through projection to the base of a real vector bundle followed by a proper map. Nearby cycles commutes with pushforward along projection to the base of the vector bundle, because our sheaves are constant along the fibers. Moreover nearby cycles always commutes with proper pushforward. In future we will often use this argument in place of properness without further comment.
\end{remark}

We used the following argument from Section 5.2.3 of \cite{GZ}.

\begin{lemma}\label{TildeBoxtimes}
Nearby cycles in the family $(\Fl \widetilde{\times} \Fl)_{\bC}$ from $\Gr \widetilde{\times} \Gr \times G/U \times G/U \times \bC^{\times}$ to $\Fl \widetilde{\times} \Fl$ sends \beq \label{PsiConv} \psi(S(A) \widetilde{\boxtimes} S(B) \boxtimes R \boxtimes R \boxtimes \bZ_{\bC^{\times}}) \simeq Z(A) \widetilde{\boxtimes} Z(B),\eeq such that monodromy acts by $\mu_A \widetilde{\boxtimes} \mu_B$.
\end{lemma}
\begin{proof}
Projection onto the first factor $(\Fl \widetilde{\times} \Fl)_{\bC} \rightarrow \Fl_{\bC}$ is given by forgetting $E$.
We will pull back along a map $\cU \rightarrow \Fl_{\bC}$ such that
\begin{enumerate}
\item[-] $\cU$ maps étale surjectively onto the closure in $\Fl_{\bC}$ of the support of $S(A) \boxtimes R \boxtimes \bZ_{\bC^{\times}}$ on $\Fl_{\bC^{\times}}$,
\item[-] and the universal $G$-bundle $E'$ on $\cU \times \bC$ is trivializable.
\end{enumerate}
Such a map exists by Footnote 5 of \cite{GZ}.

Let $\beta$, $\beta'$ be two such trivializations of the universal $G$-bundle on $\cU \times \bC$, compatible with the $U$-reduction along $\cU \times 0$. They induce two different splittings \beq \label{USplitting} b, b' : \cU \times_{\Fl_{\bC}} (\Fl \widetilde{\times} \Fl)_{\bC} \simeq \cU \times_{\bC} \Fl_{\bC},\eeq
because, having trivialized $E'$, the data of $E|_{\bC-z} \simeq E'|_{\bC-z}$ is equivalent to a trivialization of $E|_{\bC-z}$. The two trivializations differ by a map $\cU \rightarrow \mathring{I}_{\bC} \coloneqq G[t] \times_G U$, and the corresponding splittings \eqref{USplitting} differ by $\mathring{I}_{\bC}$ acting on the second factor $\Fl_{\bC}$.

Let $\cU_0$ be the zero fiber. Restricting $\beta$ to the formal disc around 0 gives a factorization
\[\begin{tikzcd}
\cU_0 \times_{\Fl} (\Fl \widetilde{\times} \Fl) \arrow[r] \arrow[d, "b"'] & \Fl \widetilde{\times} \Fl \\
\cU_0 \times \Fl \arrow[r] & G(\!(t)\!) \times \Fl \arrow[u]
\end{tikzcd}\]
and hence an isomorphism \[(Z(A) \widetilde{\boxtimes} Z(B))|_{\cU_0 \times_{\Fl} (\Fl \widetilde{\times} \Fl)} \simeq b^*(Z(A)|_{\cU_0} \boxtimes Z(B)).\]
The other trivialization $\beta'$ induces another such isomorphism, and the composition \[b'^*(Z(A)|_{\cU_0} \boxtimes Z(B)) \simeq (Z(A) \widetilde{\boxtimes} Z(B))|_{\cU_0 \times_{\Fl} (\Fl \widetilde{\times} \Fl)} \simeq b^*(Z(A)|_{\cU_0} \boxtimes Z(B))\] is induced by the $\mathring{I}_{\bC}$-equivariant structure on $Z(B)$.

Similarly, on the nonzero fibers, restricting $\beta$ to the formal disc around $z$ gives an isomorphism
\[(S(A) \widetilde{\boxtimes} S(B) \boxtimes R \boxtimes R \boxtimes \bZ_{\bC^{\times}})|_{\cU \times_{\Fl_{\bC^{\times}}} (\Fl \widetilde{\times} \Fl)_{\bC^{\times}}} \simeq b^*(S(A) \boxtimes R \boxtimes \bZ_{\bC^{\times}})|_{\cU} \boxtimes_{\bC^{\times}}(S(B) \boxtimes R \boxtimes \bZ_{\bC^{\times}}),\] related to the isomorphism induced by $\beta'$ by the $\mathring{I}_{\bC}$-equivariant structure on $(S(B) \boxtimes R \boxtimes \bZ_{\bC^{\times}})$.

Using Theorem 1.0.4 of \cite{S03}, nearby cycles on $\cU \times_{\bC} \Fl_{\bC}$ sends \[\psi((S(A) \boxtimes R \boxtimes \bZ_{\bC^{\times}}) \boxtimes_{\bC^{\times}}(S(B) \boxtimes R \boxtimes \bZ_{\bC^{\times}})) \simeq Z(A)|_{\cU_0} \boxtimes Z(B),\] such that monodromy acts by $\mu_A|_{\cU_0} \boxtimes \mu_B$.

Therefore nearby cycles on $\cU \times_{\Fl_{\bC}} (\Fl \widetilde{\times} \Fl)_{\bC}$ sends \beq \label{UMonoidal}\psi(S(A) \widetilde{\boxtimes} S(B) \boxtimes R \boxtimes R \boxtimes \bZ_{\bC^{\times}})|_{\cU_0 \times_{\Fl} (\Fl \widetilde{\times} \Fl)} \simeq (Z(A) \widetilde{\boxtimes} Z(B))|_{\cU_0 \times_{\Fl} (\Fl \widetilde{\times} \Fl)},\eeq such that monodromy acts by $(\mu_A \widetilde{\boxtimes} \mu_B)|_{\cU \times_{\Fl} (\Fl \widetilde{\times} \Fl)}$.

Moreover \eqref{UMonoidal} is independent of the choice of $\beta$, by the definition of the $\mathring{I}$-equivariant structure on $Z(B)$. Therefore by perversity and Proposition 10.2.7 of \cite{KS} it descends to the desired isomorphism \eqref{PsiConv}.
\end{proof}

\subsection{Centrality}
Now we prove that central sheaves are convolution exact, and construct the centrality isomorphisms.

\begin{proposition}\label{Centrality}
Let $A \in \Rep(\sfG)$ and $B \in \Perv_{(I)}(\Fl)$.
Then $Z(A)$ is convolution exact, and there is an isomorphism of perverse sheaves \beq\label{CentralityData} Z(A) * B \simeq B * Z(A),\eeq such that
\beq \label{CentralityMonodromy}\begin{tikzcd}
Z(A) * B \arrow[d, "\mu_A * \id_B"'] \arrow[r]  & B * Z(A) \arrow[d, "\id_B * \mu_A"]\\
Z(A) * B \arrow[r] & B * Z(A)
\end{tikzcd}\eeq
commutes.
\end{proposition}
\begin{proof}
The proof follows Proposition 6 of \cite{GZ}, although Gaitsgory does not explicitly check commutativity of \eqref{CentralityMonodromy}.

The strategy is to identify both $Z(A) * B$ and $B * Z(A)$ with nearby cycles from $\Gr \times \Fl \times \bC^{\times}$ to $\Fl$ in the following family. Let $\Fl'_{\bC}$ be the moduli of: a $G$-bundle $E$ on $\bC$, a $U$-reduction of $E|_0$, a point $z \in \bC$, and a trivialization of $E|_{\bC - \{0, z\}}$.
It suffices to prove the following two claims.
\begin{enumerate}[label=(\roman*)]
\item \label{ZABpsi}$Z(A) * B \simeq \psi(S(A) \boxtimes B \boxtimes \bZ_{\bC^{\times}})$ such that monodromy acts by $\mu_A * \id_B$.
\item \label{BZApsi} $B * Z(A) \simeq \psi(S(A) \boxtimes B \boxtimes \bZ_{\bC^{\times}})$ such that monodromy acts by $\id_B * \mu_A$.
\end{enumerate}

We now prove \ref{ZABpsi} using the following family.

Let $\Fl_{\bC} \widetilde{\times} \Fl$ be the moduli of: $G$-bundles $E, E'$ on $\bC$, $U$ reductions of $E|_0 $ and $E'|_0$, a point $z \in \bC$, a trivialization of $E'|_{\bC - z}$, and an isomorphism $E|_{\bC^{\times}} \simeq E'|_{\bC^{\times}}$. Forgetting $E'$ gives a map $m: \Fl_{\bC} \widetilde{\times} \Fl \rightarrow \Fl'_{\bC}$, that is generically a $G/U$-torsor and induces convolution $\Fl \widetilde{\times} \Fl \rightarrow \Fl$ on the zero fiber.

For $A \in \Rep(\sfG)$ and $B \in \Perv_{(I)}(\Fl)$, we claim that \beq \label{(i)ZAB} Z(A) \widetilde{\boxtimes} B \simeq \psi((S(A) \boxtimes R) \widetilde{\boxtimes} B  \boxtimes \bZ_{\bC^{\times}}),\eeq is obtained by nearby cycles from $\Fl_{\bC^{\times}} \widetilde{\times} \Fl \simeq \Gr \times G/U \times \Fl \times \bC^{\times}$ to $\Fl \widetilde{\times} \Fl$, and moreover that that monodromy acts by $\mu_A \widetilde{\boxtimes} \id_B$.

For this, let $\Fl^{\infty}_{\bC} \times \Fl$ classify the same data as $\Fl_{\bC} \widetilde{\times} \Fl$, plus a trivialization of $E'$ on the formal disc at 0, compatible with the $U$-reduction at 0.

Then \eqref{(i)ZAB} follows because nearby cycles commutes with pro-smooth pullback along the uniformization maps  \[\Fl_{\bC} \times \Fl \leftarrow \Fl^{\infty}_{\bC} \times \Fl \rightarrow \Fl_{\bC} \widetilde{\times} \Fl.\footnote{To avoid sheaves with infinite dimensional support, one can argue as in \cite{GZ}. Construct a finite dimensional uniformization, by trivializing $E'$ on the $n$th order disc, for $n$ sufficiently large compared to the support of $B$.}\]

Using that nearby cycles commutes with pushforward along $m: \Fl_{\bC} \widetilde{\times} \Fl \rightarrow \Fl'_{\bC}$, we obtain \beq\label{ZAB} \psi(S(A) \boxtimes B \boxtimes \bZ_{\bC^{\times}}) \simeq \psi m_*((S(A) \boxtimes R) \widetilde{\boxtimes} B \boxtimes \bZ_{\bC^{\times}})  \simeq m_* \psi((S(A) \boxtimes R) \widetilde{\boxtimes} B \boxtimes \bZ_{\bC^{\times}}) \simeq Z(A)* B\eeq such that monodromy acts by $\mu_A * \id_B$.

In particular $Z(A) * B$ is perverse by exactness of nearby cycles. Therefore central sheaves are convolution exact.

The proof of \ref{BZApsi} is similar to the proof of \ref{ZABpsi}. Note that the proof presented in Section 4.3 of \cite{GZ} is more involved, because Gaitsgory's argument does not assume that $B$ is weakly $I$-constructible.
\end{proof}

\subsection{Central sheaves are Wakimoto filtered}
The following proposition categorifies Bernstein's description of the center of the affine Hecke algebras, described in Equation (8.2) of \cite{Lus}.  The proof follows Proposition 5 of \cite{AB} and uses the notion of universal perversity.

First we generalize Wakimoto sheaves to all elements of the affine Weyl group, not just the weight lattice. Define $W_v \coloneqq W_{\lambda} * \nabla_x$, where we uniquely expressed $v = \lambda x \in W$ as a product of $\lambda \in \Lambda$ and $x \in W^{\fnt}$. Note that these generalized Wakimoto sheaves can be defined canonically using Remark \ref{CanonicalWhittaker}.

\begin{enumerate}
\item[-] If $\lambda \in \Lambda^+$ is dominant and $x \in W^{\fnt}$, then $\lambda$ is minimal length in its left coset $\lambda W^{\fnt}$, hence $W_{\lambda x} \simeq \nabla_{\lambda x}$.
\item[-] If $\lambda \in \Lambda^{--}$ is strictly anti-dominant and $x \in W^{\fnt}$, then $\lambda$ is maximal length in its left coset $\lambda W^{\fnt}$, hence $W_{\lambda x} \simeq \Delta_{\lambda x}$.
\end{enumerate}

Let $\langle W[\geq 0] \rangle$ and $\langle W[\leq 0] \rangle$ be the subcategories of $\cH$ generated under extensions by $W_v[i]$ for $v \in W$ and respectively $i \geq 0$ and $i \leq 0$.

Now we prove that central sheaves admit Wakimoto filtrations indexed by the weight lattice.

\begin{proposition}\label{CentralFiltered}
If $A \in \Rep(\sfG)$ then $Z(A) \in \Perv_{(I)}^{\waki}(\Fl)$.
\end{proposition}
\begin{proof}
If $\mu \in \Lambda$ then $\Delta_{-\mu} * Z(A) \in \langle \Delta [\geq 0] \rangle$ is universally perverse. If $\mu$ is sufficiently dominant, then $\Delta_{-\mu} * Z(A) \in \langle W [\geq 0] \rangle$ by Lemma 15 of \cite{AB}. If $\nu \in \Lambda^+$ is dominant, then $W_{\nu - \mu} * Z(A) \simeq \nabla_{\nu} * \Delta_{-\mu} * Z(A) \in \langle W [\geq 0] \rangle$.
On the other hand $W_{\nu -\mu} * Z(A) \in \langle \nabla [\leq 0] \rangle$ is universally perverse. If $\nu$ is sufficiently dominant compared to $\mu$, then $W_{\nu -\mu} * Z(A) \in \langle W [\leq 0] \rangle$.
Therefore $W_{\nu - \mu} * Z(A) \in \langle W [\geq 0] \rangle \cap \langle W [\leq 0] \rangle$ admits a generalized Wakimoto filtration, hence so does $Z(A)$.

Since the left and right monodromy actions coincide, $Z(A) \in \Perv_{(I)}^{\waki}(\Fl)$ admits a Wakimoto filtration indexed only by the weight lattice.
\end{proof}

\subsection{Centrality}
Recall from Section 2 of \cite{BI} that a central structure on a monoidal functor is equivalent to a braided monoidal factorization through the Drinfeld center of the target.

\begin{proposition}\label{Centrality2}
Let $\cA \subset \Perv_{(I)}(\Fl)$ be a full subcategory that is closed under convolution.
Then Gaitsgory's functor lifts to a braided monoidal functor \beq \label{ZBraidedLift} Z: \Free_{\sfG}(\sfG) \rightarrow \Z(\cA).\eeq
\end{proposition}
\begin{proof}
It suffices to check that the monoidality isomorphisms \eqref{MonoidalityIso} and the central structure \eqref{CentralityData} satisfy the braiding compatibilities. This follows by the same arguments as in \cite{GBraid} (see also Proposition 13 of \cite{B}).

To prove commutativity of
\[\begin{tikzcd}
Z(A) * Z(B) \arrow[r] \arrow[d] & Z(B) * Z(A) \arrow[d] \\
Z(A * B) \arrow[r] & Z(B * A),
\end{tikzcd}\]
Lemma \ref{2DBase} shows that the properties of nearby cycles in a two-dimensional base used in Section 5 of \cite{GBraid} also hold in our weakly constructible setting.
\end{proof}

We used the following result about nearby cycles in a two-dimensional base.
Consider two varieties $\cY_1 \rightarrow \bC$ and $\cY_2 \rightarrow \bC$, and set $\cY \coloneqq \cY_1 \times \cY_2$.

Let \[\psi_{\Delta}: \Shv(\cY) \rightarrow \Shv(\cY_{0 \times 0})\]
be restriction to the diagonal $\cY_{\Delta} \coloneqq \cY_1 \times_{\bC} \cY_2$ followed by nearby cycles.

Let \[\psi \psi: \Shv(\cY) \rightarrow \Shv(\cY_{\bC \times 0}) \rightarrow \Shv(\cY_{0 \times 0})\]
be the composition of the two nearby cycles functors.

\begin{lemma}\label{2DBase}
Consider two varieties $\cY_1 \rightarrow \bC$ and $\cY_2 \rightarrow \bC$, and set $\cY \coloneqq \cY_1 \times \cY_2$. There is a functor \[\gamma: \Shv(\cY) \rightarrow \Shv(\cY_{0 \times 0})\]
that commutes with proper pushforward and smooth pullback, and is equipped with natural transformations
 \beq\label{GammaCorrespondence} \psi_\Delta(A) \leftarrow \gamma(A) \rightarrow \psi \psi(A)\eeq
 such that
 \begin{enumerate}[label=(\alph*)]
 \item both maps in \eqref{GammaCorrespondence} are isomorphisms when evaluated on external tensor products $A \simeq A_1 \boxtimes A_2$ for $A_1 \in \Shv(\cY_1)$ and $A_2 \in \Shv(\cY_2)$,
\item the isomorphisms witnessing commutativity of the three nearby cycles functors with proper pushfoward and smooth pullback commute with the natural transformations in \eqref{GammaCorrespondence}.
 \end{enumerate}
\end{lemma}
\begin{proof}
Let $\exp \times \exp: \cY \times_{\bC \times \bC}\bC \times \bC \rightarrow \cY$ be induced by the exponential map $\exp \times \exp: \bC \times \bC \rightarrow \bC \times \bC$.
Define \[\gamma(A) \coloneqq ((\exp \times \exp)_*(\exp \times \exp)^*A)|_{\cY_{0 \times 0}}.\]

Using the Beck-Chevalley morphism, for the commutative square pulled back from 
\[\begin{tikzcd}
\bC \arrow[r, "\Delta"] \arrow[d, "\exp"'] & \bC \times \bC \arrow[d, "\exp \times \exp"] \\
\bC \arrow[r, "\Delta"'] & \bC \times \bC,
\end{tikzcd}\]
we get a natural transformation 
\beq\label{GammaPsiDelta} \gamma(A) \simeq (\Delta^* (\exp \times \exp)_* (\exp \times \exp)^*A)|_{\cY_{0 \times 0}} \rightarrow \psi_{\Delta}(A) \simeq (\exp_* \exp^* \Delta^* A)|_{\cY_{0 \times 0}}.\eeq
If $A \simeq A_1 \boxtimes A_2$ is an external tensor product, then there is a commutative square
\[\begin{tikzcd}
\gamma(A) \arrow[r, "\eqref{GammaPsiDelta}"] & \psi_{\Delta}(A) \arrow[d, "\sim"] \\
\psi(A_1) \boxtimes_{\bC} \psi(A_2) \arrow[r, "\sim"'] \arrow[u, "\sim"]& \psi(A_1 \boxtimes_{\bC} A_2),
\end{tikzcd}\]
where the bottom row is an isomorphism by Theorem 1.0.4 of \cite{S03} or Proposition 9.1.6(3) of \cite{AR}. Therefore \eqref{GammaPsiDelta} is an isomorphism for external tensor products.

Using the commutative diagram pulled back from 
\[\begin{tikzcd}[column sep = large]
& \bC \times \bC  & \arrow[l] \bC \times 0  \\
\bC \times \bC \arrow[r, "\id \times \exp"] \arrow[d, "\exp \times \id"'] & \bC \times \bC  \arrow[d, "\exp \times \id"] \arrow[u, "\exp \times \id"] \arrow[ur, phantom, "(i)"] & \arrow[l] \bC \times 0 \arrow[u, "\exp \times \id"'] \arrow[d, "\exp \times \id"]\\
\bC \times \bC \arrow[r, "\id \times \exp"'] \arrow[d, "\id \times \exp"'] \arrow[ur, phantom, "(ii)"]& \bC \times \bC & \arrow[l] \bC \times 0 \\
\bC \times \bC & & 
\end{tikzcd}\]
we get a natural transformation \beq\label{GammaPsiPsi} \gamma(A) \rightarrow \psi \psi (A).\eeq
If $A \simeq A_1 \boxtimes A_2$ is an external tensor product, then the Beck--Chevalley morphism for the two commutative squares (i) and (ii) are both isomorphisms, and therefore \eqref{GammaPsiPsi} is an isomorphism.
\end{proof}

\section{The Pl\"ucker relations}\label{s:Plucker}
To lift Gaitsgory's construction to a functor out of $\Rep(\sfB)$, we need to construct highest weight arrows satisfying the Pl\"ucker relations. This is slightly more delicate in the universal monodromic setting, because certain sections of Gaitsgory's family acquire poles at the zero fiber. We will use that nearby cycles commutes with open restriction to the highest weight semi-infinite orbit. Therefore it will suffice to study a version of Gaitsgory's family for the Borel.

The results of this section will also be important in Section \ref{Fiberfunctor} to prove that associated graded of the central functor is monoidally isomorphic to restriction to the maximal torus.

\subsection{Gaitsgory's family for the Borel}
First we introduce versions of Gaitsgory's family for the Borel.

Let $\Fl_B \coloneqq B(\!(t)\!) \times^{B[\![t]\!]} B/U$, a $B/U$-torsor over $\Gr_B \coloneqq B(\!(t)\!)/B[\![t]\!]$.

Let $\Fl_{B, \bC}$ be the moduli of: a $B$-bundle $F$ on $\bC$, a trivialization of $F|_0 \times^B T$, a point $z \in \bC$, and a trivialization of $F|_{\bC - z}$. The zero fiber is $\Fl_B$ and the nonzero fiber is $\Gr_B \times B/U$.

Let $(\Fl \widetilde{\times} \Fl)_{B, \bC}$ be the moduli of: $B$-bundles $F, F'$ on $\bC$, trivializations of $F|_0$ and $F'|_0$, a point $z \in \bC$, a trivialization of $F'|_{\bC - z}$, and an isomorphism $F|_{\bC - z} \simeq F'|_{\bC - z}$. The zero fiber is $\Fl_B \widetilde{\times} \Fl_B$ and the nonzero fiber is $\Gr_B \widetilde{\times} \Gr_B \times B/U \times B/U$.

Forgetting $F'$ gives the convolution map $m: (\Fl \widetilde{\times} \Fl)_{B, \bC} \rightarrow \Fl_{B, \bC}$.

\subsection{Two different trivializations}
Here we describe two different trivializations away from 0 of Gaitsgory family for the Borel. The first is used to define Gaitsgory's functor, and the second extends through 0.

If $z \neq 0$ then there are two $U$-reductions at 0 that differ by an element of $T$. This gives trivializations \beq\label{GenericTriv}\Fl_{B, \bC^{\times}} \simeq \Gr_B \times T \times \bC^{\times} \quad \text{and} \quad \Fl_{B, \bC^{\times}} \simeq \Gr_B \widetilde{\times} \Gr_B \times T \times T \times \bC^{\times}.\eeq

Let $F$ be a $B$-bundle on $\bC$, equipped with a  trivialization of $F|_{\bC - z}$. The choice of coordinate determines a specific global trivialization of $F \times^B T$, which can be compared at 0 to a trivialization of $F|_0 \times^B T$.
This gives global trivializations \beq \label{GlobalTriv} \Fl_{B, \bC} \simeq \Gr_B \times T \times \bC \quad \text{and} \quad (\Fl \widetilde{\times} \Fl)_{B, \bC} \simeq \Gr_B \widetilde{\times} \Gr_B \times T \times T \times \bC. \eeq

The trivializations \eqref{GenericTriv} and \eqref{GlobalTriv} are not compatible. Rather on the $\lambda$-connected component $\Gr_B^{\lambda}$, they differ by the transition function \beq\label{DifferentTrivializationsTorus} \Gr_B^{\lambda} \times T \times \bC^{\times} \rightarrow \Gr_B^{\lambda} \times T \times \bC^{\times}, \qquad (b, 1, z) \mapsto (b, \lambda(z), z).\footnote{We used the choice of coordinate to write $\lambda(z)$.}\eeq

\subsection{Gaitsgory's functor for the Borel}
Here we define Gaitsgory's functor for the Borel, identify it with the identity functor, and check a monoidality property. The challenge is that the trivialization \eqref{GenericTriv} does not extend through 0, but this is overcome using the weak $T$-constructibility.

Define Gaitsgory's functors for the Borel
\[Z: \Shv(\Gr_B \times T) \rightarrow \Shv(\Fl_B) \quad \text{and} \quad Z: \Shv(\Gr_B \widetilde{\times} \Gr_B \times T \times T) \rightarrow \Shv(\Fl_B \widetilde{\times} \Fl_B)\]
by nearby cycles \[Z(A \boxtimes R) \coloneqq \psi(A \boxtimes R \boxtimes \bZ_{\bC^{\times}}) \quad \text{and} \quad Z(A \widetilde{\boxtimes} B \boxtimes R \boxtimes R) \coloneqq \psi(A \widetilde{\boxtimes} B \boxtimes R \boxtimes R \boxtimes \bZ_{\bC^{\times}}).\]
Here $A \boxtimes R \boxtimes \bZ_{\bC^{\times}}$ and $A \widetilde{\boxtimes} B \boxtimes R \boxtimes R \boxtimes \bZ_{\bC^{\times}}$ were defined using the trivialization \eqref{GenericTriv}.

Using the choice of coordinate, identify \beq\label{FlBSplit} \Fl_B \simeq \Gr_B \times T \quad \text{and} \quad \Fl_B \widetilde{\times} \Fl_B \simeq \Gr_B \widetilde{\times} \Gr_B \times T \times T.\eeq With respect to these isomorphisms, we will show that Gaitsgory's functor for the Borel is naturally isomorphic to the identity functor.

\begin{lemma}\label{ZBorelTriv}
For $A, B \in \Shv_{B[\![t]\!]}(\Gr_B)$, under \eqref{FlBSplit} there are canonical isomorphisms \beq \label{ZBTrivial} Z(A \boxtimes R) \simeq A \boxtimes R \quad \text{and} \quad Z(A \widetilde{\boxtimes} B \boxtimes R \boxtimes R) \simeq A \widetilde{\boxtimes} B \boxtimes R \boxtimes R.\eeq
\end{lemma}
\begin{proof}
Let $\exp: \Fl_{B, \bC} \times_{\bC} \bC \rightarrow \Fl_{B, \bC}$ be induced by the exponential map $\exp: \bC \rightarrow \bC$. Recall that nearby cycles is defined \[Z(A \boxtimes R) \coloneqq (\exp_* \exp^* (A \boxtimes R \boxtimes \bZ_{\bC^{\times}}))|_{\Fl_{B, 0}}\]

We may assume that $A$ and $B$ are supported on single connected components $\Gr_B^{\lambda}$ and $\Gr_B^{\mu}$. Write $(A \boxtimes R \boxtimes \bZ_{\bC})|_{\Fl_{B, \bC^{\times}}}$ and $(A \widetilde{\boxtimes} B \boxtimes R \boxtimes R \boxtimes \bZ_{\bC})|_{(\Fl \widetilde{\times} \Fl)_{B, \bC^{\times}}}$ for the external products defined with respect to the global trivialization \eqref{GlobalTriv}.

Because the universal local system $R$ is equivariant for the universal cover of the image of $\lambda:\bC^{\times} \rightarrow T$, there is a canonical isomorphism \beq \label{Exp} \exp^*(A \boxtimes R \boxtimes \bZ_{\bC^{\times}}) \simeq \exp^*((A \boxtimes R \boxtimes \bZ_{\bC})|_{\Fl_{B, \bC^{\times}}}).\footnote{Because $R$ is only weakly $T$-constructible not $T$-equivariant, $A \boxtimes R \boxtimes \bZ_{\bC^{\times}} \not\simeq (A \boxtimes R \boxtimes \bZ_{\bC})|_{\Fl_{B, \bC^{\times}}}$ are not isomorphic.}\eeq This gives a canonical isomorphism $Z(A \boxtimes R) \simeq A \boxtimes R$.

Because the universal local system $R \boxtimes R$ is equivariant for the universal cover of the image of $(\lambda, \mu):\bC^{\times} \rightarrow T \times T$, there is a canonical isomorphism \beq \label{ExpProduct} \exp^*(A \widetilde{\boxtimes} B \boxtimes R \boxtimes R \boxtimes \bZ_{\bC^{\times}}) \simeq \exp^*((A \widetilde{\boxtimes} B \boxtimes R \boxtimes R \boxtimes \bZ_{\bC})|_{(\Fl \widetilde{\times} \Fl)_{B, \bC^{\times}}}).\eeq This gives a canonical isomorphism $Z(A \widetilde{\boxtimes} B \boxtimes R \boxtimes R) \simeq A \widetilde{\boxtimes} B \boxtimes R \boxtimes R$.
\end{proof}

Now we prove a monoidality property of Gaitsgory's functor for the Borel.

\begin{lemma}\label{ZBConvolution}
For $A, B \in \Shv_{B[\![t]\!]}(\Gr_B)$ there is a canonical isomorphism \beq\label{ZBMonoidal} Z((A * B) \boxtimes R) \simeq m_*Z(A \widetilde{\boxtimes} B \boxtimes R \boxtimes R),\eeq such that the following square commutes
\beq \label{ZMonoidalBorel}\begin{tikzcd}
Z((A * B) \boxtimes R) \arrow[r] \arrow[d, "\eqref{ZBTrivial}"'] & m_*Z(A \widetilde{\boxtimes} B \boxtimes R \boxtimes R) \arrow[d, "\eqref{ZBTrivial}"] \\
(A*B) \boxtimes R \arrow[r] & m_* (A \widetilde{\boxtimes} B \boxtimes R \boxtimes R). \\
\end{tikzcd}\eeq
\end{lemma}
\begin{proof}
The isomorphism \eqref{ZBMonoidal} follows because nearby cycles commutes with pushforward along convolution $m: (\Fl \widetilde{\times} \Fl)_{B, \bC} \rightarrow \Fl_{B, \bC}$. The square \eqref{ZMonoidalBorel} commutes because \eqref{Exp} is the pushforward along $m$ of \eqref{ExpProduct}.
\end{proof}

\subsection{Highest weight arrows}
Now we construct highest weight arrows and check the Pl\"{u}cker relations.\footnote{The proof of Lemma 10(b) of \cite{AB} needs more care in the universal setting because $\Fl_{\bC} \rightarrow \Gr \times \bC$ does not lift to a $T$-equivariant map to $\Gr \times T \times  \bC$.}

Let $M_{\lambda} \in \Rep(\sfG)$ be the dual Weyl module of highest weight $\lambda$. For dimension reasons, $\Fl_{\lambda}$ is open in the support of $Z(M_{\lambda})$. (Typically it is not dense because the support contains other irreducible components.)

Let $\cY^{\lambda}_{\bC}$ be the closure of $\Gr^{\lambda} \times T \times \bC^{\times}$ inside $\Fl_{\bC}$. Let $\cU^{\lambda}_{\bC} \coloneqq \Fl_{B, \bC}^{\lambda} \cap \cY^{\lambda}_{\bC}$ be its intersection with a semi-infinite orbit.

Let $(\cY^{\lambda} \widetilde{\times} \cY^{\mu})_{\bC}$ be the closure of $\Gr^{\lambda} \times \Gr^{\mu} \times T \times T \times \bC^{\times}$ inside $(\Fl \widetilde{\times} \Fl)_{\bC}$. Let $(\cU^{\lambda} \widetilde{\times} \cU^{\mu})_{\bC} \coloneqq (\Fl^{\lambda} \widetilde{\times} \Fl^{\mu})_{B, \bC} \cap (\cY^{\lambda} \widetilde{\times} \cY^{\mu})_{\bC}$ be its intersection with a semi-infinite orbit.

The following claim is similar to Proposition 15 of \cite{B}. But the section $(t^{\lambda}, 1, \id): \bC^{\times} \rightarrow \Fl_{\bC^{\times}} \simeq \Gr \times G/U \times \bC^{\times}$ does not seem to extend to $\bC$.

\begin{proposition}
There is a canonical isomorphism $Z(M_{\lambda})|_{\Fl_{\lambda}} \simeq R_{\Fl_{\lambda}}$.
\end{proposition}
\begin{proof}
Nearby cycles commutes with restriction along the open inclusion $\cU^{\lambda}_{\bC} \hookrightarrow \cY^{\lambda}_{\bC}$. Therefore Lemma \ref{ZBorelTriv} implies a canonical isomorphism \[Z(M_{\lambda})|_{\Fl_{\lambda}} \simeq Z(S(M_{\lambda})|_{\Gr_B^{\lambda}} \boxtimes R) \simeq S(M_{\lambda})|_{\Gr_B^{\lambda}} \boxtimes R \simeq R_{\Fl_{\lambda}}. \qedhere\]
\end{proof}

By adjunction we obtain highest weight arrows \[b_{\lambda}: Z(M_{\lambda}) \rightarrow W_{\lambda}.\] Below we check that they satisfy the Pl\"ucker relations. We will use that the support of $Z(M_{\lambda}) \widetilde{\boxtimes} Z(M_{\mu})$ is the zero fiber of $\cY^{\lambda} \widetilde{\times} \cY^{\mu}$, and its intersection of $m^{-1}(\Fl_{\lambda + \mu})$ is $\Fl_\lambda \widetilde{\times} \Fl_\mu$.

\begin{proposition}\label{Plucker}
The following square commutes
\beq \label{EqPlucker} \begin{tikzcd} 
Z(M_{\lambda}) * Z(M_{\mu}) \arrow[r, "b_{\lambda} * b_{\mu}"] \arrow[d] & W_{\lambda} * W_{\mu}  \arrow[d, "\sim"]\\
Z(M_{\lambda + \mu}) \arrow[r, "b_{\lambda + \mu}"'] & W_{\lambda + \mu} 
 \end{tikzcd}\eeq
\end{proposition}
\begin{proof}
Construct two Cartesian cubes by pulling back the Cartesian square
\[\begin{tikzcd} (\cU^{\lambda} \widetilde{\times} \cU^{\mu})_{\bC} \arrow[r, "j"] \arrow[d, "m"']& (\cY^{\lambda} \widetilde{\times} \cY^{\mu})_{\bC} \arrow[d, "m"] \\
\cU^{\lambda + \mu}_{\bC} \arrow[r, "j"'] & \cY^{\lambda + \mu}_{\bC}
\end{tikzcd}\] to the zero fiber $i: \cY^{\lambda + \mu}_0 \rightarrow \cY^{\lambda + \mu}_{\bC}$ and along the exponential map $\exp: \cY^{\lambda + \mu}_{\bC^{\times}} \times_{\bC^{\times}} \bC \rightarrow \cY^{\lambda + \mu}_{\bC}$.

Pasting three commutative cubes, each obtained by A.3 of \cite{AHR}, we obtain a commutative hexagon
\[\begin{tikzcd}
j^* i^* \exp_* \exp^* m_* \arrow[r] \arrow[d] &  i^* \exp_* \exp^* j^* m_* \arrow[r] & i^* \exp_* \exp^* m_* j^* \arrow[d] \\
j^* m_* i^* \exp_* \exp^* \arrow[r] & m_* j^* i^* \exp_* \exp^* \arrow[r] & m_*  i^* \exp_* \exp^* j^*,
\end{tikzcd}\]
which identifies with

\[\begin{tikzcd}
Z(M_{\lambda + \mu})|_{\Fl_{\lambda + \mu}} \arrow[r] \arrow[d] & Z(S(M_{\lambda + \mu})|_{\Gr_B^{\lambda + \mu}} \boxtimes R) \arrow[r] & Z((S(M_{\lambda})|_{\Gr_B^{\lambda}} * S(M_{\mu})|_{\Gr_B^{\mu}}) \boxtimes R) \arrow[d] \\
(Z(M_{\lambda}) * Z(M_{\mu}))|_{\Fl_{\lambda + \mu}} \arrow[r] &  Z(M_{\lambda})|_{\Fl_{\lambda}} * Z(M_{\mu})|_{\Fl_{\mu}} \arrow[r] & m_* Z(S(M_{\lambda})|_{\Gr_B^{\lambda}} \widetilde{\boxtimes} S(M_{\mu})|_{\Gr_B^{\mu}} \boxtimes R \boxtimes R).
\end{tikzcd}\]
The Pl\"ucker relations \eqref{EqPlucker} now follow by Lemma \ref{ZBConvolution}.
\end{proof}

\section{The fiber functor} \label{Fiberfunctor}
Here we prove that the associated graded of the central functor is monoidally isomorphic to restriction to the maximal torus. This will be important in Proposition \ref{F'}\ref{F'2} below, to identify the associated graded of the Arkhipov--Bezrukavnikov functor.

By a Tannakian argument of Arkhipov--Bezrukavnikov, it suffices to identify a certain $\sfG$-torsor on $\sfT/\sfT$. Using the Pl\"ucker relations, we will construct a $\sfB$-reduction such that the associated graded $\sfT$-torsor is $\sfT \rightarrow \sfT/\sfT$. Then we will argue that there is a unique such $\sfB$-torsor, thereby identifying the $\sfG$-torsor.

\subsection{The tensor structure}
Define the combined central and Wakimoto functor by \beq \label{ZW} Z_W: \Rep(\sfG \times \sfT) \rightarrow \Perv_{(I)}^{\waki}(\Fl), \qquad A \boxtimes \bZ_{\mu} \mapsto Z(A) * W_{\mu}.\eeq
Note that $Z_W$ admits a natural monoidal structure by the centrality of $Z$.
The following argument is from Section 3.6.5 of \cite{AB}.

\begin{lemma}\label{TensorFunctor}
The monoidal structure on $\gr Z_W$ constructed in Propositions \ref{ZProperties} and \ref{grFaithful} makes it a tensor functor.
\end{lemma}
\begin{proof}
We need to check that the monoidal structure on $\gr Z_W$ is compatible with the symmetric monoidal structures on the source and target. Since $\gr$ admits a monoidal section, the central structure on $Z$ constructed in  Proposition \ref{Centrality2} induces a central structure on $\gr Z$, denoted \beq\label{CentralStructure} z: \gr Z(M_{\lambda}) \otimes R_{\mu} \simeq R_{\mu} \otimes \gr Z(M_{\lambda}).\eeq  
We now claim that
\begin{enumerate}
\item[($\star$)] the above \eqref{CentralStructure} coincides with the commutativity constraint in $\Free_{\sfT}(\sfT)$.
\end{enumerate}

Proposition \ref{Plucker} implies that the outer square of the following diagram commutes
\[\begin{tikzcd}
Z(M_{\mu}) * Z(M_{\nu}) \arrow[r] \arrow[d, "\id * b_{\nu}"'] & Z(M_{\nu}) * Z(M_{\mu}) \arrow[d, "b_{\nu} * \id"] \\
Z(M_{\mu}) * W_{\nu} \arrow[d, "b_{\mu} * \id"'] \arrow[r] & W_{\nu} * Z(M_{\mu})  \arrow[d, "\id * b_{\mu}"]\\
W_{\mu} * W_{\nu} \arrow[r] & W_{\nu} * W_{\mu}.
\end{tikzcd}\]
Since $\id * b_{\lambda}$ and $b_{\lambda} * \id$ are surjective, the bottom square also commutes.

Let $c$ denote the commutativity constraint in the symmetric monoidal category $\Free_{\sfT}(\sfT)$. Taking associated graded gives commutativity of
\[\begin{tikzcd} \gr Z(M_{\mu}) \otimes R_{\nu} \arrow[d] \arrow[r, "z"] & R_{\nu} \otimes \gr Z(M_{\mu})  \arrow[d]\\
R_{\mu} \otimes R_{\nu} \arrow[r, "c"'] & R_{\nu} \otimes R_{\mu}. \end{tikzcd}\]
 
\noindent This implies commutativity of
\[\begin{tikzcd}[cramped]
\gr Z(M_{\lambda}) \otimes \gr Z(M_{\mu}) \arrow[r, "z"] \arrow[d] & \gr Z(M_{\mu}) \otimes \gr Z(M_{\lambda}) \arrow[r, "z"]\arrow[d] & \gr Z(M_{\lambda}) \otimes \gr Z(M_{\mu}) \arrow[d]\\
\gr Z(M_{\lambda}) \otimes R_{\mu} \arrow[r, "z"'] & R_{\mu} \otimes \gr Z(M_{\lambda}) \arrow[r, "c"']& \gr Z(M_{\lambda}) \otimes R_{\mu}.
\end{tikzcd}\]
The top row composes to the identity because \eqref{ZBraidedLift} is \textit{braided} monoidal, hence so does the bottom row.
Therefore ($\star$) holds for all $\mu \in \Lambda^+$.

It further holds for all $\mu \in \Lambda$ by writing $\mu =  \eta - \nu$ for $\eta, \nu \in \Lambda^+$ and using commutativity of
\[\begin{tikzcd}[column sep = tiny]
\gr Z(M_{\lambda}) \otimes R_{\mu} \otimes R_{\nu} \arrow[rr, "z \otimes \id"] \arrow[dr, "z"'] && R_{\mu} \otimes \gr Z(M_{\lambda}) \otimes  R_{\nu} \arrow[dl, "\id \otimes z"] \\
& R_{\mu} \otimes R_{\nu} \otimes \gr Z(M_{\lambda}). &
\end{tikzcd}\]

Proposition \ref{Centrality2} (see Definition 1(ii) of \cite{BI}) and ($\star$) give a commutative diagram
\[\begin{tikzcd}\gr Z(M_{\lambda}) \otimes \gr Z(M_{\mu}) \arrow[r, "c"]\arrow[d] & \gr Z(M_{\mu}) \otimes \gr Z(M_{\lambda}) \arrow[d] \\
\gr Z(M_{\lambda} \otimes M_{\mu}) \arrow[r] & \gr Z(M_{\mu} \otimes M_{\lambda}).
\end{tikzcd}\]
Moreover every finite dimensional representation of $\sfG$ admits an injection into a module that admits a dual Weyl filtration. Therefore the monoidal structure on $\gr Z_W$ intertwines the commutativity constraints.
\end{proof}

\subsection{Identifying the torsor}\label{IDTorsor}
We will (abusively) denote $\pt \coloneqq \Spec \bZ$.

By Lemma \ref{TensorFunctor} and Tannakian formalism, $\gr Z_W \simeq j^*$ is monoidally isomorphic to pullback along some map \beq \label{TannakianMap} j: \sfT/\sfT \rightarrow \pt/\sfG \times \pt/\sfT,\eeq such that the composition $\sfT/\sfT \rightarrow \pt/\sfG \times \pt/\sfT \rightarrow \pt/\sfT$ is isomorphic to projection.
It remains to identify this $(\sfG \times \sfT)$-torsor on $\sfT/\sfT$. 

For this, we first use highest weight arrows to reduce the $(\sfG \times \sfT)$-torsor to a $\sfB$-torsor.
Let $\beta_{\lambda}: M_{\lambda}|_{\sfB} \rightarrow \bZ_{\lambda}|_{\sfB}$ be the highest weight arrow, a map of $\sfB$-modules.

\begin{lemma}\label{FactorB}
The map \eqref{TannakianMap} factors through
\[\begin{tikzcd}
\sfT/\sfT \arrow[dr, "j"'] \arrow[r] & \pt/\sfB \arrow[d] \\
&  \pt/\sfG \times \pt/\sfT
\end{tikzcd}\]
such that the resulting diagram
\beq \label{grB} \begin{tikzcd}
\gr Z(M_{\lambda}) \arrow[r, "\gr b_{\lambda}"] \arrow[d]& \gr W_{\lambda} \arrow[d] \\
j^* M_{\lambda} \arrow[r, "\beta_{\lambda}|_{\sfT/\sfT}"'] & j^* \bZ_{\lambda}. \end{tikzcd}\eeq
commutes.
\end{lemma}
\begin{proof}
Let $(\sfG/\sfU)^{\aff}$ be the affine closure of base affine space. Proposition \ref{Plucker} supplies highest weight arrows satisfying the Pl\"{u}cker relations. This gives a factorization
\[\begin{tikzcd}
\sfT/\sfT \arrow[dr, "j"'] \arrow[r] & \sfT \backslash (\sfG/\sfU)^{\aff}/\sfG \arrow[d] \\
&  \pt/\sfG \times \pt/\sfT.
\end{tikzcd}\]

After taking associated graded, the highest weight arrow $\gr b_{\lambda}: \gr Z(M_{\lambda}) \rightarrow \gr W_{\lambda}$ becomes a split surjection in $\Free_{\sfT}(\sfT)$. Therefore these highest weight arrows define a genuine $\sfB$-reduction \[\sfT/\sfT \rightarrow \sfT \backslash (\sfG/\sfU)/\sfG = \pt/\sfB \rightarrow \pt/\sfG \times \pt/\sfT. \qedhere\]
\end{proof}

Having constructed a $\sfB$-reduction, we now identify the $(\sfG \times \sfT)$-torsor.

\begin{proposition}\label{ConstantTermEquivalence}
Associated graded of the central functor \[\gr Z_W: \Rep(\sfG \times \sfT) \rightarrow \Free_{\sfT}(\sfT),\]
is isomorphic as a monoidal functor to pullback along \[j: \sfT/\sfT \rightarrow \pt/\sfT \rightarrow \pt/\sfG \times \pt/\sfT,\]
where $j$ is projection followed by the map induced by the diagonal inclusion $\sfT \rightarrow \sfG \times \sfT$.
\end{proposition}
\begin{proof}
By adjunction, the map $\sfT/\sfT \rightarrow \pt/\sfB$ constructed in Lemma \ref{FactorB} corresponds to a map from $\sfT$ to the stack classifying maps $\pt/\sfT \rightarrow \pt/\sfB$ such that the composition to $\pt/\sfT$ is isomorphic to the identity map, i.e. \[\sfT \rightarrow \Hom(\pt/\sfT, \pt/\sfB) \times_{\Hom(\pt/\sfT, \pt/\sfT)} \id/\sfT.\] 

We will rewrite the target as follows. Using  that all $\sfB$-bundles on a point are trivializable, it follows that $\Hom(\pt/\sfT, \pt/\sfB) \simeq \Hom(\sfT, \sfB)/\sfB,$ so that in particular \[\Hom(\pt/\sfT, \pt/\sfB) \times_{\Hom(\pt/\sfT, \pt/\sfT)} \id/\sfT  \simeq (\Hom(\sfT, \sfB) \times_{\Hom(\sfT, \sfT)} \id)/\sfB \simeq \pt/\sfT,\] using Lemma \ref{ToriConjugate} for the final isomorphism.
Since $R$ is a UFD, all $\sfT$-torsors on $\sfT$ are trivializable. Thus $\sfT/\sfT \rightarrow \pt/\sfB$ is identified with projection to $\pt/\sfT$ followed by the map induced by inclusion of the maximal torus.
\end{proof}

We used the following lemma to show that the stack that classifies sections of $\sfB \rightarrow \sfT$ up to $\sfB$-conjugation is equivalent to $\pt/\sfT$.

\begin{lemma}\label{ToriConjugate}
Any two sections of $\sfB \rightarrow \sfT$ are conjugate. Moreover $\sfT$ is its own centralizer inside $\sfB$. 
\end{lemma}
\begin{proof} 
Let $i:\sfT \rightarrow \sfB$ and $i': \sfT \rightarrow \sfB$ be two different sections of $\sfB \rightarrow \sfT$. Let $\Transp(i, i')$ be the subfunctor of $\sfG$ whose $S$-points are elements $g \in \sfG(S)$ such that $g i_S g^{-1} = i'_S$.
Theorem 5.4 of \cite[XI]{SGA3} implies that $\Transp(i, i')$ is a $\sfT$-torsor over $\pt \coloneqq \Spec \bZ$. All such $\sfT$-torsors are trivial. Therefore the two different inclusions of $\sfT$ are conjugate.

Theorem 5.3 of \cite[XI]{SGA3} implies that the centralizer $Z_\sfB(\sfT)$ exists as a closed subgroup of $\sfB$. The map $\sfT \rightarrow Z_\sfB(\sfT)$ is an isomorphism by Corollary 17.9.5 of \cite[IV$_4$]{EGA}.
\end{proof}

\subsection{Alternative approach}
An alternative approach to Proposition \ref{ConstantTermEquivalence} is to identify the associated graded functor with the constant term functor. Using that nearby cycles commutes with hyperbolic localization by Braden's theorem \cite{Brad} then gives the desired isomorphism of plain functors.

However, it appears more difficult to prove directly that this isomorphism intertwines the monoidal structures, except on the highest weight lines where we checked the Pl\"ucker relations above. For example, in the universal monodromic setting, total cohomology computes only the fiber at $1 \in \sfT$ of constant term, so the approach of Proposition 4.8.2 of \cite{AR} does not directly adapt.

\section{Construction of the functor I (factorization through a closed subscheme)}\label{ConstructionFunctorI}
Here we begin the construction of the Arkhipov--Bezrukavnikov functor \[F: \IndPerf_{\sfG}(\widetilde{\sfG}) \rightarrow \cH.\]
We first explain that monodromy of nearby cycles and highest weight arrows gives a functor out of free coherent sheaves on $\sfG \times (\sfG/\sfU)^{\aff} \times \sfT$. Then we prove that it factors through restriction to a certain closed subscheme cut out by equations coming from the compatibility of monodromy of nearby cycles with highest weight arrows.

It will remain to factor the functor through the relevant open subscheme corresponding to base affine space inside its affine closure. This will be done in Section \ref{ConstructionFunctorII} after passing to bounded homotopy categories.

In the course of the argument, we prove that the functor of taking the associated graded of the Wakimoto filtration agrees with restriction to the maximal torus.
This property will be important in Section \ref{ConstructionFunctorII} to show that complexes supported on the irrelevant boundary are killed, and also in Part \ref{Part2} for the proof of fully faithfulness.

\subsection{Base affine space}\label{BaseAffine}
Let $(\sfG/\sfU)^{\aff}$ be the affine closure of base affine space, and \[\sfG \times \sfT \curvearrowright \sfG \times (\sfG/\sfU)^{\aff} \times \sfT \quad \text{by} \quad (a, b)(g, x, t) \coloneqq (aga^{-1}, axb^{-1}, btb^{-1}).\]
The inclusion \beq \label{i} \sfT \rightarrow \sfG \times (\sfG/\sfU)^{\aff} \times \sfT, \qquad t \mapsto (t, 1, t)\eeq
is equivariant for the diagonal $\sfT \subset \sfG \times \sfT$.

Now we construct $\widetilde{\sfG}$ as the quotient by $\sfT$ of a locally closed subscheme of $\sfG \times (\sfG/\sfU)^{\aff} \times \sfT$.
\begin{enumerate}
\item[-] Let $\sfX \coloneqq \{(g, x, t) \in \sfG \times (\sfG/\sfU)^{\aff} \times \sfT \text{ satisfying } g x = x t\}$.
\item[-] Let $\sfV \coloneqq \{(g, x, t) \in \sfX \text{ satisfying } x \in \sfG/\sfU\}$ be the relevant open subscheme of $\sfX$.
\end{enumerate}
Then $\sfV = \sfG \times^{\sfU} \sfB$ is the canonical $\sfT$-torsor over the Grothendieck--Springer scheme.

\begin{remark} Let $\sfG = \SL(2)$ and $(\sfG/\sfU)^{\aff} = \bA^2$. Then $\sfX \subset \sfG \times \bA^2 \times \sfT$ has two irreducible components: the closure of $\sfV$ and the irrelevant boundary $\sfG \times 0 \times \sfT$.
\end{remark}

\subsection{Monodromy and highest weight arrows}
Proposition \ref{ZProperties} implies that loop rotation monodromy induces a tensor automorphism of \[Z_W: \Rep(\sfG \times \sfT) \rightarrow \cH,\]
which we will denote by \[m_A: Z_W(A) \rightarrow Z_W(A).\]
Proposition \ref{Plucker} gives highest weight arrows $b_{\lambda}: Z(M_{\lambda}) \rightarrow W_{\lambda}$ satisfying the Pl\"{u}cker relations.

For $A \in \Rep(\sfG \times \sfT)$, the coaction $A \rightarrow A \otimes \cO(\sfG \times \sfT)$ induces an $\cO(\sfG \times \sfT)$-linear tensor automorphism \[M_A: A \otimes \cO(\sfG \times \sfT) \rightarrow A \otimes \cO(\sfG \times \sfT).\]

Let $M_{\lambda} \in \Rep(\sfG)$ be the dual Weyl module of highest weight $\lambda$. We fix an identification of $\cO(\sfG/\sfU)$ with the direct sum over dominant weights of $M_{\lambda} \boxtimes \bZ_{-\lambda}$. Define the highest weight arrows by \[B_{\lambda}: M_{\lambda} \otimes \cO(\sfG/\sfU) \rightarrow \cO(\sfG/\sfU) \otimes \bZ_{\lambda}, \qquad v \otimes 1 \mapsto (v \boxtimes 1) \otimes 1.\]

\subsection{The initial functor}
Monodromy and highest weight arrows give the following functor.

\begin{proposition} \label{F'}
The central and Wakimoto functor extends to a monoidal functor $F$ fitting into a commutative triangle
\beq \begin{tikzcd}\label{OriginalF'}
\Free_{\sfG \times \sfT}(\sfG \times (\sfG/\sfU)^{\aff} \times \sfT) \arrow[r, "F"] & \Perv_{(I)}(\Fl)^{\waki} \\
\Rep(\sfG \times \sfT). \arrow[ur, "Z_W"'] \arrow[u, "p^*"] &  
\end{tikzcd}\eeq
Here $p^*$ is pullback along $p : \sfG \times (\sfG/\sfU)^{\aff} \times \sfT \rightarrow \pt$. Moreover
\begin{enumerate}[label=(\alph*)]
\item \label{F'1} $F$ is compatible with monodromy and highest weight arrows in the sense that $F(M_A) = m_A$ and $F(B_{\lambda}) = b_{\lambda}$,
\item \label{F'3} $F$ is $R$-linear with respect to projection to $\sfT$ and $T$-monodromy,
\item \label{F'2} $\gr F$ is monoidally equivalent to pullback along \eqref{i}.
\end{enumerate}
\end{proposition}
\begin{proof}
First we construct the functor $F$ satisfying \ref{F'1}. Consider the commutative algebra $A \coloneqq \Hom(\Delta_1, Z_W(\cO(\sfG) \boxtimes \cO(\sfT)))$. Proposition 4 of \cite{AB} gives a commutative triangle
\beq\label{AFunct}\begin{tikzcd}
\Free_{\sfG \times \sfT}(A) \arrow[r, "F'"] & \Perv_{(I)}^{\waki}(\Fl)\\
\Rep(\sfG \times \sfT) \arrow[ur, "Z_W"'] \arrow[u]& 
\end{tikzcd}\eeq
such that $F'$ is monoidal and faithful. Loop rotation monodromy and highest weight arrows induce a map \beq\label{ToA} \cO(\sfG \times \sfG/\sfU \times \sfT) \rightarrow A.\eeq Tannakian formalism and Proposition 4 of \cite{AB} imply that the resulting functor \[F: \Free_{\sfG \times \sfT}(\sfG \times (\sfG/\sfU)^{\aff} \times \sfT) \rightarrow \Free_{\sfG \times \sfT}(A) \rightarrow \Perv_{(I)}^{\waki}(\Fl)\]
satisfies \ref{F'1}.

Now we check \ref{F'3}. Proposition \ref{WProperties} says that loop rotation monodromy on $W_{\mu}$ coincides with both left and right $T$-monodromy by $e^{\mu}$. Therefore $F$ is $R$-linear with respect to projection to $\sfT$ and $T$-monodromy.

Finally we check \ref{F'2}. Proposition \ref{ConstantTermEquivalence} implies that $\gr F p^* \simeq j^*$ is monoidally equivalent to restriction along the diagonal maximal torus $j: \sfT/\sfT \rightarrow \pt/\sfG \times \pt/\sfT$. Therefore we have a commutative diagram
\[\begin{tikzcd}
\Free_{\sfG \times \sfT}(\sfG \times (\sfG/\sfU)^{\aff} \times \sfT) \arrow[r, "F"]  & \Perv_{(I)}(\Fl) \arrow[d, "\gr"] \\
\arrow[u, "p^*"] \Rep(\sfG \times \sfT) \arrow[ur, "Z_W"] \arrow[r, "j^*"'] & \Free_{\sfT}(\sfT).
\end{tikzcd}\]
By Tannakian formalism\footnote{See for example Section 6.3.2 of \cite{AR}.} $\gr F \simeq i^*$ is isomorphic to pullback along some $\sfT$-equivariant map \[i: \sfT \rightarrow \sfG \times (\sfG/\sfU)^{\aff} \times \sfT.\] We now prove that $i$ is the map \eqref{i} sending $t \mapsto (t, 1, t)$.

There is a commutative square
\beq\label{NaturalSquare}\begin{tikzcd} \gr F p^* \arrow[r, "\sim"] \arrow[d, "\sim"'] & i^* p^* \arrow[d, "\sim"] \\
\gr Z_W \arrow[r, "\sim"'] & j^*.  \end{tikzcd}\eeq
Equation \eqref{grB} gives a commutative diagram
\beq \label{grBb} \begin{tikzcd}
i^*(M_{\lambda} \otimes \cO(\sfG/\sfU)) \arrow[r, "\sim"] \arrow[d, "i^* B_{\lambda}"] & \gr F(M_{\lambda} \otimes \cO(\sfG/\sfU)) \arrow[r, "\sim"] \arrow[d, "\gr F B_{\lambda}"] & \gr Z(M_{\lambda}) \arrow[r, "\sim"] \arrow[d, "\gr b_{\lambda}"] & M_{\lambda} \otimes \cO(\sfT) \arrow[d, "\beta_{\lambda}|_{\sfT/\sfT}"] \\
i^*(\cO(\sfG/\sfU) \otimes \bZ_{\lambda}) \arrow[r, "\sim"'] & \gr F (\cO(\sfG/\sfU) \otimes \bZ_{\lambda}) \arrow[r, "\sim"']  & \gr W_{\lambda} \arrow[r, "\sim"'] &  \cO(\sfT) \otimes \bZ_{\lambda}.
\end{tikzcd}\eeq
Combining \eqref{NaturalSquare} and \eqref{grBb} shows that under the identification $i^* p^* \simeq j^*$ we have \[i^* B_{\lambda} = \beta_{\lambda}|_{\sfT/\sfT}.\]

Therefore the commutative diagram
\[\begin{tikzcd}
M_{\lambda} \otimes \cO(\sfG/\sfU) \arrow[r, "i^*"] \arrow[d, "B_{\lambda}"'] & M_{\lambda} \otimes \cO(\sfT) \arrow[d, "i^* B_{\lambda}"] \\
\cO(\sfG/\sfU) \otimes \bZ_{\lambda} \arrow[r, "i^*"'] & \cO(\sfT) \otimes \bZ_{\lambda} 
\end{tikzcd}\]
sends
\[\begin{tikzcd}
v \otimes 1 \arrow[r, mapsto] \arrow[d, mapsto] & v \otimes 1 \arrow[d, mapsto]\\
(v \boxtimes 1) \otimes 1 \arrow[r, mapsto] & i^*(v \boxtimes 1) \otimes 1 = 1 \otimes \beta_{\lambda}(v).
\end{tikzcd}\]
Hence restriction along $i$ induces projection onto the highest weight line
\[i^*: \cO(\sfG/\sfU) = \bigoplus M_{\lambda} \boxtimes \bZ_{-\lambda} \; \rightarrow \;\cO(\sfT), \qquad v \boxtimes 1 \in M_{\lambda} \boxtimes \bZ_{-\lambda}  \; \mapsto \; \beta_{\lambda}(v) \in \bZ \subset \cO(\sfT).\]
Therefore $i: \sfT \rightarrow (\sfG/\sfU)^{\aff}$ is given by $i(t) = 1$ for all $t \in \sfT$.

Moreover $\gr m_V$ acts by $e^{\mu}$ on the $\mu$th graded piece. Hence $i: \sfT \rightarrow \sfG$ is the inclusion of the maximal torus, and $i: \sfT \rightarrow \sfT$ is the identity map. Therefore $\gr F$ is restriction along \eqref{i}.
\end{proof}

\subsection{Factorization through a closed subscheme}
Here we factor the Arkhipov--Bezrukavnikov functor through the closed subscheme $\sfX \subset \sfG \times (\sfG/\sfU)^{\aff} \times \sfT$, by checking certain equations coming from the commutativity of \eqref{ClosedEquations}.

\begin{proposition}\label{FunctorClosed}
The central and Wakimoto functor extends to a monoidal functor $F$ making the following triangle commute
\beq\label{ClosedFunct}\begin{tikzcd}
 \Free_{\sfG \times \sfT}(\sfX)  \arrow[r, "F"] & \Perv_{(I)}(\Fl)^{\waki} \\
\Rep(\sfG \times \sfT). \arrow[u, "p^*"]  \arrow[ur, "Z_W"'] &
\end{tikzcd}\eeq
Here $p^*$ is pullback along $p: \sfX \rightarrow \pt$. Moreover
\begin{enumerate}[label=(\alph*)]
\item $F(M_A) = m_A$ and $F(B_{\lambda}) = b_{\lambda}$,
\item $F$ is $R$-linear with respect to projection to $\sfT$ and $T$-monodromy,
\item $\gr F$ is monoidally equivalent to pullback along \eqref{i}.
\end{enumerate}
\end{proposition}
\begin{proof}
First we will show that \eqref{ToA} kills the equations defining $\sfX$.

Let $\Delta: M_{\lambda} \rightarrow M_{\lambda} \otimes \cO(\sfG)$ be the coaction and \beq \label{ActionMaps} \act_{\sfG}: \sfG \times (\sfG/\sfU)^{\aff} \rightarrow (\sfG/\sfU)^{\aff} \quad \text{and} \quad \act_{\sfT}: (\sfG/\sfU)^{\aff} \times \sfT \rightarrow (\sfG/\sfU)^{\aff}\eeq be the left and right actions respectively. For $v \in M_{\lambda}$ we write $\act_{\sfG}^*(v \boxtimes 1) \in \cO(\sfG \times \sfG/\sfU)$ and $\act_{\sfT}^*(v \boxtimes 1) \in \cO(\sfG/\sfU \times \sfT)$ for the pullback along \eqref{ActionMaps} of $v \boxtimes 1$ regarded as an element of $\cO(\sfG/\sfU) = M_{\lambda} \boxtimes \bZ_{-\lambda}$.

The following square does not commute \beq\label{BMNotCommute} \begin{tikzcd}
M_{\lambda} \otimes \cO(\sfG \times \sfG/\sfU \times \sfT) \arrow[r, "M_{M_{\lambda}}"] \arrow[d, "B_{\lambda}"'] & M_{\lambda} \otimes \cO(\sfG \times \sfG/\sfU \times \sfT) \arrow[d, "B_{\lambda}"]\\
\cO(\sfG \times \sfG/\sfU \times \sfT)\otimes \bZ_{\lambda} \arrow[r, "M_{\bZ_{\lambda}}"'] & \cO(\sfG \times \sfG/\sfU \times \sfT) \otimes \bZ_{\lambda}.
\end{tikzcd}\eeq
Rather for $v \in M_{\lambda}$, one can check that it sends
\[\begin{tikzcd}
v \otimes 1 \otimes 1 \otimes 1 \arrow[r, mapsto] & (\Delta v) \otimes 1 \otimes 1 \arrow[d, mapsto] \\
& (\act_{\sfG}^* (v \boxtimes 1)) \otimes 1 \otimes 1
\end{tikzcd}\]
and
\[\begin{tikzcd}
v \otimes 1 \otimes 1 \otimes 1 \arrow[d, mapsto]&\\
1 \otimes v \otimes 1 \otimes 1 \arrow[r, mapsto] & 1 \otimes v \otimes e^{\lambda} \otimes 1 = 1 \otimes (\act_{\sfT}^* (v \boxtimes 1)) \otimes 1.
\end{tikzcd}\]

Since all morphisms in $\Perv_{(I)}(\Fl)$ commute with loop rotation monodromy, we get a commutative square \beq\label{ClosedEquations}\begin{tikzcd}
Z(M_{\lambda}) \arrow[r, "m_{M_{\lambda}}"] \arrow[d, "b_{\lambda}"'] & Z(M_{\lambda}) \arrow[d, "b_{\lambda}"]\\
W_{\lambda} \arrow[r, "m_{\bZ_{\lambda}}"'] & W_{\lambda}.
\end{tikzcd}\eeq
Therefore, using that $F'$ is faithful, the restriction of \eqref{BMNotCommute} along \eqref{ToA} gives a commuting square
  \[\begin{tikzcd} M_{\lambda} \otimes A \arrow[r, "M_{M_{\lambda}}|_A"] \arrow[d, "B_\lambda|_A"'] & M_{\lambda} \otimes A \arrow[d, "B_\lambda|_A"]\\
A \otimes \bZ_{\lambda} \arrow[r, "M_{\bZ_{\lambda}}|_A"'] & A \otimes \bZ_{\lambda} .
\end{tikzcd}\]
Therefore \eqref{ToA} kills all functions of the form \[(\act_{\sfG}^* (v \boxtimes 1)) \otimes 1 - 1 \otimes (\act_{\sfT}^* (v \boxtimes 1)) \in \cO(\sfG \times \sfG/\sfU \times \sfT),\] so it factors through
\[\begin{tikzcd}
  \cO(\sfX) \arrow[r] & A\\
 \cO(\sfG \times \sfG/\sfU \times \sfT) \arrow[u]\arrow[ur, "\eqref{ToA}"'] &   
\end{tikzcd}\]
Hence \eqref{AFunct} factors though a functor that fits into a commutative triangle
\beq\label{XGraded} \begin{tikzcd}
\Free_{\sfG \times \sfT}(\sfX) \arrow[r, "F"] \arrow[dr, "\eqref{i}"'] & \Perv_{(I)}^{\waki}(\Fl) \arrow[d, "\gr"] \\
& \Free_{\sfT}(\sfT).
\end{tikzcd}\eeq
\end{proof}

\section{Multi-graded rings}\label{SectionMulti}
To complete the construction of the Arkhipov--Bezrukavnikov functor, we will interpret the Grothendieck--Springer scheme as a multi-graded proj-scheme.
The spectrum of a multi-graded ring decomposes as the union of the relevant open subscheme and the irrelevant boundary. Before proceeding, we prove the following preliminaries concerning multi-graded rings.
\begin{enumerate}
\item[-] Passing to a submonoid containing all sufficiently dominant coweights does not change the relevant open subscheme.
\item[-] In sufficiently dominant weight spaces, restriction of a finitely generated module to the relevant open subscheme does not affect its global sections nor introduce higher cohomology.
\item[-] The quotient of perfect complexes modulo those supported on the irrelevant boundary embeds fully faithfully into the category of perfect complexes on the relevant open subscheme.
\end{enumerate}
Similar results about multi-graded rings and the Grothendieck--Springer scheme are discussed in \cite{MR}.

\subsection{The relevant open subscheme}
Let $\bR$ be a ring that is graded by the monoid of dominant weights $\Lambda^+$.
Let $\bI \subset \bR$ be the direct sum of the weight spaces indexed by strictly dominant coweights.

Define the `relevant open subscheme' of $\Spec \bR$ by \beq\label{QSpecDefinition} \QSpec \bR \coloneqq \Spec \bR - \Spec \bR/\bI.\eeq
This construction depends on the weight grading not just the underlying ring. It admits a $\sfT$-equivariant map to the multi-graded proj scheme, which is a $\sfT$-torsor under certain conditions that will hold in our case of interest.

We call a multiplicative subset $\bA \subset \bR$ `relevant' if it consists of homogeneous elements and contains an element of strictly dominant weight.

The following lemma provides an affine open cover of the relevant open subscheme of the spectrum of a multi-graded ring.

\begin{lemma}\label{AffineCover}
The relevant open subscheme $\QSpec \bR$ is the union of $\Spec \bR[\bA^{-1}]$ over relevant multiplicative subsets $\bA$.
\end{lemma}
\begin{proof}
If $\bA \subset \bR$ is a relevant multiplicative subset, then $(\bR/\bI)[\bA^{-1}] = 0$ because $\bI$ contains an element of $\bA$. Hence $\Spec \bR[\bA^{-1}] \rightarrow \Spec \bR$ factors through $\QSpec \bR$. Therefore the union of $\Spec \bR[\bA^{-1}]$ over relevant multiplicative subsets is contained in $\QSpec \bR$.

Consider a point in $x \in \QSpec \bR$, i.e. a prime ideal $\bP \subset \bR$ such that $\bI \not\subset \bP$. Then there exists $a \in \bR_{\lambda} \subset \bI$ such that $\lambda \in \Lambda^{++}$ is strictly dominant and $a \not\in \bP$. Hence $\bA \coloneqq \{1, a, a^2, \dots\} \subset \bR$ is a relevant multiplicative subset and we have $x \in \Spec \bR[\bA^{-1}]$. Therefore $\QSpec \bR$ is contained in the union of $\Spec \bR[\bA^{-1}]$ over relevant multiplicative subsets $\bA$.
\end{proof}

We call an element of $\Lambda^+$ `sufficiently dominant' if $\langle \check{\alpha}, \lambda \rangle$ is sufficiently large for every simple coroot $\check{\alpha}$.

The following lemma allows us to pass to a submonoid containing all sufficiently dominant coweights without changing the relevant open subscheme.

\begin{lemma}\label{QSpecInclusion}
Let $\Lambda^{+++} \subset \Lambda^+$ be a submonoid that contains all sufficiently dominant coweights. Let $\bS$ be the direct sum of $\bR_{\lambda}$ over $\lambda \in \Lambda^{+++}$. Then $\QSpec \bR = \QSpec \bS$.
\end{lemma}
\begin{proof}
First we claim that for any relevant multiplicative subset $\bA \subset \bR$, the maps \beq\label{DifferentLocalizations} \bS[(\bA \cap \bS)^{-1}] \rightarrow \bR[(\bA \cap \bS)^{-1}] \rightarrow \bR[\bA^{-1}]\eeq are both isomorphisms.

Let $r \in \bR$ and $a \in \bA$. Choose $b \in \bA$ homogeneous of sufficiently dominant weight, so that $rb, ab \in \bS$. Then $r/a = rb/ab \in \bS[(\bA \cap \bS)^{-1}]$, so \eqref{DifferentLocalizations} are both surjective.

Suppose that $s/a \in \bS[(\bA \cap \bS)^{-1}]$ maps to zero in $\bR[\bA^{-1}]$. Then there exists $b \in \bA$ such that $sb = 0$. Choose $c \in \bA$ homogeneous of sufficiently dominant weight dominant weight, so that $bc \in \bA \cap \bS$. Then $sbc = 0$ implies that $s/a = 0 \in \bS[(\bA \cap \bS)^{-1}]$. Therefore \eqref{DifferentLocalizations} are both injective.

Therefore Lemma \ref{AffineCover} gives \[\QSpec \bR = \bigcup_{\bA} \Spec \bR[\bA^{-1}] = \bigcup_{\bA} \Spec \bS[(\bA \cap \bS)^{-1}] = \QSpec \bS. \qedhere\]
\end{proof}

The following lemma will be used in the proof of Lemmas \ref{MultiSerre} and \ref{XCostandard}.

\begin{lemma}\label{SurjectRelevant}
Let $\phi: \bS \rightarrow \bR$ be a map of $\Lambda^+$-graded rings.
\begin{enumerate}[label=(\alph*)]
\item\label{SurjectRelevant1} The open embedding $\Spec \bR \otimes_{\Spec \bS} \QSpec \bS \rightarrow \Spec \bR$ factors through $\QSpec \bR$.
\item\label{SurjectRelevant2} Assume that for every homogeneous element of $\bR$, some higher power of it is contained in the image of $\phi$. Then \[\Spec \bR \times_{\Spec \bS}  \QSpec \bS \simeq \QSpec \bR.\]
\end{enumerate}
\end{lemma}
\begin{proof}
Let $\bJ \subset \bS$ be the direct sum of the weight spaces indexed by strictly dominant coweights.

First we prove \ref{SurjectRelevant1}. Let $\bB \subset \bS$ be a relevant multiplicative subset. Choose $b \in \bB$ homogeneous of strictly dominant weight.
\begin{enumerate}
\item[-] If $\phi(b) = 0$ then $\bR \otimes_\bS \bS[\bB^{-1}] = 0$.
\item[-] If $\phi(b) \neq 0$ then $\phi(b) \in \bI \subset \bR$ is homogeneous of strictly dominant weight. Therefore $\phi(\bB) \subset \bR$ is a relevant multiplicative subset. Hence $\Spec(\bR \otimes_\bS \bS[\bB^{-1}]) \subset \QSpec \bR$.
\end{enumerate}
Lemma \ref{AffineCover} says that $\Spec \bR \otimes_{\Spec \bS} \QSpec \bS$ is the union over relevant multiplicative subsets $\bB$ of affine opens $\Spec(\bR \otimes_\bS \bS[\bB^{-1}])$. Therefore it is contained inside $\QSpec \bR$ as desired.

Now we prove \ref{SurjectRelevant2}. Let $\bA \subset \bR$ be a relevant multiplicative subset and let $a \in \bA$ be a homogeneous element of strictly dominant weight. By assumption we can write $a^n = \phi(s)$ for some $n > 0$ and $s \in \bJ \subset \bS$. Therefore $\bR[\bA^{-1}] \otimes_\bS \bS/\bJ = 0$, because $1 = a^{-n} \otimes s = 0$. Hence $\QSpec \bR \rightarrow \Spec \bS$ factors through $\QSpec \bS$.
\end{proof}

\subsection{Serre vanishing}
We will need the following multi-graded version of Serre vanishing. It says that, in sufficiently dominant weight spaces, restriction away from the irrelevant boundary does not affect global sections nor introduce higher cohomology.

\begin{lemma}\label{MultiSerre}
Assume that $\bR$ is finitely generated. Let $\bM$ be a finitely generated $\Lambda$-graded $\bR$-module. If $\lambda$ is sufficiently dominant then
\begin{enumerate}[label=(\alph*)]
\item\label{MultiSerre1} $H^i(\QSpec \bR, j^* \bM)_{\lambda} \simeq 0$ for $i > 0$,
\item\label{MultiSerre2} $\bM_{\lambda} \rightarrow H^0(\QSpec \bR, j^* \bM)_{\lambda}$ is an isomorphism.
\end{enumerate}
\end{lemma}
\begin{proof}
Choose a surjection $\bS \twoheadrightarrow \bR$ from a finitely generated $\Lambda^+$-graded polynomial algebra $\bS$. Lemma \ref{SurjectRelevant} implies that $\QSpec \bR = \Spec \bR \times_{\Spec \bS} \QSpec \bS$ is a closed subscheme of $\QSpec \bS$. Therefore we may replace $\bR$ by $\bS$, and assume that $\bR$ is a polynomial algebra.

First we prove \ref{MultiSerre1}.
Note that $\QSpec \bR$ is a finite union of affine varieties. Hence there exists $n$, such that the cohomology of every coherent sheaf on $\QSpec \bR$ is concentrated in degree $\leq n$.
Therefore we may argue by descending induction on $i$.

There exists a free module $\bF$ fitting into a short exact sequence of finitely generated $\bR$-modules \beq\label{FreeSES} 0 \rightarrow \bN \rightarrow  \bF \rightarrow \bM \rightarrow 0.\eeq
Restricting to $\QSpec \bR$ and taking the long exact sequence in cohomology gives \[\dots \rightarrow H^i(\QSpec \bR, j^* \bF) \rightarrow H^i(\QSpec \bR, j^* \bM) \rightarrow H^{i+1}(\QSpec \bR, j^* \bN) \rightarrow \dots.\]
If $i > 0$ and $\lambda$ is sufficiently dominant, then the usual calculation of the cohomology of line bundles on projective space shows that $H^i(\QSpec \bR, \bF)_{\lambda} \simeq 0$. Therefore the inductive hypothesis implies for $\lambda$ sufficiently dominant the desired $H^i(\QSpec \bR, j^* \bM)_{\lambda} \simeq 0$.

Now we prove surjectivity of \ref{MultiSerre2}. Choose a resolution of $\bM$ as in \eqref{FreeSES}. If $\lambda$ is sufficiently dominant then $H^1(\QSpec \bR, j^* \bN)_{\lambda} \simeq 0$. Therefore, by commutativity of
\[\begin{tikzcd}
\bF_{\lambda} \arrow[r, two heads] \arrow[d, "\sim"'] & \bM_{\lambda} \arrow[d, "\alpha"] \\
H^0(\QSpec \bR, \bF)_{\lambda} \arrow[r, two heads] & H^0(\QSpec \bR, \bM)_{\lambda},
\end{tikzcd}\]
the map $\alpha$ is surjective.

Finally we prove injectivity of \ref{MultiSerre2}. Choosing a resolution of $\bM$ as in \eqref{FreeSES}, we get the following diagram with exact rows
\[\begin{tikzcd}
\bN_{\lambda} \arrow[r] \arrow[d, two heads, "\beta"'] & \bF_{\lambda} \arrow[r, two heads] \arrow[d, "\sim"'] & \bM_{\lambda} \arrow[d, "\alpha"]\\
H^0(\QSpec \bR, \bN)_{\lambda} \arrow[r] & H^0(\QSpec \bR, \bF)_{\lambda} \arrow[r] & H^0(\QSpec \bR, \bM)_{\lambda}.
\end{tikzcd}\]
If $\lambda$ is sufficiently dominant, we have already proved that $\beta$ is surjective. A diagram chase shows that $\alpha$ is injective.
\end{proof}

\subsection{Perfect complexes}
Let $\bR$ be a finitely generated $\Lambda^+$-graded ring acted on by $\sfG$. Let $\Lambda^{+++}$ be a submonoid that contains all sufficiently dominant coweights. Let $\bS$ be the direct sum of $\bR_\lambda$ over $\lambda \in \Lambda^{+++}$.

Let $\Perf_{\sfG \times \sfT}^{\ind}(\bS)$ be the full subcategory of $\Perf_{\sfG \times \sfT}(\bS)$ generated under retracts and finite colimits by pullbacks of finite dimensional representations of $\sfG \times \sfT$.

Let $\Bdry \subset \Perf_{\sfG \times \sfT}^{\ind}(\bS)$ be the full subcategory of complexes killed by restriction to the relevant open, i.e. they are supported on the irrelevant boundary.

We now prove that the corresponding quotient category embeds fully faithfully into perfect complexes on the relevant open.
The idea of the proof is to combine the following observations.
\begin{enumerate}
\item[-] Every complex of equivariant $\bS$-modules is equivalent modulo a complex supported on the irrelevant boundary to a complex of free modules twisted by very anti-dominant weights.
\item[-] Let $\bL, \bM$ be complexes of $\bS$-modules. If $\mu$ is sufficiently dominant, a multi-graded version of Serre vanishing implies that $\RHom(\bL(-\mu), \bM)$ is unchanged by restriction to the relevant open subscheme.
\end{enumerate}

\begin{proposition}\label{TTLocalization}
Assume that
\begin{enumerate}[label=(\alph*)]
\item\label{TTLocalization1}  for all $\lambda$ sufficiently dominant, there is a finite dimensional submodule $V_{\lambda} \subset \bR_{\lambda}$ such that $\bR/V_{\lambda}\bR$ is supported on the irrelevant boundary.
\end{enumerate}
Then the functor \[\Perf^{\ind}_{\sfG \times \sfT}(\bS)/\Bdry \rightarrow \Perf_{\sfG \times \sfT}(\QSpec \bS)\] is fully faithful.
\end{proposition}
\begin{proof}
Let $d$ be the dimension of $V_{\lambda}$ and set \[\bK_{\lambda} \coloneqq \Lambda^d V_{\lambda} \otimes \bS(-d\lambda) \rightarrow \dots \Lambda^2 V_{\lambda} \otimes \bS(-2\lambda) \rightarrow V_{\lambda} \otimes \bS(-\lambda) \; \in \; \Perf^{\ind}_{\sfG \times \sfT}(\bS).\] The restriction $j^*(\cone(V_{\lambda} \otimes \bS(-\lambda) \rightarrow \bS))$ is quasi-isomorphic to a vector bundle of rank $d-1$ on the relevant open $\QSpec \bS$ concentrated in degree -1. Therefore \[\cone(\bK_{\lambda} \rightarrow \bS) \simeq \Sym^d(\cone(V_{\lambda} \otimes \bS(-\lambda) \rightarrow \bS)) \; \in \; \Bdry\]
is supported on the irrelevant boundary.

Let $\bL, \bM \in \Perf^{\ind}_{\sfG \times \sfT}(\bS)$. For all $\lambda \in \Lambda^+$, the cone of $\alpha: \bK_{\lambda} \otimes_{\bS} \bL \rightarrow \bL$ is also supported on the irrelevant boundary.

Lemma \ref{HomRestrictLemma} implies that for $\lambda$ sufficiently dominant, \[\RHom_{\Perf_{\sfG \times \sfT}(\bS)}(\Lambda^i V_{\lambda} \otimes \bL (-i\lambda), \bM) \xrightarrow{\sim} \RHom_{\Perf_{\sfG \times \sfT}(\QSpec \bS)}(j^*(\Lambda^i V_{\lambda} \otimes \bL (-i\lambda)), j^*\bM)\]
is an isomorphism for all $1 \leq i \leq d$.

Thus for $\lambda$ sufficiently dominant, \[\RHom_{\Perf_{\sfG \times \sfT}(\bS)}(\bK_{\lambda} \otimes_{\bS} \bL, \bM)
\xrightarrow{\sim} \RHom_{\Perf_{\sfG \times \sfT}(\QSpec \bS)}(j^*(\bK_{\lambda} \otimes_{\bS} \bL), j^*\bM)\]
is an isomorphism.

Therefore Lemma \ref{GeneralLocalization} implies the desired fully faithfulness.
\end{proof}

We used the following lemma that restricting to the relevant open subscheme does not affect sufficiently dominant weight spaces of Hom spaces between perfect complexes.

\begin{lemma}\label{HomRestrictLemma}
Let $\bL, \bM \in \Perf_{\sfG \times \sfT}^{\ind}(\bS)$. Let $j: \QSpec \bS \rightarrow \Spec \bS$ be the inclusion of the relevant open subscheme.

If $\mu$ is sufficiently dominant then \beq \label{HomRestrict}\RHom_{\Perf_{\sfT}(\bS)}(\bL(-\mu), \bM) \xrightarrow{\sim} \RHom_{\Perf_{\sfT}(\QSpec \bS)}(j^* \bL(-\mu), j^* \bM)\eeq
is an isomorphism.
\end{lemma}
\begin{proof}
Lemma \ref{QSpecInclusion} says that $\QSpec \bR \simeq \QSpec \bS$.

For $\mu$ sufficiently dominant, Lemma \ref{MultiSerre} implies that \[\RHom_{\Perf_{\sfT}(\bR)}(\bL (-\mu) \otimes_\bS \bR, \bM \otimes_\bS \bR) \xrightarrow{\sim} \RHom_{\Perf_{\sfT}(\QSpec \bR)}(j^*\bL(-\mu), j^*\bM)\]
is an isomorphism.

(The reason that we could not apply Lemma \ref{MultiSerre} directly for $\bS$-modules is because $\bS$ is typically not finitely generated.)

Moreover if $\mu$ is sufficiently dominant then \[\RHom_{\Perf_{\sfT}(\bS)}(\bL(-\mu), \bM) \xrightarrow{\sim} \RHom_{\Perf_{\sfT}(\bR)}(\bL (-\mu) \otimes_{\bS} \bR, \bM \otimes_{\bS} \bR))\]
is an isomorphism.

Therefore, commutativity of
\[\begin{tikzcd}
\RHom_{\Perf_{\sfT}(\bS)}(\bL (-\mu), \bM) \arrow[r, "\eqref{HomRestrict}"] \arrow[d, "\sim"'] & \RHom_{\Perf_{\sfT}(\QSpec \bS)}(j^* \bL (-\mu), j^* \bM) \arrow[d, "\sim"]\\
\RHom_{\Perf_{\sfT}(\bR)}(\bL (-\mu) \otimes_\bS \bR, \bM \otimes_\bS \bR) \arrow[r, "\sim"'] & \RHom_{\Perf_{\sfT}(\QSpec \bR)}(j^*\bL(-\mu), j^*\bM)
\end{tikzcd}\]
implies that \eqref{HomRestrict} is an isomorphism as desired.
\end{proof}

We also used the following lemma about quotient categories to prove fully faithfulness.

\begin{lemma}\label{GeneralLocalization}
Let $\cC$ be a small stable DG-category. Let $\cB \subset \cC$ be a subcategory closed under finite colimits and direct summands. Let $j^*: \cC \rightarrow \cD$ be a functor that vanishes on $\cB$.

Let $C_i \in \cC$ be a collection of objects indexed by $i \in I$ that generate $\cC$ under retracts and finite colimits. Assume that for all $C \in \cC$ and $i \in I$, there exists a map $\alpha: C' \rightarrow C$ such that $\cone(\alpha) \in \cB$ and $\Hom_\cC(C', C_i) \simeq \Hom_\cD(j^* C', j^* C_i)$.

Then the induced functor $j^*: \cC/\cB \rightarrow \cD$ is fully faithful.
\end{lemma}
\begin{proof}
Since $C_i$ generate $\cC$ under retracts and finite colimits, it suffices to prove that for all $i, j \in I$ and integers $n$, the map \beq\label{JStarFullyFaithful} j^*: \Hom_{\cC/\cB}^n(C_i, C_j) \rightarrow \Hom_{\cD}^n (j^*C_i, j^*C_j) \quad \text{is an isomorphism.}\eeq

First we prove surjectivity. Choose $\alpha:C_i' \rightarrow C_i$ such that $\cone(\alpha) \in \cB$ and $\Hom_\cC(C_i', C_j) \simeq \Hom_\cD(j^* C_i', C_j)$. Then surjectivity follows by the commutative diagram
\[\begin{tikzcd}
\Hom_{\cC/\cB}(C_i, C_j) \arrow[r] & \Hom_\cD (j^*C_i, j^*C_j) \\
\Hom_\cC(C_i', C_j) \arrow[r, "\sim"'] \arrow[u] & \Hom_\cD (j^*C_i', j^*C_j). \arrow[u, "\sim"']
\end{tikzcd}\]

Now we prove injectivity. Any $\beta \in \Hom_{\cC/\cB}^n(C_i, C_j)$ can be represented by \[C_i \xleftarrow{\gamma} C \xrightarrow{\delta} C_j[n],\]
where $\delta \in \Hom_{\cC}^n(C, C_j)$ and $\gamma \in \Hom_{\cC}^0(C, C_i)$ satisfies $\cone(\gamma) \in \cB$. Suppose that $j^* \beta = 0$, or equivalently $j^* \delta = 0$. Choose $\alpha:C' \rightarrow C$ such that $\cone(\alpha) \in \cB$ and $\Hom_\cC(C', C_j) \simeq \Hom_\cD(j^* C', j^*C_j)$.
Then $j^* \delta \alpha = 0 \in \Hom_{\cD}^n(j^*C', j^*C_j)$ implies that $\delta \alpha = 0 \in \Hom_{\cC}^n(C', C_j)$. Since $\alpha$ becomes an isomorphism in the quotient category $\cC/\cB$, we get $\delta = 0 \in \Hom_{\cC/\cB}^n(C, C_j)$. This proves injectivity.
\end{proof}

\section{Construction of the functor II  (factorization through the relevant open)}\label{ConstructionFunctorII}
In Section \ref{ConstructionFunctorI} already constructed a functor out of free modules over a certain multi-graded ring. Here we will factor this functor through the relevant open subscheme to obtain the Arkhipov--Bezrukavnikov functor \[\IndPerf_\sfG(\widetilde{\sfG}) \rightarrow \cH.\]

We choose a submonoid of dominant weights such that corresponding subring of the multi-graded ring admits a dual Weyl filtration. The existence of this filtration implies that the category of perfect complexes for the subring can be recovered in terms of complexes of free modules. By passing to bounded homotopy categories and showing that complexes supported on the irrelevant boundary are killed we obtain the desired Arkhipov--Bezrukavnikov functor.

\subsection{Overview of the argument}
We seek to construct a functor out of $\Perf_{\sfG \times \sfT}(\sfV)$, where $\sfV \coloneqq \sfG \times^{\sfU} \sfB$ is the natural $\sfT$-torsor over $\widetilde{\sfG}$.
Recall that $\sfV$ is an open subscheme of $\sfX$, where $\sfX$ is a certain closed subscheme of $\sfG \times (\sfG/\sfU)^{\aff} \times \sfT$ cut out by explicit equations.

The scheme $\sfX$ is the spectrum of a multi-graded ring. Let $\bI \subset \cO(\sfX)$ be the ideal of strictly positively graded elements. We call $\Spec \cO(\sfX)/\bI$ the `irrelevant boundary'. Then $\sfV$ is the relevant open subscheme of $\sfX$ obtained by deleting this irrelevant boundary.

Using highest weight arrows and monodromy of nearby cycles, we have already constructed a functor out of $\Free^{\tilt}_{\sfG \times \sfT}(\sfX)$. We would like to obtain a functor out of $\Perf_{\sfG \times \sfT}(\sfV)$ by passing to bounded homotopy categories and showing that complexes supported on the irrelevant boundary are killed.
But, working over the integers, we face the following obstacle.
\begin{enumerate}
\item[-] Although $\sfX$ is affine, its ring of functions probably does not admit a dual Weyl filtration. Therefore $\Perf_{\sfG \times \sfT}^{\ind}(\sfX)$ is probably not the bounded homotopy category of $\Free^{\tilt}_{\sfG \times \sfT}(\sfX)$.
\end{enumerate}
We overcome this obstacle by restricting to sufficiently dominant weight spaces and using the following two facts.
\begin{enumerate}
\item[-] For all \textit{sufficiently dominant} $\lambda$, the weight space $\cO(\sfX)_{\lambda}$ is isomorphic to $H^0(\widetilde{\sfG}, \cO(\lambda))$, by a multi-graded version of Serre vanishing.
\item[-] For all weights $\lambda$, the $\sfG$-module $H^0(\widetilde{\sfG}_{\bF_p}, \cO(\lambda))$ admits a dual Weyl filtration, because $\widetilde{\sfG}_{\bF_p}$ admits a canonical Frobenius splitting.
\end{enumerate}
Therefore, for all \textit{sufficiently dominant} $\lambda$, the weight space $\cO(\sfX)_{\lambda}$ admits a dual Weyl filtration (at least after reducing modulo each prime).

We define a subring $\bS \subset \cO(\sfX)$ by simply discarding the offending lower weight spaces. By a general result about multi-graded rings, this does not change the relevant open subscheme.

Thus we can obtain the desired functor out of $\Perf_{\sfG \times \sfT}(\sfV)$ by the following procedure.
\begin{enumerate}
\item[-] Restrict the already constructed functor out of $\Free_{\sfG \times \sfT}(\sfX)$ to a functor out of $\Free^{\tilt}_{\sfG \times \sfT}(\bS)$.
\item[-] Because $\bS$ admits a dual Weyl filtration, the bounded homotopy category of $\Free^{\tilt}_{\sfG \times \sfT}(\bS)$ is a certain full subcategory $\Perf_{\sfG \times \sfT}^{\ind}(\bS) \subset \Perf_{\sfG \times \sfT}(\bS)$.
\item[-] Pass to bounded homotopy categories and prove that the resulting functor kills complexes supported on the irrelevant boundary. 
\item[-] A general result about multi-graded rings implies that $\Perf_{\sfG \times \sfT}^{\ind}(\sfV)$ is the quotient of $\Perf_{\sfG \times \sfT}^{\ind}(\bS)$ by the subcategory of complexes supported on the irrelevant boundary.
\item[-] Show that $\Perf_{\sfG \times \sfT}^{\ind}(\sfV)$ is the entire $\Perf_{\sfG \times \sfT}(\sfV)$.
\end{enumerate}

\subsection{The Grothendieck--Springer scheme}
Here we prove that $\sfV$ is the relevant open subscheme of $\sfX$, i.e. we realize the Grothendieck--Springer scheme as a multi-graded proj-scheme.

Recall that $\cO(\sfG/\sfU)_{\lambda} \simeq M_{\lambda}$ is the dual Weyl module of highest weight $\lambda \in \Lambda^+$. Recall also the following subschemes of $\sfG \times (\sfG/\sfU)^{\aff} \times \sfT$ that were defined in Section \ref{BaseAffine}.
\begin{enumerate}
\item[-] Let $\sfX \coloneqq \{(g, x, t) \in \sfG \times (\sfG/\sfU)^{\aff} \times \sfT \text{ such that } gx = xt\}$, the spectrum of a $\Lambda^+$-graded ring.
\item[-] Let $\sfV \coloneqq \sfX \cap (\sfG \times \sfG/\sfU \times \sfT)$, a $\sfT$-torsor over the Grothendieck--Springer scheme.
\end{enumerate}
The following lemma says that $\sfV$ is the relevant open subscheme of $\sfX$. The proof is similar to Lemma 6.2.2 of \cite{AR}.

\begin{lemma}\label{XCostandard}
Assume that $G$ is not a torus. Then there are isomorphisms
\begin{enumerate}[label=(\alph*)]
\item\label{XCostandard1} $\QSpec \cO(\sfG/\sfU) \simeq \sfG/\sfU$.
\item\label{XCostandard2} $\QSpec \cO(\sfX) = \sfV$.
\end{enumerate}
\end{lemma}
\begin{proof}
It is clear that \ref{XCostandard1} implies \ref{XCostandard2}. Now we prove \ref{XCostandard1}. Restriction to the torus factors through \[i^*: \cO(\sfG/\sfU) \rightarrow \bZ[\Lambda^+] \rightarrow R.\] Let $\bI \subset \cO(\sfG/\sfU)$ be the ideal of strictly positively graded elements. The assumption that $G$ is not a torus implies that there exists $x \in \bI$ such that $i^* x \in R$ is a unit. Therefore $R \otimes_{\cO(\sfG/\sfU)} \cO(\sfG/\sfU)/\bI \simeq 0$, so the map $i: \sfT \rightarrow (\sfG/\sfU)^{\aff}$ factors through the relevant open subscheme.

Now set $\overline{\sfT} \coloneqq \Spec \bZ[\Lambda^+]$.
Lemma \ref{SurjectRelevant} gives a Cartesian square
\beq\label{CompactifyTorus}\begin{tikzcd}
\sfT \arrow[r] \arrow[d] & \QSpec \cO(\sfG/\sfU) \arrow[d] \\
\overline{\sfT} \arrow[r] & (\sfG/\sfU)^{\aff}.
\end{tikzcd}\eeq
Therefore we get a commutative diagram
\beq\label{InduceCompactTorus} \begin{tikzcd}
\sfG \times^{\sfB} \sfT \arrow[r, "c"] \arrow[d] & \sfG \times^{\sfB} \QSpec \cO(\sfG/\sfU) \arrow[r, "b"] \arrow[d] & \QSpec \cO(\sfG/\sfU) \arrow[d, "a"] \\
\sfG \times^{\sfB} \overline{\sfT} \arrow[r] & \sfG \times^{\sfB} (\sfG/\sfU)^{\aff} \arrow[r] & (\sfG/\sfU)^{\aff},
\end{tikzcd}\eeq
and the left square is Cartesian.

Since $\sfG/\sfU$ is quasi-affine (choose a representation of $\sfG$ and a vector whose stabilizer is $\sfU$), the map $abc: \sfG \times^{\sfB} \sfT \rightarrow (\sfG/\sfU)^{\aff}$ is an open immersion. Since $a$ is an open immersion, this implies that $cb:\sfG \times^{\sfB} \sfT \rightarrow \QSpec \cO(\sfG/\sfU)$ is also an open immersion.

The map $\overline{\sfT} \rightarrow (\sfG/\sfU)^{\aff}$ is proper (it is a closed embedding if $\sfG$ has simply connected derived subgroup). Therefore $c$ is proper, because the left square in \eqref{InduceCompactTorus} is Cartesian. Therefore $cb: \sfG \times^{\sfB} \sfT \rightarrow \QSpec \cO(\sfG/\sfU)$ is also proper, hence it is an isomorphism.
\end{proof}

The following lemma shows that $\cO(\sfX)$ satisfies condition \ref{TTLocalization1} of Lemma \ref{TTLocalization}.
\begin{lemma}\label{QuotientBoundary}
Let $\bJ_{\mu} \subset \cO(\sfX)$ be the ideal generated by $M_{\mu} \boxtimes \bZ_{-\mu} \subset \cO(\sfG/\sfU)$.
Then for all $\mu \in \Lambda^+$, the quotient of $\cO(\sfX)/\bJ_{\mu}$ is supported on the irrelevant boundary.
\end{lemma}
\begin{proof}
For all $\lambda, \mu \in \Lambda^+$ the map $M_{\lambda} \otimes M_{\mu} \rightarrow M_{\lambda + \mu}$ is surjective. Therefore $\bJ_{\mu} \subset \cO(\sfX)$ contains $M_{\lambda} \boxtimes \bZ_{-\lambda}$ for all $\lambda \geq \mu$.

Let $\bI \subset \cO(\sfG/\sfU)$ be the direct sum of the weight spaces indexed by strictly dominant weights. For $n$ sufficiently large, we have $\bI^n \subset \bJ_{\mu}$. Therefore the quotient $\cO(\sfX)/\bJ_{\mu}$ is supported on the irrelevant boundary.
\end{proof}

\subsection{Generation by vector bundles}
The inclusion $\Perf^{\ind} \subset \Perf$ need not be an equality in general. Indeed it is possible that an equivariant module admits a finite free resolution after forgetting the equivariance, but not a finite equivariant resolution.

\begin{lemma}\label{PerfGeneration}
The category $\Perf_{\sfB}(\sfB)$ is generated under retracts and finite colimits by pullbacks from $\Rep(\sfB)$, i.e. the inclusion $\Perf_B^{\ind}(\sfB) \subset \Perf_{\sfB}(\sfB)$ is an equivalence.
\end{lemma}
\begin{proof}
Let $\cO_{1/\sfB} \in \Perf_{\sfB}(\sfB)$ be the skyscraper sheaf at the identity. There is a sequence of closed $\sfB$-equivariant immersions $1 \subset \dots \subset  \sfB_2 \subset \sfB_1 \subset \sfB$ such that $\sfB_i \subset \sfB_{i-1}$ is cut out by a $\sfB$-equivariant polynomial. Therefore
\begin{enumerate}
\item[($\star$)] $\cO_{1/\sfB}$ admits a finite $\sfB$-equivariant resolution by free modules.
\end{enumerate}

Now we will use the argument from Proposition VIII.5.11 of \cite{FS}. Let $\sfB$ act on $\sfB \times \sfB$ by the adjoint action on both factors. Define maps \[p, m: (\sfB \times \sfB)/\sfB \rightarrow \sfB/\sfB, \quad p(b, b') = b, \quad q(b, b') = b', \quad m(b, b') = b(b')^{-1}.\]
Then there is a Cartesian square \[\begin{tikzcd} (\sfB \times \sfB)/\sfB \arrow[r, "q"] \arrow[d, "p"'] & \sfB/\sfB \arrow[d, "r"] \\ \sfB/\sfB \arrow[r, "r"'] & \pt/\sfB. \end{tikzcd}\]

If $E \in \Perf_{\sfB}(\sfB)$ then \beq\label{ConvolveIdentity} E \simeq q_* (p^*E \otimes \Delta_* \cO_{\sfB/\sfB}) \; \in \; \QCoh_{\sfB}(\sfB).\eeq
Moreover ($\star$) implies that $\Delta_* \cO_{\sfB/\sfB} \simeq m^* \cO_{1/\sfB}$ is a finite colimit of free modules pulled back from $\pt/\sfB$. Therefore \eqref{ConvolveIdentity} is a colimit of vector bundles of the form $q_*p^*E \simeq (r_* E) \otimes \cO_{\sfB/\sfB}$ tensored with representations of $\sfB$.
Hence $E$ is a colimit of vector bundles that are pulled back from $\pt/\sfB$ and concentrated in bounded below cohomological degrees. Since $\Hom(E, -)$ commutes with uniformly bounded below filtered colimits in $\QCoh_{\sfB}(\sfB)$, it follows that $E$ is a \textit{finite} colimit of vector bundles pulled back from $\pt/\sfB$.
Therefore $E \in \Perf_{\sfB}^{\ind}(\sfB)$ as desired.
\end{proof}

\subsection{Tilting modules}
For $\lambda \in \Lambda^+$, let
\begin{enumerate}
\item[-] $M_{\lambda} \coloneqq H^0(\sfG/\sfB, \cO(\lambda)) \in \Rep(\sfG)$ be the dual Weyl module of highest weight $\lambda$,
\item[-] $W_{\lambda} \coloneqq H^0(\sfG/\sfB, \cO(-w_0(\lambda)))^* \in \Rep(\sfG)$ be the Weyl module of highest weight $\lambda$.
\end{enumerate}
A $\sfG$-module is called tilting if it admits both Weyl and dual Weyl filtrations.

Tilting modules over the integers satisfy the following properties.

\begin{lemma}\label{TiltingModules}
Let $M, N \in \Rep(\sfG)$.
\begin{enumerate}[label=(\alph*)]
\item\label{TiltingModules1} If $M$ admits a Weyl filtration and $N$ admits a dual Weyl filtration, then $\Ext^i(M, N) \simeq 0$ for all $i > 0$.
\item\label{TiltingModules2} For each $\lambda \in \Lambda^+$, there exists an indecomposable tilting module $T_{\lambda}$ of highest weight $\lambda$.
\item\label{TiltingModules3} If $M$ and $N$ both admit dual Weyl filtrations, then so does $M \otimes N$.
\item\label{TiltingModules4} If $M$ and $N$ are both tilting, then so is $M \otimes N$.
\end{enumerate}
\end{lemma}
\begin{proof}
Lemma B.9 of \cite{Jan} implies \ref{TiltingModules1}.

Now we prove \ref{TiltingModules2}. Choose an ordering $\lambda_0, \dots \lambda_n = \lambda$ on $\Lambda^+_{\leq \lambda}$ refining the usual partial order.
We claim that for each $i$ there exists $T(\lambda)^{\geq i} \in \Rep(\sfG)$ of highest weight $\lambda$ such that
\begin{enumerate}
\item[-] $T(\lambda)^{\geq i}$ admits a Weyl filtration,
\item[-]  $\Ext^1(W(\lambda_j), T(\lambda)^{\geq i}) \simeq 0$ for all $i \leq j$.
\end{enumerate}
The proof is by downward induction on $i$. In the base case $T(\lambda)^{\geq n} \coloneqq W(\lambda)$.
For the inductive step, choose a surjection $E \twoheadrightarrow \Ext^1(W(\lambda_i), T(\lambda)^{\geq i+1})$ from a free $\bZ$-module $E$. Let $T(\lambda)^{\geq i}$ be the natural extension \[0 \rightarrow T(\lambda)^{\geq  i +1} \rightarrow T(\lambda)^{\geq i}  \rightarrow W(\lambda_i) \otimes E \rightarrow 0.\] Taking $\Hom(W(\lambda_i), -)$ gives an exact sequence
\begin{align*}\Hom^0(W(\lambda_i), W(\lambda_i) \otimes E) \twoheadrightarrow \Ext^1(W(\lambda_i), T(\lambda)^{\geq i + 1}) &\rightarrow \Ext^1(W(\lambda_i), T(\lambda)^{\geq i}) \\ &\rightarrow \Ext^1(W(\lambda_i), W(\lambda_i) \otimes E).\end{align*}
Proposition B.4 of \cite{Jan} implies that $\Ext^1(W(\lambda_i), W(\lambda_i) \otimes E) \simeq 0$. Therefore also $\Ext^1(W(\lambda_i), T(\lambda)^{\geq i}) \simeq 0$.

By construction $T(\lambda) \coloneqq T(\lambda)^{\geq 0}$ admits a Weyl filtration. Lemma B.9 of \cite{Jan} implies that it also admits a dual Weyl filtration.

Now we prove \ref{TiltingModules3}. Suppose that $M, N \in \Rep(\sfG)$ both admit dual Weyl filtrations. Theorem 1 of \cite{Mat} implies that $M \otimes N \otimes \bF_p$ admits a dual Weyl filtration for each prime $p$. Therefore Lemma B.9 of \cite{Jan} implies that $M \otimes N$ admits a dual Weyl filtration. Using duality a similar result holds for Weyl filtrations, which implies \ref{TiltingModules4}.
\end{proof}

\subsection{Factorization through the relevant open}
In Proposition \ref{FunctorClosed} we already constructed a functor fitting into a commutative diagram
\beq\begin{tikzcd}\label{FactorizationClosedTriangle}
\Free_{\sfG \times \sfT}(\sfX) \arrow[r, "F"] \arrow[dr, "i^*"'] &  \Perv_{(I)}^{\waki}(\Fl)\arrow[d, "\gr"] \\
& \Free_{\sfT}(\sfT).
\end{tikzcd}\eeq

To factor this functor through the relevant open we will discard some lower weight spaces of $\cO(\sfX)$. Lemma \ref{MultiSerre} implies that there exists a submonoid $\Lambda^{+++} \subset \Lambda^+$ such that
\begin{enumerate}[label=(\roman*)]
\item $\Lambda^{+++}$ contains all sufficiently dominant coweights and all coweights that are orthogonal to all roots,
\item\label{XTilde} if $\lambda \in \Lambda^{+++}$ is strictly dominant, then $\cO(\sfX)_{\lambda} \simeq H^0(\widetilde{\sfG}, \cO(\lambda))$.
\end{enumerate}
Let $\bS$ be the direct sum over $\lambda \in \Lambda^{+++}$ of $\cO(\sfX)_{\lambda}$.

Let $\Free^{\tilt}_\sfG(\bS)$ be the additive category of free $\bS$-modules that are pulled back from finite dimensional tilting representations of $\sfG$. (This category is closed under tensor product by Lemma \ref{TiltingModules}\ref{TiltingModules4}.)

Having passed to a submonoid satisfying \ref{XTilde}, the following lemma says that the bounded homotopy category of $\Free^{\tilt}_\sfG(\bS)$ maps fully faithfully to perfect complexes.

\begin{lemma}\label{KbS}
There is an equivalence $\K^{\bnd}(\Free_{\sfG \times \sfT}^{\tilt}(\bS)) \simeq \Perf_{\sfG \times \sfT}^{\ind}(\bS)$.
\end{lemma}
\begin{proof}
Therefore direct sum of $\bS_{\lambda}$ over all $\lambda \in \Lambda$ that are orthogonal to all roots is the ring $\cO(\sfG \times \sfT)$, which admits a costandard filtration.
Combined with \ref{XTilde} and Proposition \ref{SVan}\ref{SVan2}, this implies the desired equivalence.
\end{proof}

The following proposition gives the desired factorization through the relevant open subscheme.

\begin{proposition}\label{FunctorOpen}
The central and Wakimoto functor extends to an $R$-linear monoidal functor fitting into a commutative diagram
\[\begin{tikzcd}
\Perf_\sfG(\widetilde{\sfG}) \arrow[r, "F"]  & \cH^{\cpt}  \\
\Rep(\sfG \times \sfT). \arrow[u] \arrow[ur, "Z_W"']  &
\end{tikzcd}\]
\end{proposition}
\begin{proof}
Restricting \eqref{ClosedFunct} along $\bS \rightarrow \cO(\sfX)$ gives a monoidal functor fitting into a commutative triangle
\[\begin{tikzcd}
\Free_{\sfG \times \sfT}(\bS) \arrow[r, "F"] \arrow[dr, "i^*"'] & \Perv_{(I)}^{\waki}(\Fl)\arrow[d, "\gr"] \\
\Rep(\sfG \times \sfT) \arrow[r] \arrow[u] & \Free_{\sfT}(\sfT). 
\end{tikzcd}\]
Lemma \ref{KbS} gives functors
\beq \label{PassHomotopy} \begin{tikzcd}
\Perf_{\sfG \times \sfT}^{\ind}(\bS) \simeq \K^{\bnd}(\Free^{\tilt}_{\sfG \times \sfT}(\bS)) \arrow[r, "\K(F)"] \arrow[dr, "\K (i^*)"']& \K^{\bnd}(\Perv_{(I)}^{\waki}(\Fl))\arrow[d, "\K (\gr)"] \arrow[r]  & \cH^{\cpt}\\
& \Perf_{\sfT}(\sfT). & 
\end{tikzcd}\eeq

Now we assume that $G$ is not a torus (otherwise there is nothing to prove). Suppose that $M \in \Perf_{\sfG \times \sfT}^{\ind}(\bS)$ is supported on the irrelevant boundary. Then $\K(\gr)\K(F) M \simeq \K(i^*) M \simeq 0$ vanishes because the image of $i$ factors through the relevant open $\QSpec \bS$. Therefore the image of $M$ vanishes in $\cH^{\cpt}$ (if an object admits a finite filtration whose associated graded vanishes then the object is isomorphic to zero). Thus \eqref{PassHomotopy} factors through
\[\begin{tikzcd}
\Perf_{\sfG \times \sfT}^{\ind}(\bS) \arrow[r] \arrow[dr]& \cH^{\cpt} \\
& \Perf_{\sfG \times \sfT}^{\ind}(\bS)/\Bdry. \arrow[u]&  
\end{tikzcd}\]

Lemma \ref{QuotientBoundary} implies that $\bS$ satsifies condition \ref{TTLocalization1} of Proposition \ref{TTLocalization}. Therefore the functor \[\Perf_{\sfG \times \sfT}^{\ind}(\bS)/\Bdry \rightarrow \Perf_{\sfG \times \sfT}(\QSpec \cO(\sfX)) \overset{\text{Lemma } \ref{QSpecInclusion}}{\simeq} \Perf_{\sfG \times \sfT}(\sfV)\] is fully faithful. Moreover it is essentially surjective by Lemma \ref{PerfGeneration} and Lemma \ref{TiltingModules}\ref{TiltingModules2}.
This gives the desired functor \[\Perf_\sfG(\widetilde{\sfG}) \simeq \Perf_{\sfG \times \sfT}(\sfV) \rightarrow \cH^{\cpt}. \qedhere\]
\end{proof}

{\large \part{Proof of the Whittaker equivalence}\label{Part2}}

In this second part of the paper we prove that $F: \IndPerf_{\sfG}(\widetilde{\sfG}) \rightarrow \cH$ becomes an equivalence after Whittaker averaging. The proof is by localizing away from certain intersections of walls, to reduce to an order of vanishing calculation in semi-simple rank 1.

Let $M_{\lambda} \in \Rep(\sfG)$ the costandard representation of highest weight $\lambda$. To prove fully faithfulness it suffices to check for all $\lambda, \mu \in \Lambda^+$, that the following composition is an isomorphism \beq \label{AvFIso}\RHom_{\sfB/\sfB}(\cO, M_{\lambda} \otimes \cO(\mu)) \rightarrow \RHom(\Delta_1, Z(M_{\lambda}) * W_{\mu}) \rightarrow  \RHom({}^{\chi}\nabla_1, {}^{\chi}\nabla_1 * Z(M_{\lambda}) * W_{\mu}).\eeq
Using that the Whittaker averaged central sheaves admit costandard filtrations, both sides are free $R$-modules concentrated in degree zero. Moreover the second map induces an isomorphism in degree zero.
Therefore it suffices to prove that the following is an isomorphism \beq \label{FullyFaithfulCheck} \Hom_{\sfB/\sfB}(\cO, M_{\lambda} \otimes \cO(\mu)) \rightarrow \Hom(\Delta_1, Z(M_{\lambda}) * W_{\mu}).\eeq

Using faithfulness of the associated graded functor, we compatibly embed both sides of \eqref{FullyFaithfulCheck} as free submodules of the same free $R$-module. To check that these two submodules coincide, we characterize both in terms of certain order of vanishing conditions along the walls.

By Hartogs' lemma, it suffices to prove that the submodules coincide after the following two different types of localization.

The first type of localization reduces us to semi-simple rank 1 in characteristic 0. After localizing in $\sfT$ away from all but one wall and away from all primes, $\sfB/\sfB$ becomes isomorphic to the analogous adjoint quotient $\sfB_{\alpha}/\sfB_{\alpha}$ for a certain centralizer subgroup $\sfG_{\alpha} \subset \sfG_{\bQ}$ of semi-simple rank 1. In this semi-simple rank 1 case, we calculate the image of the associated graded functor using the weight filtration. It is determined by a certain order of vanishing condition along the remaining wall.

The second type of localization reduces us to semi-simple rank 0 over the integers. After localizing in $\sfT$ away from all walls, $\sfB/\sfB$ becomes isomorphic to the adjoint quotient of the torus. Then all filtrations split and \eqref{FullyFaithfulCheck} becomes trivially an isomorphism.

\section{Coherent side} \label{s:specside}
Here we exhibit generators of $\Perf_{\sfB}(\sfB)$ and prove that certain Hom spaces are free $R$-modules concentrated in degree 0. Moreover we calculate the image of a certain associated graded map in terms of order of vanishing conditions along the walls.

\subsection{Compact generators}
The following is similar to Lemma 21 of \cite{AB}.

\begin{proposition}\label{SpecGen}
The category $\Perf_{\sfG}(\widetilde{\sfG}) \simeq \Perf_{\sfB}(\sfB)$ is generated under finite colimits by each of the following:
\begin{enumerate}[label=(\alph*)]
\item \label{SpecGen1} $\cO(\eta)$ for $\eta \in \Lambda$,
\item \label{SpecGen2} $V \otimes \cO(\mu)$ for $V \in \Rep(\sfG)$ and $\mu \in \Lambda^+$,
\item \label{SpecGen3} $M_{\lambda} \otimes \cO(\mu)$ for $\lambda, \mu \in \Lambda^+$,
\end{enumerate}
\end{proposition}
\begin{proof}
Lemma \ref{PerfGeneration} says that $\Perf_{\sfB}(\sfB)$ is generated by vector bundles pulled back from representations of $\sfB$.
Since $\Perf_{\sfB}(\pt)$ is generated by characters, $\Perf_{\sfB}(\sfB)$ is generated by \ref{SpecGen1}.

Let $d$ be the dimension of the dual Weyl module $M_{\nu}$ of highest weight $\nu \in \Lambda^+$. Then the complex $\Sym^d(\cone(M_{\nu} \rightarrow \bZ_{\nu}))$ is acyclic, because the kernel of the highest weight arrow has dimension $d-1$. Therefore the following Koszul complex is exact
\beq \label{Koszul} 0 \rightarrow \Lambda^d M_{\nu} \otimes \cO \rightarrow \Lambda^{d-1}M_{\nu} \otimes \cO(\nu) \rightarrow \dots M_{\nu} \otimes \cO((d-1)\nu)  \rightarrow \cO(d\nu) \rightarrow 0.\eeq If $\nu \in \Lambda^+$ is sufficiently dominant compared to $\eta \in \Lambda$, then twisting \eqref{Koszul} by $(\Lambda^d M_{\nu})^* \otimes \cO(\eta)$ shows that $\cO(\eta)$ is a colimit of $V \otimes \cO(\mu)$ for $V \in \Rep(\sfG)$ and $\mu \in \Lambda^+$. Therefore $\Perf_{\sfB}(\sfB)$ is generated by \ref{SpecGen2}.

Costandard representations generate $\Perf_{\sfG}(\pt)$ under finite colimits. Therefore $\Perf_{\sfB}(\sfB)$ is also generated by \ref{SpecGen3}.
\end{proof}

\subsection{Affinization} Now we calculate the affinization of the Grothendieck--Springer scheme. The answer involves the characteristic polynomial map $\sfG = \sfG^{\sct}/\sfZ \rightarrow \sfC \coloneqq (\sfG^{\sct}/\!\!/\sfG^{\sct})/\sfZ$.

\begin{proposition}\label{PushStructureSheaf}
The underived pushforward of the structure sheaf along $p: \widetilde{\sfG} \rightarrow \sfG \times_{\sfC} \sfT$ yields $p_* \cO_{\widetilde{\sfG}} \simeq \cO_{\sfG \times_{\sfC} \sfT}$.
\end{proposition}
\begin{proof}
First we claim that $\sfG^{\reg} \times_{\sfC} \sfT$ is smooth. If $k$ is the residue field of a point in $\Spec \bZ$, then Corollary 3.12 of \cite{St} gives an isomorphism $(\widetilde{\sfG}^{\reg})_k \simeq \sfG^{\reg} \times_{\sfC} \sfT_k$. Therefore all $\sfG^{\reg} \times_{\sfC} \sfT_k$ are smooth of the same dimension. Moreover $\sfG^{\reg} \times_{\sfC} \sfT$ is a complete intersection because $\sfT \rightarrow \sfC$ is flat. This implies the desired smoothness of $\sfG^{\reg} \times_{\sfC} \sfT$ by miracle flatness.

Therefore the fiber product $\sfG \times_{\sfC} \sfT$ is smooth in codimension 1 and a complete intersection.
By Serre's criterion it is normal.\footnote{If the derived subgroup of $\sfG$ is not simply connected, then $\sfG \times_{\sfG/\!\!/\sfG} \sfT$ is neither smooth in codimension 1 nor normal.}
Hence $p_* \cO_{\widetilde{\sfG}} \simeq \cO_{\sfG \times_{\sfC} \sfT}$, because $p$ is proper and birational.
\end{proof}

\subsection{Specializing in the torus}\label{SpecialTorus}
For each finite coroot $\alpha$, let $\sfT_{\alpha} = \Spec R/(e^{\alpha} - 1)$ be the corresponding wall.

We will need to consider three different types of specialization in the torus directions.

\begin{enumerate}
\item[-] First we will consider specialization to a point $\zeta \in \sfT$. Let $\sfB^{\zeta} \coloneqq \sfB \times_{\sfT} \zeta$. Let $\sfG_{\zeta}$ be the centralizer of $\zeta$.
Set $\sfB_{\zeta} \coloneqq \sfG_{\zeta} \cap \sfB$, $\sfU_{\zeta} \coloneqq \sfG_{\zeta} \cap \sfU$, and $\sfB_{\zeta}^{\zeta} \coloneqq \sfB_{\zeta} \times_{\sfT} \zeta$.

\item[-] Second we will consider the following localization to semi-simple rank 1 in characteristic 0. Fix a coroot $\alpha$. Let $\sfT^{(\alpha)}$ be the localization of $\sfT$ away from all walls except for $\sfT_{\alpha}$ and away from all primes.
Let $\sfB^{(\alpha)} \coloneqq \sfB \times_{\sfT} \sfT^{(\alpha)}$. Let $\sfG_\bQ$ be the base change of $\sfG$ to the rational numbers. Let $\sfG_{\alpha} \subset \sfG_\bQ$ be the subgroup generated by the torus and $\pm \alpha$ root subgroups. Set $\sfB_{\alpha} \coloneqq \sfG_{\alpha} \cap \sfB$ and $\sfB_{\alpha}^{(\alpha)} \coloneqq  \sfB_{\alpha} \times_{\sfT} \sfT^{(\alpha)}$.

\item[-] Third we will consider the following localization to semi-simple rank 0. Let $\sfT^{(\bZ)}$ be the localization of $\sfT$ away from all walls. Let $\sfB^{(\bZ)} \coloneqq \sfB \times_{\sfT} \sfT^{(\bZ)}$.
\end{enumerate}

\begin{proposition}\label{B/BZeta}
There are natural isomorphisms of stacks
\begin{enumerate}[label=(\alph*)]
\item \label{B/BZeta1} $\sfB_{\zeta}^{\zeta}/\sfB_{\zeta} \simeq \sfB^{\zeta}/\sfB$,
\item \label{B/BZeta2}  $\sfB_{\alpha}^{(\alpha)}/\sfB_{\alpha} \simeq \sfB^{(\alpha)}/\sfB$,
\item \label{B/BZeta3}  $\sfT^{(\bZ)}/\sfT \simeq \sfB^{(\bZ)}/\sfB$.
\end{enumerate}
\end{proposition}
\begin{proof}
Write $\sfB_{\zeta}^{\zeta}/\sfB_{\zeta} \simeq (\sfB \times^{\sfB_{\zeta}} \sfB^{\zeta}_{\zeta})/\sfB$, where $\sfB_{\zeta}$ acts by conjugation on $\sfB^{\zeta}_{\zeta}$ and by right multiplication on $\sfB$.

We will prove that \beq \label{ZUnif} \sfB \times^{\sfB_{\zeta}} \sfB^{\zeta}_{\zeta} \rightarrow \sfB^{\zeta}, \qquad (a, b) \mapsto a b a^{-1}\eeq is an isomorphism.

First we check that it is surjective on geometric points. Let $b \in \sfB^{\zeta}$ be a geometric point and write $b = us$ for the Jordan decomposition. Both $\zeta$ and $s$ are semisimple with the same image in $\sfB/\!\!/\sfB$. Therefore $\zeta$ and $s$ are conjugate $\zeta = a^{-1} s a$ by some $a \in \sfB$. Hence $b = a ((a^{-1} u a) \zeta) a^{-1}$ where $a^{-1} u a \in \sfB_{\zeta} \cap \sfU$.

Now we check that it is injective on geometric points. Suppose that $aba^{-1} = b'$ for geometric points $a \in \sfB$ and $b, b' \in \sfB^{\zeta}_{\zeta} = \sfU_{\zeta} \zeta$. Write $b = u \zeta$ and $b' = u' \zeta$ for $u, u' \in \sfU_{\zeta}$. These are the Jordan decompositions of $b$ and $b'$, because both $b$ and $b'$ commute with $\zeta$. By uniqueness of the Jordan decomposition, $a \in \sfB_{\zeta}$ commutes with $\zeta$. Hence $(a, b) \sim (1, b')$ are equivalent in  $\sfB \times^{\sfB_{\zeta}} \sfB^{\zeta}$.

Since $\sfB \times \sfB_{\zeta}^{\zeta} \rightarrow \sfB^{\zeta}$ separable, \eqref{ZUnif} is also separable. Moreover $\sfB^{\zeta}$ is normal. Therefore Lemma \ref{ZMT} implies that \eqref{ZUnif} is an isomorphism.
Moreover \eqref{ZUnif} is $\sfB$-equivariant, and quotienting by $\sfB$ gives the desired \ref{B/BZeta1}.

The same argument also proves \ref{B/BZeta2} and \ref{B/BZeta3}, using that $\sfB_{\zeta} \subset \sfB_{\alpha}$ for all $\zeta \in \sfT^{(\alpha)}$ and $\sfB_{\zeta} \subset \sfT$ for all $\zeta \in \sfT^{(\bZ)}$.\footnote{In contrast, if $\sfG$ has disconnected center and $\zeta \in \sfT^{(\alpha)}$, then $\sfG_{\zeta}$ may be disconnected and need not be contained in $\sfG_{\alpha}$.}
\end{proof}

We used the following lemma to reduce to checking an isomorphism on geometric points.

\begin{lemma}\label{ZMT}
Let $f: X = \Spec A \rightarrow Y = \Spec B$ be a separated and finitely presented morphism of integral schemes. Assume that
\begin{enumerate}[label=(\alph*)]
\item\label{ZMT1} $f$ is a bijection on geometric points,
\item\label{ZMT2} $f$ is separable,
\item\label{ZMT3} $Y$ is normal.
\end{enumerate}
Then $f$ is an isomorphism.
\end{lemma}
\begin{proof}
Zariski's main theorem and \ref{ZMT1} imply that $f$ factors as an open immersion followed by a finite map. Therefore $A \otimes_B \Frac B = \Frac A$ and this is a finite extension of $\Frac B$.

Moreover \ref{ZMT2} implies that the degree of this extension equals number of $\Frac B$-linear embeddings of $A \otimes_B \Frac B$ into the algebraic closure of $\Frac B$.

Now \ref{ZMT1} (applied to the geometric points that factor through the generic point) implies that there is exactly one such embedding. Therefore $f$ is birational. Finally Zariski's main theorem and \ref{ZMT3} imply that $f$ is an isomorphism.
\end{proof}

\subsection{Nakayama's Lemma}
The following lemma is similar to Lemma 5.2 of \cite{BR}. 
\begin{lemma}\label{TorLemma}
Let $M \in \Perf(R)$ be a perfect complex of $R$-modules such that $M \otimes_R k_{\zeta}$ is concentrated in degree 0 for every closed point $\zeta \in \sfT$. Then $M$ is quasi-isomorphic to a free $R$-module in degree 0.
\end{lemma}
\begin{proof}
It suffices to show that for each closed point $\zeta \in \sfT$, the localization $M_{\zeta}$ is a free $R_{\zeta}$-module concentrated in degree 0. Represent $M_{\zeta}$ by a finite complex $M_{\zeta}^*$ of free $R_{\zeta}$-modules.

Let $i$ be the largest integer such that $M_{\zeta}^i \neq 0$. If $i>0$ then $M_{\zeta}^{i-1}  \otimes_R k_{\zeta} \rightarrow M_{\zeta}^i \otimes_R k_{\zeta}$ is surjective. Therefore $M_{\zeta}^{i-1} \rightarrow M_{\zeta}^i$ is a surjection by Nakayama's lemma, so we may replace $M_{\zeta}^*$ by a complex in degrees $\leq i-1$. Repeating shows that $M_{\zeta}$ is quasi-isomorphic to a finite complex of free $R_{\zeta}$-modules in degrees $\leq 0$.


Let $i$ be the smallest integer such that $M_{\zeta}^i \neq 0$. If $i < 0$ then $M_{\zeta}^i \otimes_{R_{\zeta}} k_{\zeta} \rightarrow M_{\zeta}^{i+1} \otimes_{R_{\zeta}} k_{\zeta}$ is injective. By Nakayama's Lemma $M_{\zeta}^i \rightarrow M_{\zeta}^{i+1}$ is a split injection, so we can replace $M_{\zeta}^*$ by a finite complex of free $R_{\zeta}$-modules in degrees $\geq i + 1$. Repeating shows that $M_{\zeta}$ is quasi-isomorphic to a free $R_{\zeta}$-module concentrated in degree 0.

Therefore $M$ is quasi-isomorphic to a finitely generated flat $R$-module in degree 0. The Laurent polynomial Quillen--Suslin theorem \cite{Sw} implies $M$ is free.
\end{proof}

\subsection{Freeness and higher vanishing} The following freeness and higher vanishing will be needed later to invoke Hartogs' lemma.
\begin{proposition}\label{SVan}
If $\lambda, \mu \in \Lambda^+$ then $R\Gamma(\widetilde{\sfG}/\sfG, M_{\lambda} \otimes \cO(\mu))$ is
\begin{enumerate}[label=(\alph*)]
\item \label{SVan1} perfect as a complex of $R$-modules,
\item \label{SVan2} concentrated in degree 0,
\item \label{SVan3} free over $R$.
\end{enumerate}
\end{proposition}
\begin{proof}
First we prove \ref{SVan1}. Note that $R\Gamma(\widetilde{\sfG}/\sfG, M_{\lambda} \otimes \cO(\mu))$ is the derived $\sfB$-invariants of $M_{\lambda} \otimes \cO(\sfB)(\mu)$. Let $I \subset \cO(\sfU)$ be the kernel of evaluation at the identity. Section 4.16 of \cite{Jan} gives the normalized Hochschild complex \beq\label{Hochschild} M_{\lambda} \otimes \cO(\sfB)(\mu) \rightarrow M_{\lambda} \otimes \cO(\sfB)(\mu) \otimes I \rightarrow M_{\lambda} \otimes \cO(\sfB)(\mu) \otimes I^{\otimes 2} \rightarrow \dots \eeq
whose $\sfT$-invariants compute the derived $\sfB$-invariants. If $\nu \in \Lambda$ is a weight appearing in $\cO(\sfB) \otimes I^{\otimes n}$, then $\langle \check{\rho}, \nu \rangle \geq n$ and the corresponding weight space is a finite rank over $R$. Therefore, after passing to $T$-invariants, \eqref{Hochschild} becomes a finite complex of finite rank $R$-modules.

Now we prove \ref{SVan2} and \ref{SVan3}. For $\zeta \in \sfT$ with residue field $k_{\zeta}$, base change and Proposition \ref{B/BZeta}\ref{B/BZeta1} imply \beq\label{RGammaZeta} R\Gamma(\sfB/\sfB, M_{\lambda} \otimes \cO(\mu)) \otimes_R k_{\zeta} \simeq R\Gamma(\sfB^{\zeta}_{\zeta}/\sfB_{\zeta}, M_{\lambda} \otimes \cO(\mu)).\eeq Theorem 2 and 7 of \cite{KLT} implies that $R\Gamma(\widetilde{\sfG}_{\zeta}, M_{\lambda} \otimes \cO(\mu))$ is concentrated in degree 0 and admits a dual Weyl filtration. (They work over a field of good characteristic in order to identify $\sfU \simeq \sfu$, but we do not need this assumption because we do not consider the Lie algebra.) Therefore \eqref{RGammaZeta} is concentrated in degree 0. Lemma \ref{TorLemma} implies $R\Gamma(\sfB/\sfB, M_{\lambda} \otimes \cO(\mu))$ is concentrated in degree 0 and free over $R$.
\end{proof}

\subsection{Order of vanishing in semi-simple rank 1}\label{Rank1Spectral}
Here we perform the coherent order of vanishing calculation. After localizing to semi-simple rank 1, we show that the image of a certain restriction functor is characterized by an order of vanishing condition along the remaining wall.

Let $\alpha$ be a coroot. Let $\sfG_{\alpha} \subset \sfG_{\bQ}$ be the subgroup generated by the torus and the $\pm \alpha$ root spaces. Let $\sfB_{\alpha} \coloneqq \sfG_{\alpha} \cap \sfB$.

Since $\sfG_\alpha$ is defined over the rationals, its representation theory is semisimple. Let $V_{\lambda} \in \Rep(\sfG_{\alpha})$ be the simple module of highest weight $\lambda \in \Lambda^+$.

\begin{proposition}\label{SpectralImageGraded}
If $0 \leq i \leq n \coloneqq \langle \check{\alpha}, \lambda \rangle$, then restriction to the torus \beq \label{istar} i^*: (V_{\lambda}(-\lambda + (n-i)\alpha) \otimes \cO(\sfB_{\alpha}))^{\sfB_{\alpha}} \rightarrow (V_{\lambda}(-\lambda + (n-i)\alpha) \otimes \cO(\sfT))^{\sfT} \simeq R\eeq is injective with image $(e^{\alpha} - 1)^iR$.
\end{proposition}
\begin{proof}
We first consider the case when $\sfG_{\alpha}$ has derived subgroup $\SL(2)$. Write $\cO(\sfB_\alpha) = R[u]$ where $u$ has weight $\alpha$.

Let $\omega \in \Lambda$ be a coweight such that $\langle \check{\alpha}, \omega \rangle = 1$.
Since $V_{\lambda} = V_{n\omega} \otimes \bQ_{\lambda - n\omega}$, it suffices to consider the case $\lambda = n\omega$.
 
Choose a basis $x, y \in V_{\omega}$ of weights $\omega - \alpha, \omega$ respectively. The $n\omega - (n-i) \alpha$ weight space of $V_{n\omega} \otimes \cO(\sfB_{\alpha})$ is spanned as an $R$-module by the vectors $x^{n-i}y^i \otimes 1, x^{n-i+1}y^{i-1} \otimes u, \dots x^n \otimes u^i$.

Choose a nontrivial element $b$ in the $-\alpha$ weight space of the Lie algebra of $\sfB_{\alpha}$. Up to scaling it acts on $\cO(\sfB_\alpha)$ by $bu = e^{\omega} - e^{\omega - \alpha}$ and it acts on $V_{n\omega} \simeq \Sym^n (V_{\omega})$ by $b x^i y^j = jx^{i+1}y^{j-1}$.
Therefore it acts on the tensor product by
\[\begin{tikzcd}[column sep = huge]
0 & \arrow[l] x^{n-i}y^i \arrow[d, "i"]& &\\
& x^{n-i+1}y^{i-1} & \arrow[l, "e^{\omega} - e^{\omega - \alpha}"'] x^{n-i+1}y^{i-1} \otimes u \arrow[d, "i-1"] & \\
& & x^{n-i+2}y^{i-2} \otimes u & \arrow[l, "2(e^{\omega} - e^{\omega - \alpha})"'] x^{n-i+2}y^{i-2} \otimes u^2.
\end{tikzcd}\]
Hence $(V_{n\omega}(-\lambda + (n-i)\alpha) \otimes \cO(\sfB_\alpha))^{\sfB_\alpha}$ is spanned as an $R$-module by the vector \[x^{n-i}y^i \otimes (e^{\omega - \alpha} - e^{\omega})^i + i x^{n - i + 1} y^{i-1} \otimes (e^{\omega - \alpha} - e^{\omega})^{i-1}u +\dots x^n \otimes u^i.\]
Restriction to the torus kills all terms divisible by $u$, so the image of \eqref{istar} is the ideal $(e^{\omega} - e^{\omega - \alpha})^iR = (e^{\alpha} - 1)^i R$.

If $\sfG_{\alpha}$ has derived subgroup $\PGL(2)$, then there exists a reductive group $\sfG^{\sct}_{\alpha}$ with derived subgroup $\SL(2)$, such that $\sfG_{\alpha} = \sfG^{\sct}_{\alpha}/\mu_2$.
If $\sfB^{\sct}_{\alpha} \subset \sfG^{\sct}_{\alpha}$ is the Borel, then \beq \label{i*Adjoint} i^*:(V_{\lambda}(-\lambda + (n-i)\alpha) \otimes \cO(\sfB_{\alpha}))^{\sfB_{\alpha}} \rightarrow (V_{\lambda}(-\lambda + (n-i)\alpha) \otimes \cO(\sfT))^{\sfT}\eeq identifies with the $\mu_2$-invariants of \[i^*:(V_{\lambda}(i) \otimes \cO(\sfB^{\sct}_{\alpha}))^{\sfB^{\sct}_{\alpha}} \rightarrow (V_{\lambda}(i) \otimes \cO(\sfT^{\sct}))^{\sfT^{\sct}},\] where $\mu_2$ acts by multiplication on $\sfT^{\sct}$ and $\sfB^{\sct}_{\alpha}$. Therefore the image of \eqref{i*Adjoint} is the desired $((e^{\alpha} - 1)^i\cO(\sfT^{\sct}))^{\mu_2} = (e^{\alpha} - 1)^i\cO(\sfT)$.
\end{proof}

\section{Iwahori--Whittaker sheaves}\label{WhittakerSheaves}
Here we define the Iwahori--Whittaker category, the constructible side of the universal monodromic Arkhipov--Bezrukavnikov equivalence. First we define the finite Whittaker category as the module on which the finite Hecke category acts through a universal monodromic version of Soergel's functor. Then we define the Iwahori--Whittaker category by tensoring over the finite Hecke category with the finite Whittaker module.

\subsection{The finite Whittaker category} \label{FiniteWhittaker}
Let $\Xi \in \Perv_{(I)}(\Fl)$ be the universal monodromic big tilting sheaf \cite{T} supported on $\overline{\Fl}_{w_0} = G/U$.

Let $\cH^{\fnt} \coloneqq \Shv_{(B)}(G/U)$ be the universal monodromic finite Hecke category.

\begin{definition}
Define Soergel's functor \beq\label{VFunctor} \bV- \coloneqq \Hom(\Xi, -): \cH^{\fnt} \rightarrow \Bim(R).\eeq
\end{definition}

Soergel's functor is lax monoidal according to \cite{LNY} and strict monoidal by \cite{T}.

\begin{definition}
Define the finite Whittaker category \[{}_{\chi}\cH^{\fnt} \coloneqq \QCoh(\sfT) \; \in \; \module\cH^{\fnt},\] a right module category on which the finite Hecke category acts via \[\cH^{\fnt} \xrightarrow{\bV} \Bim(R) \curvearrowright \Mod(R).\]
\end{definition}

Let ${}^{\chi}\nabla_1 \in {}_{\chi}\cH^{\fnt}$ correspond to the structure sheaf $R \in \QCoh(\sfT)$.

\begin{definition}
Define Whittaker averaging as the $\cH^{\fnt}$-linear functor \beq \label{AvChiFin}  {}^{\chi}\nabla_1 * -: \cH^{\fnt} \rightarrow {}_{\chi}\cH^{\fnt}, \qquad A \mapsto \bV(A)|_{1 \otimes R}\eeq
obtained by applying Soergel's functor and remembering only the right $R$-action.
\end{definition}

\begin{lemma}
Whittaker averaging is continuous and admits a left adjoint $\cH^{\fnt}$-module functor \beq \label{AvBFin} \Delta_1^{\chi} * - : {}_{\chi}\cH^{\fnt} \rightarrow \cH^{\fnt}, \qquad M \mapsto \Xi \otimes_R M.\eeq
\end{lemma}
\begin{proof}
Whittaker averaging is continuous because $\Xi$ is compact. Moreover its left adjoint $\Delta_1^{\chi} * -$ is a strict, not just lax, $\cH^{\fnt}$-module functor by the rigidity of $\cH^{\fnt}$, proved in Lemma \ref{Compact}
\end{proof}

\begin{lemma} \label{WhitFiniteLem}Whittaker averaging satisfies the following properties.
\begin{enumerate}[label=(\alph*)]
\item \label{WhitFiniteLem2} ${}^{\chi}\nabla_1 * \nabla_w \simeq {}^{\chi}\nabla_1 * \Delta_w \simeq {}^{\chi}\nabla_1 \in {}_{\chi}\cH^{\fnt}$ for all $w \in W^{\fnt}$.
\item \label{WhitFiniteLem1} The counit of adjunction $\eta: \Xi \coloneqq \Delta_1^{\chi} * {}^{\chi}\nabla_1 \rightarrow \Delta_1$ is surjective for the perverse t-structure, and $\ker \eta$ admits a filtration with graded pieces $\Delta_w$ for $w \neq 1 \in W^{\fnt}$.
\end{enumerate}
\end{lemma}
\begin{proof}
First we prove \ref{WhitFiniteLem2}. It is shown in \cite{T} that $\bV(\Delta_w) = \bV(\nabla_w) = R_w$, and restricting this bimodule to $1 \otimes R$ gives ${}^{\chi}\nabla_1 * \Delta_w \simeq {}^{\chi}\nabla_1 *\nabla_w \simeq R$.

Now we prove \ref{WhitFiniteLem1}. It is shown in \cite{T} that $\Xi$ admits a universal standard filtration with each $\Delta_w$ for $w \in W^{\fnt}$ appearing once. The counit $\eta$ generates the $R$-module $\Hom(\Xi, \Delta_1) = \Hom({}^{\chi}\nabla_1, {}^{\chi}\nabla_1) = R$, therefore part \ref{WhitFiniteLem1} follows by the Cousin filtration.
\end{proof}

\subsection{The Iwahori--Whittaker category}
Define the Iwahori--Whittaker category \[{}_\chi \cH \coloneqq {}_{\chi}\cH^{\fnt} \underset{\cH^{\fnt}}{\otimes} \cH \; \in \; \module \Shv_{(I)}(Fl).\]

\begin{remark} The use of Lurie's tensor product was not essential, alternatively this is the category of left comodules for the coalgebra object $\Xi \in \cH$.
\end{remark}

The functors \eqref{AvChiFin} and \eqref{AvBFin} induce a pair of adjoint $\cH$-module functors \[\Delta_1^{\chi} * -: {}_\chi \cH \rightleftarrows \cH
:{}^{\chi}\nabla_1 * -\] such that there is an isomorphism of comonads $\Delta_1^{\chi} * {}^{\chi}\nabla_1 * - \simeq \Xi * -$.

\subsection{Compact generators}
Here we exhibit compact generators of the Iwahori--Whittaker category.
\begin{lemma}\label{Compact}
For $A \in \cH$ the following are equivalent:
\begin{enumerate}[label=(\alph*)]
\item \label{Compact1} $A$ is compact,
\item \label{Compact2}$A$ is a finite colimit of $\Delta_w$ for $w \in W$,
\item \label{Compact3} the stalks $A|_{\Fl_w}$ are perfect over $R$, and vanish except for finitely many $w \in W$,
\item\label{Compact4} $A$ is dualizable.
\end{enumerate}
Moreover $\cH$ is rigid monoidal.
\end{lemma}
\begin{proof}
Conditions \ref{Compact1}, \ref{Compact2}, \ref{Compact3} are equivalent by arguments similar to in \cite{T}. We now show that they are also equivalent to \ref{Compact4}.

Universal standards $\Delta_w$ are invertible by Proposition \ref{Convolution} and hence also dualizable. Therefore finite colimits of $\Delta_w$ are dualizable. Conversely, all dualizable objects are compact because the monoidal unit $\Delta_1$ is compact.

To prove rigidity, it remains to check that $\cH$ is compactly generated by $\Delta_w$. Indeed every sheaf in $\cH$ admits a Cousin filtration whose $w$th graded piece is a colimit of $\Delta_w$.
\end{proof}

\begin{proposition}\label{AutGen}
The Whittaker category ${}_\chi \cH$ is compactly generated by each of the following,
\begin{enumerate}[label=(\alph*)]
\item \label{AutGen1} ${}^{\chi}\nabla_1 * W_{\eta}$ for $\eta \in \Lambda$,
\item \label{AutGen2} ${}^{\chi}\nabla_1 * Z(M_{\lambda}) * W_{\mu}$ for $\lambda, \mu \in \Lambda^+$.
\end{enumerate}
\end{proposition}
\begin{proof}
For $w \in W^{\fnt}$ and $\eta \in \Lambda$, the sheaves $\nabla_w * W_{\eta}$ compactly generate $\cH$ for support reasons.
Therefore ${}^{\chi}\nabla_1 * \nabla_w * W_{\eta}$ compactly generate ${}_\chi \cH$.
Lemma \ref{WhitFiniteLem}\ref{WhitFiniteLem2} says ${}^{\chi}\nabla_1 * \nabla_w * W_{\eta}  \simeq {}^{\chi}\nabla_1 * W_{\eta}$, therefore the objects \ref{AutGen1} compactly generate.

The objects \ref{AutGen2} also compactly generate using the Koszul complex for highest weight arrows, as in Proposition \ref{SpecGen}.
\end{proof}

\section{Order of vanishing calculation}\label{OrderVanishingSection}
Here we work with rational coefficients and specialize to unipotent monodromy. We prove that all endomorphisms of $\nabla_w$ that factor through $\Delta_w$ satisfy certain order of vanishing conditions. This will be used to calculate the image of the associated graded functor.
 
 \subsection{Mixed standard and costandards}
 Let $R_\bQ \coloneqq \bQ[\Lambda]$. Let $I \subset R_\bQ$ be maximal ideal of functions that vanish at the identity $1 \in \sfT_\bQ$.
 
 Define $\cH^i \coloneqq \cH \otimes_{\QCoh(\sfT)} \QCoh(R_\bQ/I^i)$. Its objects are ind-constructible with unipotent monodromy, so it admits a graded lift $\cH^i_{\mix}$, the derived category of mixed Hodge modules.
 
 Let $(R_{\bQ}/I^n)_{\Fl_w}$ be the perverse mixed local system on $\Fl_w$, such that its cosocle $(R_{\bQ}/I)_{\Fl_w}$, a constant sheaf perverse sheaf on $\Fl_w$, is pure of weight 0. Setting $I^0 \coloneqq R_{\bQ}$, the subquotient $(I^i/I^{i+1})_{\Fl_w}$ is pure of weight $-2i$.  Let \[\Delta_w^i \coloneqq j_{w!} (R_{\bQ}/I^i)_{\Fl_w}, \quad \IC_w^i \coloneqq j_{w!*} (R_{\bQ}/I^i)_{\Fl_w}, \quad \nabla_w^i \coloneqq j_{w*} (R_{\bQ}/I^i)_{\Fl_w} \; \in \; \cH^i_{\mix}.\]

 The following is from Lemma 4.4.7 of \cite{BY}.

\begin{proposition}\label{SocleCosocle} Let $w = \zeta u \in W$, where $\ell(\zeta) = 0$ and $u$ belongs to the (non-extended) affine Weyl group.
\begin{enumerate}[label=(\alph*)]
\item \label{SocleCosocle1} The socle $\soc \Delta_w^1 \simeq \IC_\zeta^1(\ell(w)/2)$ and no further Tate twists of $\IC_\zeta^1$ appear as Jordan-Holder factors.
\item \label{SocleCosocle2}The cosocle $\cos \nabla_w^1 \simeq \IC_\zeta^1(-\ell(w)/2)$ and no further Tate twists of $\IC_\zeta^1$ appear as Jordan-Holder factors.
\end{enumerate}
\end{proposition}
\begin{proof}
The proof is by induction on the length of $w$. Write $w = vs$ such that $\ell(w) = \ell(v) + 1$ and $s$ is a simple reflection.

There is a short exact sequence \[0 \rightarrow \IC_1^1(1/2) \rightarrow \Delta_s^1 \rightarrow \IC_s^1 \rightarrow 0.\] Convolving with $\Delta_v^1$ gives a short exact sequence \[0 \rightarrow \Delta_v^1(1/2) \rightarrow \Delta_w^1 \rightarrow \Delta_v^1 \star \IC_s^1 \simeq p^* \Delta_{\overline{v}}^1[1](1/2) \rightarrow 0.\] The quotient sheaf is the pull back along $p: G(\!(t)\!)/I \rightarrow G(\!(t)\!)/P_s$ of the standard extension $\Delta_{\overline{v}}^1 \in \Shv_{(I)}(G(\!(t)\!)/P_s)$.  By induction, there is an injection $\IC_w^1(\ell(w)/2) \hookrightarrow \Delta_w^1$ and no further Tate twists of $\IC_w^1$ appear in the Jordan-Holder filtration.

It remains to check that $\Hom(\IC_x^1, \Delta_w^1) \simeq 0$ for all $x \neq \zeta \in W$.  
\begin{enumerate}
\item[-] If $x < xs$, then $\IC_x^1$ does not appear as a Jordan-Holder factor of $\Delta_v^1 \star \IC_s^1$, hence $\Hom(\IC_x^1, \Delta_v^1 \star \IC_s^1) \simeq 0$. Therefore $\Hom(\IC_x^1, \Delta_w^1) \simeq \Hom(\IC_x^1, \Delta_v^1(1/2)) \simeq 0$ vanishes by induction.
\item[-] If $xs < x$, then
\begin{align*}\Hom(\IC_x^1, \Delta_w^1) &\simeq \Hom(p^* \IC_{\overline{x}}^1[1](1/2), \Delta_w^1) \\ &\simeq \Hom(\IC_{\overline{x}}^1[1](1/2), p_* \Delta_w^1)
\\ &\simeq \Hom(\IC_{\overline{x}}^1[1](1/2), \Delta_{\overline{w}}^1[-1](-1/2)) \\ &\simeq 0\end{align*}
vanishes for perverse cohomological degree reasons.
\end{enumerate}
This completes the proof of \ref{SocleCosocle1}, and \ref{SocleCosocle2} is similar.
\end{proof}

\subsection{Order of vanishing at unipotent monodromy}
The following proposition says that endomorphisms of $\nabla_w$ that factor through $\Delta_w$ must vanish to order $\ell(w)$ at the identity $1 \in \sfT$.

\begin{proposition}\label{NablaDeltaIdealZeta}
For $i \coloneqq \ell(w)$, the  map
\beq \label{ZetaPairing}\Hom(\nabla_w, \Delta_w) \otimes \Hom(\Delta_w, \nabla_w) \rightarrow \Hom(\nabla_w, \nabla_w) \rightarrow R_{\bQ}/I^i\eeq
vanishes.
\end{proposition}
\begin{proof}
Let $\phi \in \Hom(\nabla_w^i, \Delta_w^i)$ be nonzero of weight $-2m$. For some $j$ it factors through a map $\nabla_w^i \rightarrow I^j \Delta_w^i$ such that the composition \beq\label{grI}  \nabla_w^i \rightarrow I^j \Delta_w^i \rightarrow I^j \Delta_w^i/I^{j+1} \Delta_w^i\eeq is nonzero.  Moreover \eqref{grI} factors through a nonzero map \[\nabla^1_w \simeq \nabla_w^i/I\nabla_w^i \rightarrow I^j \Delta_w^i/I^{j+1} \Delta_w^i.\]

Proposition \ref{SocleCosocle}\ref{SocleCosocle2} says that $\cos \nabla^1_w \simeq \IC_\zeta^1(-i/2)$. Therefore $\IC_\zeta(-i/2 + m)$ is a Jordan-Holder factor of  $I^j \Delta_w^i/I^{j+1} \Delta_w^i$. But $I^j \Delta_w^i/I^{j+1} \Delta_w^i \simeq \Delta_w^1 \otimes I^j/I^{j+1}$ is a direct sum of copies of $\Delta_w^1(j)$. Thus Proposition \ref{SocleCosocle}\ref{SocleCosocle1} implies that $\phi$ has weight $-2m = - 2(i + j)$. Therefore $\Hom(\nabla_w^i, \Delta_w^i)$ is concentrated in weights $\leq - 2i$.

However $\Hom(\Delta_w^i, \nabla_w^i) \simeq R_\bQ/I^i$ is concentrated in weight $\leq 0$, and $\Hom(\nabla_w^i, \nabla_w^i) \simeq R_\bQ/I^i$ is concentrated in weights $> -2i$.
Therefore \eqref{ZetaPairing} vanishes.
\end{proof}

\subsection{Order of vanishing in rank one}
In this section we will assume that $G$ has semi-simple rank 1. Let $\alpha$ be the unique coroot and $s$ be the unique simple reflection. We will also temporarily work with rational instead of integral coefficients.

\begin{lemma}\label{NablaDeltaIdeal}
The image of \beq\label{BPairing} B_w: \Hom(\nabla_w, \Delta_w) \otimes \Hom(\Delta_w, \nabla_w) \rightarrow \Hom(\nabla_w, \nabla_w) \simeq R\eeq is the ideal $(e^{\alpha} - 1)^{\ell(w)} R$.
\end{lemma}
\begin{proof}
First we we prove \beq \label{ContainedIn} (\image B_w) \; \subset \; (e^{\alpha} - 1)^{\ell(w)} R \eeq
If $\zeta \in \sfT_\alpha$ is contained on the wall, it defines a local system on the entire $\Fl$. Tensoring with this local system defines an equivalence between sheaves with unipotent monodromy and sheaves with generalized monodromy $\zeta$. Therefore Proposition \ref{NablaDeltaIdealZeta} implies that the image of \eqref{BPairing} vanishes to order $i$ at $\zeta$. This implies \eqref{ContainedIn}.

Conversely, we claim by induction on $\ell(w)$ that \[(e^{\alpha} - 1)^{\ell(w)} R \; \subset \; (\image B_w).\] Write $w = vs$ such that $\ell(w) = \ell(v) + 1$ and $s$ is a simple reflection.
Recall that there is a short exact sequence \[0 \rightarrow \Delta_s \rightarrow \nabla_s \rightarrow \nabla_1/(e^{\alpha} - 1) \rightarrow 0\] where the first map generates $\Hom(\Delta_s, \nabla_s) \simeq R$. This gives an exact sequence \[0 \rightarrow \Hom(\nabla_s, \Delta_s) \rightarrow \Hom(\nabla_s, \nabla_s) \rightarrow \Hom(\nabla_s, \nabla_1/(e^{\alpha} - 1)) \simeq R/(e^{\alpha} - 1),\] in which the image of the first map contains $(e^{\alpha} - 1)R$. Hence $(\image B_s) = (e^{\alpha} - 1) R$. Therefore by induction \[(e^{\alpha} - 1)^{\ell(w)} R \; \subset \; (\image B_v)(\image B_s) \; \subset \; (\image B_w). \qedhere\]
\end{proof}

\section{Localized central sheaves} \label{s:loccentr}
Here we explain that, after localizing away from all but one wall, the central functor factors through restriction to a reductive subgroup of semi-simple rank 1. Then we will calculate the image of certain maps induced by taking associated graded with respect to the Wakimoto filtration.

\subsection{The constant term functor}
Let $G_{\check{\alpha}} \subset G$ be the subgroup generated by the torus and the $\pm \check{\alpha}$ root spaces. Let $P_{\check{\alpha}} \subset G$ be the subgroup generated by the Borel and the $-\check{\alpha}$ root space.

Let $\mathring{I'} \subset G_{\check{\alpha}}(\!(t)\!)$ be the pro-unipotent radical of the Iwahori subgroup. Let $\Fl' \coloneqq G_{\check{\alpha}}(\!(t)\!)/\mathring{I'}$ be the affine flag variety for $G_{\check{\alpha}}$. Let $\cH'$ be the affine Hecke category for $G_{\check{\alpha}}$. Let $W_\lambda' \in \cH'$ denote the Wakimoto sheaves for $G_{\check{\alpha}}$.

Let $(\Fl')^\lambda \subset \Fl'$ be the connected component indexed by $\lambda \in \Lambda/\alpha$. Let $S^\lambda \subset \Fl$ be the $P_{\check{\alpha}}(\!(t)\!)$ orbit through $t^{\lambda}$.

Let $\CT: \cH \rightarrow \cH'$
denote the constant term functor, defined using hyperbolic localization.
The following lemma says that it sends Wakimoto sheaves to Wakimoto sheaves.

\begin{lemma}\label{CTWaki}
The constant term functor sends $\CT(W_{\lambda}) \simeq W_{\lambda}'$.
\end{lemma}
\begin{proof}
If $\lambda \in \Lambda^+$ is dominant then $\Fl_\lambda \subset S^\lambda$.
Therefore we get a commutative diagram
\[\begin{tikzcd}
\Fl \widetilde{\times} \Fl_\lambda \arrow[d, "m"'] & \arrow[l, "s"'] S^{\mu - \lambda} \widetilde{\times} \Fl_\lambda \arrow[d] \arrow[r, "t"] & \Fl' \widetilde{\times} \Fl' \arrow[d, "m"] \\
\Fl & \arrow[l, "s"] S^\mu \arrow[r, "t"'] & \Fl'
\end{tikzcd}\]
where the left square is Cartesian. Hence for $A \in \cH$ we have 
\[\CT(A * \nabla_\lambda) \simeq t_*s^! m_* (A \widetilde{\boxtimes} \nabla_\lambda)[-\langle 2\check{\rho} - \check{\alpha}, \mu \rangle] \simeq m_*t_*s^!(A \widetilde{\boxtimes} \nabla_\lambda)[-\langle 2\check{\rho} - \check{\alpha}, \mu \rangle] \simeq \CT(A) * \nabla_\lambda'.\]

Therefore if $\lambda, \mu \in \Lambda^+$ are dominant, then \[\nabla_\lambda' \simeq \CT(\nabla_\lambda) \simeq \CT(W_{\lambda - \mu} * \nabla_\mu) \simeq \CT(W_{\lambda - \mu}) * \nabla_\mu'.\]
This implies the desired $\CT(W_{\lambda - \mu}) \simeq W_{\lambda - \mu}'$.
\end{proof}

\subsection{Localizing the central functor}
If $\alpha$ is a coroot, let $\sfT^{(\alpha)} = \Spec (R^{(\alpha)})$ be the localization of $\sfT$ away from all walls except for $\sfT_{\alpha}$ and away from all primes. Write $\cH^{(\alpha)} \coloneqq \cH \otimes_{\QCoh(\sfT)} \QCoh(\sfT^{(\alpha)})$ for the localization of the affine Hecke category. If $A \in  \cH^{(\alpha)}$ write $A^{(\alpha)} \coloneqq A \otimes R^{(\alpha)} \in \cH^{(\alpha)}$ for its localization.

Recall that $\sfG_{\alpha} \subset \sfG_\bQ$ is the subgroup generated by the torus and the $\pm \alpha$ root spaces.

\begin{lemma}\label{Zalpha}
After localizing away from all but one wall and away from all primes, Gaitsgory's functor factors through restriction 
\beq\label{ZalphaFact} \begin{tikzcd}
\Rep(\sfG) \arrow[r, "Z"] \arrow[d] & \cH^{(\alpha)} \\
\Rep(\sfG_\alpha). \arrow[ur] &
 \end{tikzcd}\eeq

Moreover if $\alpha$ is a simple coroot then there is a commutative diagram
\beq\label{LeviCT} \begin{tikzcd} \Rep(\sfG) \arrow[r, "Z"] \arrow[d] & \Perv_{(I)}^{\waki}(\Fl)^{(\alpha)} \arrow[r, "\gr"] \arrow[d, "\CT"'] & \Free_\sfT(\sfT^{(\alpha)}) \\
\Rep(\sfG_\alpha) \arrow[r, "Z"'] \arrow[ur] & \Perv_{(I')}^{\waki}(\Fl')^{(\alpha)}. \arrow[ur, "\gr"'] &
\end{tikzcd}\eeq
\end{lemma}
\begin{proof}
Recall that the Arkhipov--Bezrukavnikov functor $\Perf_\sfB(\sfB) \rightarrow \cH$ is $R$-linear with respect to the projection $\sfB/\sfB \rightarrow \sfT$ and the right $T$-monodromy action on $\cH$. Proposition \ref{B/BZeta}\ref{B/BZeta2} gives a commutative diagram
\[\begin{tikzcd}
\Rep(\sfG) \arrow[r] \arrow[d] & \Perf_\sfB(\sfB) \arrow[r] \arrow[d] & \cH \arrow[d] \\
\Rep(\sfG_\alpha) \arrow[r] & \Perf_{\sfB_{\alpha}}(\sfB^{(\alpha)}_{\alpha}) \arrow[r] & \cH^{(\alpha)}.
\end{tikzcd}\]

Hyperbolic localization commutes with nearby cycles by Braden's theorem \cite{Brad}. Therefore Proposition 15.3 of \cite{BR18} gives a commutative square
\beq\label{LeviRestriction} \begin{tikzcd}
\Rep(\sfG) \arrow[r, "Z"] \arrow[d] & \cH \arrow[d, "\CT"] \\
\Rep(\sfG_\alpha) \arrow[r, "Z"'] & \cH'.
\end{tikzcd}\eeq

Lemma \ref{CTWaki} gives a commutative triangle
\beq\label{LeviGr} \begin{tikzcd} \Perv_{(I)}^{\waki}(\Fl) \arrow[r, "\gr"] \arrow[d, "\CT"'] & \Free_\sfT(\sfT) \\
\Perv_{(I')}^{\waki}(\Fl'). \arrow[ur, "\gr"'] &
\end{tikzcd}\eeq
Combining \eqref{LeviRestriction} with \eqref{LeviGr} gives the desired \eqref{LeviCT}.
\end{proof}

\subsection{Image of associated graded} 
Since $\sfG_\alpha$ is defined over the rationals, its representation theory is semisimple. Let $Z_{\lambda}^{(\alpha)} \in \cH^{(\alpha)}$ be the image under \eqref{ZalphaFact} of the simple $\sfG_\alpha$ module of highest weight $\lambda$. This is an indecomposable direct summand of $Z(M_\lambda)^{(\alpha)}$.

\begin{lemma}\label{TopGraded}
Let $\alpha$ be a coroot and let $n \coloneqq \langle \check{\alpha}, \lambda \rangle$.
\begin{enumerate}[label=(\alph*)]
\item\label{TopGraded1} The image of \[\gr: \Hom(W_{\lambda}^{(\alpha)}, Z_{\lambda}^{(\alpha)}) \rightarrow \Hom(\gr W_{\lambda}^{(\alpha)}, \gr Z_{\lambda}^{(\alpha)}) \simeq R^{(\alpha)}\]
is the ideal $(e^{\alpha}-1)^n R^{(\alpha)}$.
\item\label{TopGraded2} The image of 
\[\Hom(\nabla_{\lambda}^{(\alpha)}, \Delta_{\lambda}^{(\alpha)}) \otimes \Hom(\Delta_{\lambda}^{(\alpha)}, \nabla_{\lambda}^{(\alpha)}) \rightarrow \Hom(\nabla_{\lambda}^{(\alpha)}, \nabla_{\lambda}^{(\alpha)}) \simeq R^{(\alpha)}\]
is the ideal $(e^{\alpha}-1)^n R^{(\alpha)}$.
\end{enumerate}
\end{lemma}
\begin{proof}
Using the Wakimoto filtration for the positive Borel, Lemma \ref{Hom0} implies $\Hom(\nabla_{\lambda}, \Delta_{\lambda}) \simeq \Hom(\nabla_{\lambda}, Z_{\lambda})$.
Therefore \beq \label{Top} \Hom(\nabla_{\lambda}, \Delta_{\lambda}) \xrightarrow{\sim} \Hom(\nabla_{\lambda}, Z_{\lambda}) \xrightarrow{\gr_{\lambda}} \Hom(\nabla_{\lambda}, \nabla_{\lambda}) \simeq R \eeq is induced by the natural map $\Delta_{\lambda} \rightarrow \nabla_{\lambda}$. Hence \ref{TopGraded1} is equivalent to \ref{TopGraded2}.

If $G$ has semi-simple rank 1 then Proposition \ref{NablaDeltaIdeal} implies \ref{TopGraded2} and hence also \ref{TopGraded1}. Therefore, if $\alpha$ is a simple coroot, then commutativity of \eqref{LeviCT} implies \ref{TopGraded1} and hence also \ref{TopGraded2}.

Now suppose that $\alpha$ is a not necessarily simple coroot. Choose $w \in W^{\fnt}$ such that $\ell(w) < \ell(ws)$ and $\beta \coloneqq w^{-1}\alpha$ is a simple coroot. Then $ \Delta_w^{(\beta)} \simeq \nabla_w^{(\beta)}$ is clean, $\nabla_{w^{-1}}^{(\alpha)} * \Delta_{\lambda}^{(\alpha)} * \Delta_w^{(\beta)} \simeq \Delta_{w(\lambda)}^{(\beta)}$ and $\nabla_{w^{-1}}^{(\alpha)} * \nabla_{\lambda}^{(\alpha)} * \Delta_w^{(\beta)} \simeq \nabla_{w(\lambda)}^{(\beta)}$. Hence \ref{TopGraded2} for simple coroots implies \ref{TopGraded2} for all coroots.
\end{proof}

The following is an inductive formula for truncations of central sheaves. 

\begin{lemma}\label{TruncateZ}
Assume that there exists a coweight $\omega \in \Lambda$ such that $\langle \check{\alpha}, \omega \rangle = 1$. For $\lambda \in \Lambda^+$, let $(Z_{\lambda + \omega})^{\leq \lambda + s\omega}$ be the $\leq \lambda + s\omega$ filtered piece of the Wakimoto filtration on $Z_{\lambda + \omega}$.
Then there is an isomorphism \[Z_{\lambda}^{(\alpha)}*W_{s\omega}^{(\alpha)} \simeq (Z_{\lambda + \omega}^{(\alpha)})^{\leq \lambda + s\omega}.\]
\end{lemma}
\begin{proof}
By monoidality $Z_{\lambda}^{(\alpha)} * Z_{\omega}^{(\alpha)} \simeq Z_{\lambda + \omega}^{(\alpha)} \oplus Z_{\lambda + s\omega}^{(\alpha)}$, and there is a commutative diagram
\beq \begin{tikzcd} \label{TrucateSquare}
Z_{\lambda}^{(\alpha)}* W_{s\omega}^{(\alpha)} \arrow[d] \arrow[r, dashrightarrow, "a"] & (Z_{\lambda + \omega}^{(\alpha)})^{\leq \lambda + s\omega} \arrow[d] \\
Z_{\lambda}^{(\alpha)} * Z_{\omega}^{(\alpha)} \arrow[r] & Z_{\lambda + \omega}^{(\alpha)}.
\end{tikzcd}\eeq
Since $\gr Z: \Rep(\sfG) \rightarrow \Free_{\sfT}(\sfT)$ is isomorphic to restriction, $\gr$ sends \eqref{TrucateSquare} to
\[\begin{tikzcd}
V_{\lambda}|_{\sfT} \otimes R_{s\omega}^{(\alpha)} \arrow[d] \arrow[r, dashrightarrow, "\gr a"] & V_{\lambda + \omega}|_{\sfT}^{\leq \lambda + s\omega} \otimes R^{(\alpha)} \arrow[d] \\
V_{\lambda}|_{\sfT} \otimes V_{\omega}|_{\sfT} \otimes R^{(\alpha)} \arrow[r] & V_{\lambda + \omega}|_{\sfT} \otimes R^{(\alpha)},
\end{tikzcd}\]
in which $\gr a$ is an isomorphism. Therefore $a$ is also an isomorphism.
\end{proof}

The following is the constructible counterpart to Proposition \ref{SpectralImageGraded}.

\begin{proposition}\label{AutImageGraded}
Let $\lambda \in \Lambda^+$ and set $n = \langle \check{\alpha}, \lambda \rangle$. If $0 \leq i \leq n$, then \beq\label{GrWZ} \gr:\Hom(W_{\lambda - (n-i) \alpha}^{(\alpha)}, Z_{\lambda}^{(\alpha)}) \rightarrow \Hom(\gr W_{\lambda - (n-i) \alpha}^{(\alpha)}, \gr Z_{\lambda}^{(\alpha)}) \simeq R^{(\alpha)}\eeq has image $(e^{\alpha} - 1)^i R^{(\alpha)}$.
\end{proposition}
\begin{proof}
If $G$ has disconnected center, then \eqref{GrWZ} is obtained by taking $\sfZ$-invariants of the corresponding map for the affine Hecke category of $G^{\ad} \coloneqq G/Z$. Therefore we may assume that there exists a coweight $\omega \in \Lambda$ such that $\langle \check{\alpha}, \omega \rangle = 1$. 

The proof is by induction on $n$. If $i = n$ the result holds by Lemma \ref{TopGraded}\ref{TopGraded2}.
Now consider the case $i < n$. Then Lemma \ref{TruncateZ} implies \begin{align*}\Hom(W_{\lambda - (n-i) \alpha}^{(\alpha)}, Z_{\lambda}^{(\alpha)}) &\simeq \Hom(W_{\lambda - (n-i) \alpha}^{(\alpha)}, (Z_{\lambda}^{(\alpha)})^{\leq \lambda - \alpha}) \\ &\simeq \Hom(W_{\lambda - (n-i) \alpha}^{(\alpha)}, Z_{\lambda - \omega}^{(\alpha)}*W_{s\omega}^{(\alpha)})\\ &\simeq \Hom(W_{\lambda - (n-i) \alpha - s\omega}^{(\alpha)}, Z_{\lambda - \omega}^{(\alpha)}).\end{align*} By induction
$\gr \Hom(W_{\lambda - (n-i) \alpha - s\omega}^{(\alpha)}, Z_{\lambda - \omega}^{(\alpha)}) \simeq (e^{\alpha} - 1)^i R^{(\alpha)}$.
\end{proof}

\subsection{Localization to rank zero}
Let $\sfT^{(\bZ)}$ be the localization of $\sfT$ away from all walls.
Write $\Perv_{(I)}(\Fl)^{(\bZ)} \coloneqq \Perv_{(I)}(\Fl) \otimes_{\QCoh(\sfT)} \QCoh(\sfT^{(\bZ)})$ for the corresponding localization of the affine Hecke category.

\begin{lemma}\label{CentralRank0}
After localizing away from all walls and away from all but one prime, Gaitsgory's functor
factors through
\[\begin{tikzcd}
\Rep(\sfG) \arrow[r, "Z"] \arrow[d] & \cH^{(\bZ)} \\
\Rep(\sfT). \arrow[ur] &
 \end{tikzcd}\]
\end{lemma}
\begin{proof}
Proposition \ref{B/BZeta}\ref{B/BZeta3} gives a commutative diagram
\[\begin{tikzcd}
\Rep(\sfG) \arrow[r] \arrow[d] & \Perf_\sfB(\sfB) \arrow[r] \arrow[d] & \cH \arrow[d] \\
\Rep(\sfT) \arrow[r] & \Perf_{\sfT}(\sfT^{(\bZ)}) \arrow[r] & \cH^{(\bZ)}.
\end{tikzcd}\]
\end{proof}

\section{The Whittaker equivalence} \label{s:whiteq}
Here we prove the universal monodromic enhancement of the Arkhipov--Bezrukavnikov equivalence \cite{AB}.

\subsection{Costandard filtration on Whittaker averaged central sheaves}
The following theorem says that the image of costandard representations under Gaitsgory's central functor admits a costandard filtration after Whittaker averaging.

\begin{theorem}\label{XiZTilting}
If $M_{\lambda} \in \Rep(\sfG)$ is a costandard representation, then $\Xi * Z(M_{\lambda}) \in \cH$ admits a universal monodromic costandard filtration.
\end{theorem}
\begin{proof}
This is proved in \cite{T26}, extending results of \cite{AB, BRR}.
\end{proof}

This costandard filtration property implies following freeness and higher vanishing, which will be needed later to invoke Hartogs' lemma.

\begin{proposition}\label{ZChiFree} If $\lambda, \mu \in \Lambda^+$ then
\begin{enumerate}[label=(\alph*)]
\item \label{ZChiFree1} $\RHom({}^{\chi}\nabla_1, {}^{\chi}\nabla_1 * Z(M_{\lambda}) * W_{\mu})$ is concentrated in degree 0 and free over $R$,
\item  Whittaker averaging induces an isomorphism \label{ZChiFree2}$\Hom(\Delta_1, Z(M_{\lambda})*W_{\mu}) \rightarrow \Hom({}^{\chi}\nabla_1, {}^{\chi}\nabla_1 * Z(M_{\lambda})*W_{\mu})$.
\end{enumerate}
\end{proposition}
\begin{proof}
First we prove \ref{ZChiFree1}. By adjunction \beq \label{HomAvZ} \RHom({}^{\chi}\nabla_1, {}^{\chi}\nabla_1 * Z(M_{\lambda}) * W_{\mu}) \simeq \RHom(\Delta_{-\mu}, \Xi * Z(M_\lambda)).\eeq Theorem \ref{XiZTilting} implies $\Xi * Z(M_{\lambda})$ admits a universal standard filtration. Hence \eqref{HomAvZ} is concentrated in degree 0 and free over $R$, by equation \eqref{StandardtoCostandard}.

For \ref{ZChiFree2}, if $w \neq 1 \in W^{\fnt}$ and $\nu \in \Lambda$ then $\Hom(\Delta_w, W_{\nu}) \simeq 0$ by Lemma \ref{Hom0}.  
Lemma \ref{WhitFiniteLem}\ref{WhitFiniteLem1} then implies $\Hom(\ker \eta, Z(M_{\lambda}) * W_{\mu}) \simeq 0$. Therefore part \ref{ZChiFree2} follows by the exact sequence \[0 \rightarrow \Hom(\Delta_1, Z(M_{\lambda}) * W_{\mu}) \rightarrow \Hom(\Xi, Z(M_{\lambda}) * W_{\mu}) \rightarrow \Hom(\ker \eta, Z(M_{\lambda}) * W_{\mu}) \rightarrow \dots. \qedhere\]
\end{proof}

\subsection{Fully faithfulness by localization to semi-simple rank 1}
For fully faithfulness, we used the following proposition. Its proof uses Hartogs' Lemma to reduce to semi-simple rank 1. In the semi-simple rank 1 case, we use the order of vanishing calculations from Section \ref{Rank1Spectral}.

\begin{proposition}\label{FullyFaithfulMain}
If $\lambda, \mu \in \Lambda^+$ then \beq \label{MainIso} \Hom_{\sfB/\sfB}(\cO, M_{\lambda} \otimes \cO(\mu)) \xrightarrow{F} \Hom_{\cH}(W_0, Z(M_{\lambda}) * W_{\mu})\eeq is an isomorphism.
\end{proposition}
\begin{proof}
Proposition \ref{FunctorOpen} gives a commuting diagram
\[\begin{tikzcd}
\Hom_{\sfB/\sfB}(\cO, M_{\lambda} \otimes \cO(\mu)) \arrow[r, "F"] \arrow[d, "i^*"', hook]& \Hom(W_0, Z(M_{\lambda}) * W_{\mu}) \arrow[d, "\gr", hook]\\
\Hom_{\sfT/\sfT}(i^* \cO, i^* (M_{\lambda} \otimes \cO(\mu))) \arrow[r, "\sim"]& \Hom(\gr W_0, \gr (Z(M_{\lambda}) * W_{\mu})). \end{tikzcd}\]
Proposition \ref{grFaithful} says $\gr$ is injective. By a similar argument $i^*$ is also injective. Therefore it suffices to check that $\gr$ and $i^*$ have the same image.

Both sides of \eqref{MainIso} are free finite rank $R$-modules by Propositions \ref{ZChiFree} and \ref{SVan}. Every codimension 2 subvariety of $\sfT$ intersects either $\sfT^{(\alpha)}$ for some coroot $\alpha$ or $\sfT^{(\bZ)}$. Therefore by Hartogs' Lemma it suffices to check the following.
\begin{enumerate}[label=(\roman*)]
\item\label{Hartog1} For each root $\alpha$, the functors $\gr$ and $i^*$ have the same image after localizing away from all walls except $\sfT_{\alpha}$ and away from all primes.
\item\label{Hartog2} The functor $\gr$ becomes an isomorphism after localizing away from all walls.
\item\label{Hartog3} The functor $i^*$ becomes an isomorphism after localizing away from all walls.
\end{enumerate}
Part \ref{Hartog1} follows by Propositions \ref{SpectralImageGraded} and \ref{AutImageGraded} combined with commutativity of \eqref{LeviCT}. Part \ref{Hartog2} follows because Lemma \ref{CentralRank0} implies that the Wakimoto filtration splits after localizing away from all walls. Part \ref{Hartog3} follows because Proposition \ref{B/BZeta}\ref{B/BZeta3} implies that $i$ becomes an isomorphism after localizing away from all walls.
\end{proof}

\subsection{Proof of the Whittaker equivalence}
Let $F$ be the Arkhipov-Bezrukavnikov functor, constructed in Proposition \ref{FunctorOpen}. We now show that it becomes an equivalence after post composition with Whittaker averaging.

\begin{theorem}\label{Main}
The following functor is an equivalence of categories \[{}_{\chi}F : \IndPerf_{\sfG}(\widetilde{\sfG}) \xrightarrow{F} \cH \xrightarrow{{}^{\chi}\nabla_1 * -} {}_\chi \cH.\]
\end{theorem}
\begin{proof}
The essential image contains  ${}_{\chi}F (\cO(\eta)) \simeq {}^{\chi}\nabla_1 * W_{\eta}$ for all $\eta \in \Lambda$. Therefore ${}_{\chi}F$ is essentially surjective by Proposition \ref{AutGen}.

Now we check fully faithfulness. Propositions \ref{FullyFaithfulMain} and \ref{SVan}\ref{SVan2} imply that $F$ and ${}^{\chi}\nabla_1 * -$ induce isomorphisms \[\RHom(\cO, M_{\lambda} \otimes \cO(\mu)) \xrightarrow{\sim} \Hom(W_0, Z(M_{\lambda}) \otimes W_{\mu}) \xrightarrow{\sim} \RHom({}^{\chi}\nabla_1, {}^{\chi}\nabla_1 * Z(M_{\lambda}) \otimes W_{\mu})\]
for $\lambda, \mu \in \Lambda^+$.
Therefore ${}_{\chi}F$ is fully faithful by Proposition \ref{SpecGen} and rigidity.
\end{proof}

{\large \part{Proof of the bi-Whittaker equivalence}\label{Part3}}

In this third part of the paper we prove that the universal monodromic affine bi-Whittaker category is equivalent to ind-perfect complexes on $\sfG/\sfG$. This is a universal monodromic enhancement of results of \cite{BezNilp, CD}, and it is Koszul dual to derived Satake \cite{BF}.

First we will use the central structure on Gaitsgory's functor to construct a functor from ind-perfect complexes to the bi-Whittaker category. To prove fully faithfulness it suffices to pullback along $\sfT \rightarrow \sfC$, after which both sides embed fullly faithfully inside the Arkhipov--Bezrukavnikov category. Using this we will deduce the bi-Whittaker equivalence from the Arkhipov--Bezrukavnikov equivalence.

The construction of the bi-Whittaker functor is more difficult when $G$ has disconnected center, because then $\Perf_{\sfG}(\sfG)$ is not generated under finite colimits by $\cO(\sfG)$-modules that are pulled back from representations of $\sfG$.\footnote{We thank Xinyu Li and Jiahao Niu for explaining this to us.} For this reason we will prove that the bi-Whittaker functor for $G^{\ad}$ is equivariant for the action of the finite central group subscheme $\sfZ \subset \sfG^{\sct}$. Thus we will obtain a functor out of $\cO(\sfG^{\sct})$-modules that are pulled back from representations of $\sfG \times \sfZ$. Such modules generate $\Perf_{\sfG \times \sfZ}(\sfG^{\sct})$ under finite colimits, which will enable us to obtain the bi-Whittaker functor.

\section{Preliminaries} \label{BiWhitPrelim}
Here we prove that the finite bi-Whittaker category is equivalent to quasi-coherent sheaves on $\sfC$. Moreover we explain how to extend a central functor out of an additive monoidal category of perverse sheaves to a central functor out of the affine Hecke category.

\subsection{Notation}
Let $\wHfs$ and $\Hwfs$ denote the right and left module categories on which the finite Hecke category $\Hfs$ acts via Soergel's functor as in Section \ref{FiniteWhittaker}. Define the finite bi-Whittaker category \[\wHwfs := \wHfs \otimes_{\Hfs} \Hwfs \simeq \Hom_{\Hfs}(\wHfs, \wHfs),\]
where we used the rigidity of $\Hfs$ and canonical duality $\Hwfs \simeq (\wHfs)^{\vee}$.

Let $\Hws := \Hs \otimes_{\Hfs} \Hwfs$ denote the affine Whittaker category.
Define the bi-Whittaker category
\[\wHws := \wHfs \otimes_{\Hfs} \Hs \otimes_{\Hfs} \Hwfs \simeq \Hom_{\Hs}(\wHs, \wHs).\]

Recall that there exists $\sfG^{\sct}$, a product of a torus and a simply connected reductive group, such that $\sfG = \sfG^{\sct}/\sfZ$ is the quotient by a finite central subgroup. We defined $\sfC \coloneqq (\sfG^{\sct}/\!\!/\sfG^{\sct})/\sfZ$.

Let $\Xi \in \Perv_{(B)}(G/U)$ denote the universal monodromic big tilting sheaf. Recall from \cite{T} that endomorphisms of the big tilting are \beq \label{Endomorphismensatz}\Hom(\Xi, \Xi) \simeq \cO(\sfT \times_{\sfC} \sfT).\eeq

\subsection{The finite bi-Whittaker category}
The following is a Betti version of results of \cite{Gin, Lon, Gan} in the de Rham setting and of \cite{BD} in the $\ell$-adic setting. In our setup it follows from the universal monodromic Endomorphismensatz \cite{T}.

\begin{lemma}\label{FiniteBiwhit} There is a monoidal equivalence \[\QCoh(\sfC) \simeq \wHwfs,\] such that the following square commutes
\beq \label{FiniteBiWhitAveraging}\begin{tikzcd} \QCoh(\sfC) \arrow[d, "p^*"'] & \arrow[l, "\sim"'] \wHwfs \arrow[d, "- * {}^{\chi}\nabla_1"] \\
 \QCoh(\sfT) & \arrow[l, "\sim"] \wHfs.  \\ \end{tikzcd}\eeq
\end{lemma}
\begin{proof}
The actions of $\Hfs$ on $\wHfs$ and $\Hwfs$ each factor through a continuous monoidal functor \[\Hfs \rightarrow \QCoh(\sfT \times_{\sfC} \sfT),\] that admits a fully faithful left adjoint by \eqref{Endomorphismensatz}. Therefore we get an equivalence \[\wHwfs \simeq  \wHfs \otimes_{\Hfs} \Hwfs \simeq \QCoh(\sfT) \otimes_{\QCoh(\sfT \times_{\sfC} \sfT)} \QCoh(\sfT) \simeq \QCoh(\sfC),\] making \eqref{FiniteBiWhitAveraging} commute.
The last equivalence is by \cite{BFN}, using that $p:\sfT \rightarrow \sfC$ is a proper and surjective morphism of perfect stacks.
\end{proof}

\subsection{Centrality at the derived level}
Proposition \ref{Centrality2} implies that Gaitsgory's functor extends to a monoidal functor to the Drinfeld center of any monoidal subcategory of perverse sheaves. The following lemma will allow us to extend this central structure to the level of DG-categories.

\begin{lemma}\label{CentralityDerived}
Let $\cA \subset \Perv_{(I)}(\Fl)$ be the full additive monoidal subcategory generated by universal central sheaves and universal tilting sheaves.

Then there is a monoidal functor $\Z(\cA) \rightarrow \Z(\cH)$ fitting into a commutative square
\beq\label{CentralityDerivedSquare} \begin{tikzcd}
\Z(\cA) \arrow[r] \arrow[d] & \Z(\cH) \arrow[d]\\
\cA \arrow[r] & \cH.
\end{tikzcd}\eeq
\end{lemma}
\begin{proof}
Let $\cT \subset \Perv_{(I)}(\Fl)$ be the full additive monoidal subcategory generated by universal tilting sheaves.

The functor $\K^{\bnd}(\cA) \rightarrow \cH^{\cpt}$ admits a monoidal section $\cH^{\cpt} \simeq \K^{\bnd}(\cT) \rightarrow \K^{\bnd}(\cA)$.
Therefore Lemma \ref{Center} gives a functor to the center of the affine Hecke category, fitting into a commutative diagram
\[\begin{tikzcd}
\Z(\cA) \arrow[r] \arrow[d] & \Z(\K^{\bnd}(\cA)) \arrow[d] \arrow[r] & \Z(\cH^{\cpt}) \arrow[d] \arrow[r] & \Z(\cH) \arrow[d] \\
\cA \arrow[r] & \K^{\bnd}(\cA) \arrow[r] & \cH^{\cpt} \arrow[r] & \cH
\end{tikzcd}\]
as desired.
\end{proof}

We used the following lemma to show that $\K^{\bnd}(\cT) \rightarrow \cH^{\cpt}$ induces a functor on Drinfeld centers.

\begin{lemma}\label{Center}
Let $F: \cA \rightarrow \cB$ be a monoidal functor of DG-categories that admits a monoidal section. Then there is an induced monoidal functor on Drinfeld centers that fits into a commutative square
\beq\label{Center1Square} \begin{tikzcd}
\Z(\cA) \arrow[r] \arrow[d] & \Z(\cB) \arrow[d] \\
\cA \arrow[r, "F"'] & \cB.
\end{tikzcd}\eeq
\end{lemma}
\begin{proof}
Let $G: \cB \rightarrow \cA$ be a monoidal section of $F$.

Recall from Section 5.3 of \cite{L17} that the Drinfeld center satisfies the following universal property. Namely $\Z(\cA)$ is the final monoidal category equipped with a monoidal functor $\Z(\cA) \otimes \cA \rightarrow \cA$ fitting into a commutative triangle
\[\begin{tikzcd}[column sep = tiny]
& \Z(\cA) \otimes \cA \arrow[dr] & \\
\cA \arrow[rr, "\id"'] \arrow[ur, "1 \otimes \id"] && \cA.
\end{tikzcd}\]

By this universal property applied to the commutative diagram
\[\begin{tikzcd}[column sep = large, row sep = large] & \arrow[dl, dashed, "F' \otimes \id"'] \Z(\cA) \otimes \cB \arrow[r, "\id \otimes G"] & \Z(\cA) \otimes \cA \arrow[dd, bend left = 70]\\ \Z(\cB) \otimes \cB\arrow[dr] & \arrow[l, "1 \otimes \id"] \cB \arrow[u, "1 \otimes \id"'] \arrow[r, "G"'] \arrow[d, "\id"] & \cA \arrow[u, "1 \otimes \id"] \arrow[d, "\id"'] \\  & \cB & \arrow[l, "F"] \cA \end{tikzcd}\]
we obtain a monoidal functor $F': \Z(\cA) \rightarrow \Z(\cB)$.

It fits into a commutative diagram
\[\begin{tikzcd}[column sep = large]
& \arrow[dl, "\id \otimes 1"'] \Z(\cA) \arrow[r, "F'"] \arrow[d, "\id \otimes 1"] & \Z(\cB) \arrow[d, "\id \otimes 1"] \\
\Z(\cA) \otimes \cA \arrow[d] & \arrow[l, "\id \otimes G"] \Z(\cA) \otimes \cB \arrow[r, "F' \otimes \id"']  & \Z(\cB) \otimes \cB \arrow[d] \\
\cA \arrow[rr, "F"'] && \cB
\end{tikzcd}\]
whose outer circuit gives \eqref{Center1Square}.
\end{proof}

\subsection{Generation by induced bundles}\label{GenerationInduced}
The following result from Section VIII.5.2 of \cite{FS} says that $\Perf_{\sfG^{\sct} \times \sfZ}(\sfG^{\sct})$ is generated by vector bundles that are pulled back from representations of $\sfG^{\sct} \times \sfZ$. We will present a different argument that was proposed to us by ChatGPT. It uses ideas from \cite{E09, FS} and a theorem of \cite{SvdK24}.

\begin{proposition}\label{IndPerfGsc}
The category $\Perf_{\sfG^{\sct} \times \sfZ}(\sfG^{\sct})$ is generated under retracts and finite colimits by vector bundles pulled back from representations of $\sfG^{\sct} \times \sfZ$, i.e. the inclusion $\Perf_{\sfG^{\sct} \times \sfZ}^{\ind}(\sfG^{\sct}) \subset \Perf_{\sfG^{\sct} \times \sfZ}(\sfG^{\sct})$ is an equivalence.
\end{proposition}
\begin{proof}
Consider the Cartesian diagram
\[\begin{tikzcd}
\sfB^{\sct}/\sfB^{\sct} \arrow[r, "i"] \arrow[d] & \sfG^{\sct}/\sfB^{\sct} \arrow[d] \\
\sfB^{\sct} \backslash \sfB^{\sct}/\sfB^{\sct} \arrow[r, "i"'] & \sfB^{\sct} \backslash \sfG^{\sct}/\sfB^{\sct}.
\end{tikzcd}\]
(On the top row, the quotients are by the adjoint action of $\sfB^{\sct}$.)

Lemma \ref{CoherentHecke} implies that $i_* \cO_{\sfB^{\sct} \backslash \sfB^{\sct}/\sfB^{\sct}} \in \Perf_{\sfB^{\sct} \times \sfB^{\sct}}^{\ind}(\sfG^{\sct})$. Therefore by base change $i_* \cO_{\sfB^{\sct}/\sfB^{\sct}} \in \Perf_{\sfB^{\sct}}^{\ind}(\sfG^{\sct})$. Combined with Lemma \ref{PerfGeneration}, this gives $i_* \cO_{1/\sfB^{\sct}} \in \Perf_{\sfB^{\sct}}^{\ind}(\sfG^{\sct})$. Therefore Proposition VIII.5.11 of \cite{FS} gives $\Perf_{\sfB^{\sct}}^{\ind}(\sfG^{\sct}) \simeq \Perf_{\sfB^{\sct}}(\sfG^{\sct})$.

Using that $\sfZ$ is diagonalizable, the forgetful functor $\IndPerf_{\sfG^{\sct} \times \sfZ}(\pt) \rightarrow \IndPerf_{\sfB^{\sct}}(\pt)$ is conservative. Therefore the proof of Corollary VIII.5.9 of \cite{FS} shows that $\Perf_{\sfG^{\sct} \times \sfZ}^{\ind}(\sfG^{\sct}) \simeq \Perf_{\sfG^{\sct} \times \sfZ}(\sfG^{\sct})$.
\end{proof}

We used the following lemma. 

\begin{lemma}\label{CoherentHecke}
The category $\Perf_{\sfB^{\sct} \times \sfB^{\sct}}(\sfG^{\sct})$ is generated under retracts and finite colimits by vector bundles pulled back from representations of $\sfB^{\sct} \times \sfB^{\sct}$.
\end{lemma}
\begin{proof}
Let $\sfX \coloneqq \sfG^{\sct}/\sfB^{\sct}$ be the flag variety. By Theorem 13.1 of \cite{SvdK24} and the assumption that $\sfG^{\sct}$ has simply connected derived subgroup, there exist objects in $\Perf(\sfX)$ pulled back from bounded complexes of finite dimensional representations of $\sfB^{\sct}$, that form a full exceptional collection.

Now we claim that these objects admit $\sfB^{\sct}$-equivariant structures. Indeed using the factorization 
\[\begin{tikzcd}
\sfX \arrow[r] \arrow[d] & \pt/\sfB^{\sct} \\
\sfB^{\sct} \backslash \sfX, \arrow[ur] & 
\end{tikzcd}\]
the full exceptional collection lifts to a collection of $\sfB^{\sct}$-equivariant objects $E_1, \dots E_n \in \Perf_{\sfB^{\sct}}(\sfX)$. Therefore Lemma \ref{ExceptionalCollection} implies the desired $\Perf_{\sfB^{\sct} \times \sfB^{\sct}}^{\ind}(\sfG) \simeq \Perf_{\sfB^{\sct} \times \sfB^{\sct}}(\sfG)$.
\end{proof}

We also used the following lemma, that is related to Theorem 2.6 of \cite{E09}.

\begin{lemma}\label{ExceptionalCollection}
Let $\sfX$ be a projective scheme acted on by an algebraic group $\sfH$. Assume that there exist equivariant objects $E_1, \dots E_n \in \Perf_{\sfH}(\sfX)$ such that after forgetting their $\sfH$-equivariant structures they form a full exceptional collection in $\Perf(\sfX)$.
Then $\Perf_{\sfH}(\sfX)$ is generated under finite colimits by $E_i \otimes M$ for $M \in \Rep(\sfH)$.
\end{lemma}
\begin{proof}
Let $F \in \Perf_{\sfH}(\sfX)$, and we will prove that it belongs to $\Perf_{\sfH}^{\ind}(\sfX)$.

For this we will inductively construct a sequence of objects $F \eqqcolon F_n, \dots F_0 \in \Perf_{\sfH}(\sfX)$, such that
\begin{enumerate}[label=(\roman*)]
\item\label{Exceptional1} $\RHom_{\Perf(\sfX)}(E_j, F_i) = 0$ for all $j > i$,
\item\label{Exceptional2} there is a map $F_i \rightarrow F_{i-1}[1]$ whose cone is of the form $E_i \otimes M$ where $M$ is a bounded complex of finite dimensional representations of $\sfH$.
\end{enumerate}

Assume that $F_n, \dots F_i$ have already been defined satisfying \ref{Exceptional1} and \ref{Exceptional2}. Define \[F_{i-1} \coloneqq \cone(\RHom_{\Perf(\sfX)}(E_i, F_i) \otimes E_i \rightarrow F_i).\] 
Using that $\sfX$ is projective, $\RHom_{\Perf(\sfX)}(E_i, F_i)$ is a bounded complex of finite dimensional representations of $\sfH$. Therefore $F_{i-1} \in \Perf_{\sfH}(\sfX)$ is perfect and it satisfies \ref{Exceptional2}.

Recall that the definition of an exceptional collection says
\[\RHom_{\Perf(\sfX)}(E_j, E_i) = \begin{cases} \bZ & j = i \\ 0 & j > i. \end{cases}\] Therefore by induction, $F_{i-1}$ also satisfies \ref{Exceptional1}. 

Moreover $F_0 \simeq 0$, because it is a perfect complex satisfying $\RHom_{\Perf(\sfX)}(E_j, F_0) \simeq 0$ for all $j$, and the exceptional collection is full. Therefore \ref{Exceptional2} implies that $F$ belongs to the subcategory generated under finite colimits by $E_i \otimes M$ for $M \in \Rep(\sfH)$.
\end{proof}

\begin{remark}
If $\sfG = \PGL(2)$ then $\cO, \cO(1) \in \Perf(\sfG/\sfB)$ form a full exceptional collection, but $\cO(1)$ admits only a $\sfB^{\sct}$-equivariant structure not a $\sfB$-equivariant structure.
\end{remark}

\section{Construction of the functor} \label{BiWhitFunctorSection}
Here we use centrality of Gaitsgory's functor to construct the bi-Whittaker functor. Most of the work is reducing to the case when $G$ has connected center by studying the fundamental group of Gaitsgory's family.

\subsection{Action of the center}
Here we define actions of $\sfZ$ on various coherent and constructible categories.

Let $\sfG^{\sct}_{\sfZ}$ denote the base change of $\sfG^{\sct}$ to a group scheme over $\sfZ$.

For the trivial action of $\sfZ$ on $\Rep(\sfG^{\sct})$, let \[\coact: \Rep(\sfG^{\sct}) \rightarrow \Rep(\sfG^{\sct}_{\sfZ}), \qquad A \mapsto A \otimes \cO(\sfZ)\]
be the monoidal coaction functor.

The action of $\sfZ$ on $\sfG^{\sct}$ induces a monoidal action of $\sfZ$ on the category of coherent sheaves on $\sfG^{\sct}/\sfG^{\sct}$.  The coaction functor \[\coact: \Free_{\sfG^{\sct}}(\sfG^{\sct}) \rightarrow \Free_{\sfG^{\sct}_{\sfZ}}(\sfG^{\sct}_{\sfZ})\]
is the monoidal functor given by pullback along $\sfG^{\sct}_{\sfZ} \rightarrow \sfG^{\sct}$. 

Let $n$ be the order of $\sfZ$ and let $\mu_n$ be the group of complex $n$th roots of unity. Then $\mu_n$ acts on $\bC(\!(t^{1/n})\!)$ by Galois automorphisms. Define \[\widetilde{G^{\ad}(\!(t)\!)} \coloneqq \{x \in G(\!(t^{1/n})\!) \text{ such that } \mu_n x \subset Z x\}.\]
Then $Z$ acts on $\widetilde{G^{\ad}(\!(t)\!)}$ and the quotient map is a covering space group homomorphism
\beq\label{CoverLoop} c: \widetilde{G^{\ad}(\!(t)\!)} \rightarrow G^{\ad}(\!(t)\!).\eeq

Let $\cH^{\ad}$ be the affine Hecke category for $G^{\ad}$ with integral coefficients. Then we get the following monoidal action of $\sfZ$ on $\cH^{\ad}$. 

Let $c_!c^!: \cH^{\ad} \rightarrow \cH^{\ad}$ be pullback then pushforward along the covering space induced by \eqref{CoverLoop}. Write $\cH^{\ad}_\sfZ$ for the affine Hecke category for $G^{\ad}$ with $\cO(\sfZ)$-coefficients.  If $A \in \cH^{\ad}$ then the group of deck transformations $Z$ acts on $c_!c^!A$. Therefore $c_!c^!$ lifts to a coaction functor fitting into a commutative triangle
\[\begin{tikzcd}
\cH^{\ad} \arrow[r, "\coact"] \arrow[dr, "c_!c^!"'] & \cH^{\ad}_{\sfZ} \arrow[d, "\oblv"] \\
& \cH^{\ad}.
\end{tikzcd}\]

\subsection{The fundamental group of Gaitsgory's family}
Here we prove that every local system on the affine flag variety can be extended to a local system on Gaitsgory's family. Moreover the loop rotation monodromy of its restriction away from zero is nontrivial depending on the connected component.

\begin{lemma}\label{LocalSystemsGaitsgory}
Recall that $\Fl_\bC$ denotes Gaitsgory's family, a $G/U$-bundle over $\Gr \times \bC$.
\begin{enumerate}[label=(\alph*)]
\item\label{LocalSystemsGaitsgory1} For each connected component, the maps \beq\label{FundamentalGroupComparision} \pi_1(G/U) \rightarrow \pi_1(\Fl) \rightarrow \pi_1(\Fl_{\bC})\eeq are both isomorphisms.
\item\label{LocalSystemsGaitsgory2} Let $u$ generate $\pi_1(\bC^{\times})$. On the connected component indexed by $z \in \pi_1(G)$, the map \[\pi_1(G/U \times \bC^{\times}) \simeq \pi_1(\Fl_{\bC^{\times}}) \rightarrow \pi_1(\Fl_\bC) \simeq \pi_1(G/U)\]
sends $(x, u^m) \mapsto xz^m$.
\end{enumerate}
\end{lemma}
\begin{proof}
First we prove \ref{LocalSystemsGaitsgory1}.
Commutativity of
\[\begin{tikzcd}
\pi_2(\Fl_{\bC^{\times}}) \arrow[r, "\sim"] \arrow[d] & \pi_2(\Gr \times \bC^{\times}) \arrow[d, "\sim"] \\
\pi_2(\Fl_{\bC}) \arrow[r] & \pi_2(\Gr \times \bC)
\end{tikzcd}\]
implies that $\pi_2(\Fl_\bC) \rightarrow \pi_2(\Gr \times \bC)$ is surjective. Moreover $\pi_1(\Gr \times \bC) \simeq 0$. Therefore the long exact sequence \[\pi_2(\Fl_\bC) \rightarrow \pi_2(\Gr \times \bC) \rightarrow \pi_1(G/U) \rightarrow \pi_1(\Fl_\bC) \rightarrow \pi_1(\Gr \times \bC),\]
gives an isomorphism
\[\begin{tikzcd}
\pi_1(G/U) \arrow[r] \arrow[dr, "\sim"'] & \pi_1(\Fl) \arrow[d] \\
& \pi_1(\Fl_\bC).
\end{tikzcd}\]
Moreover $\pi_1(\Gr) \simeq 0$ implies that $\pi_1(G/U) \rightarrow \pi_1(\Fl)$ surjective.
Hence both maps \eqref{FundamentalGroupComparision} are isomorphisms.

Now we prove \ref{LocalSystemsGaitsgory2}. The case of a torus follows by \eqref{DifferentTrivializationsTorus}. This implies the general case because
\[\begin{tikzcd}
\pi_1(T \times \bC^\times) \arrow[r, "\sim"] \arrow[d] & \pi_1(\Fl_{T, \bC^{\times}}) \arrow[r] \arrow[d] & \pi_1(\Fl_{T, \bC}) \arrow[d] \\
\pi_1(G/U \times \bC^\times) \arrow[r, "\sim"'] &  \pi_1(\Fl_{\bC^{\times}}) \arrow[r] & \pi_1(\Fl_\bC)  
\end{tikzcd}\]
commutes and the vertical maps are surjective.
\end{proof}

Using these properties of the fundamental group of Gaitsgory's family, we will deduce some additional structure on Gaitsgory's central functor.

\begin{lemma}\label{MonodromyTwistCenter}
For $A \in \Rep(\sfG^{\sct})$, there are functorial isomorphisms 
\beq \label{RepCommuteMonad} \coact Z(A) \simeq Z(\coact A).\eeq
Moreover if $\sfZ$ acts on $A$ by a character $z \in Z$, then
\beq\label{MonodromySquare} \begin{tikzcd}
\coact Z(A) \arrow[r, "\eqref{RepCommuteMonad}"] \arrow[dd, "\coact(m_A)"'] & Z(\coact A) \arrow[d, "Z(z)"]\\
& Z(\coact A) \arrow[d, "m_{\coact A}"] \\
\coact Z(A) \arrow[r, "\eqref{RepCommuteMonad}"']  & Z(\coact A) 
\end{tikzcd}\eeq
and
\beq\label{ZEquivariantMonoidal}
\begin{tikzcd}
\coact Z(A \otimes B) \arrow[r, "\sim"] \arrow[d, "\sim"'] & \coact (Z(A) * Z(B)) \arrow[r, "\sim"] & (\coact Z(A)) * (\coact Z(B)) \arrow[d, "\sim"] \\
Z(\coact(A \otimes B)) \arrow[r, "\sim"'] & Z((\coact A) \otimes (\coact B)) \arrow[r, "\sim"'] & Z(\coact A) * Z(\coact B) 
\end{tikzcd}\eeq
both commute.
\end{lemma}
\begin{proof}
Recall that Gaitsgory's central functor $Z: \Rep(\sfG^{\sct}) \rightarrow \cH^{\ad}$
is constructed using nearby cycles in the family $\Fl_\bC^{\ad}$ that degenerates $\Gr \times G^{\ad}/U \times \bC^{\times}$ towards $\Fl^{\ad}$.

Lemma \ref{LocalSystemsGaitsgory}\ref{LocalSystemsGaitsgory1} implies that there is a covering space $c: \widetilde{\Fl}_\bC^{\ad} \rightarrow \Fl_\bC^{\ad}$ that is given by \eqref{CoverLoop} on the zero fiber and $\Gr \times G/U \rightarrow \Gr \times G^{\ad}/U$ on each nonzero fiber.\footnote{However, importantly for \eqref{MonodromySquare}, the resulting cover of $\Gr \times G^{\ad}/U \times \bC^{\times}$ is not given by $\Gr \times G/U \times \bC^{\times} \rightarrow \Gr \times G^{\ad}/U \times \bC^{\times}$, rather it also has monodromy in the base $\bC^{\times}$-direction.}

Since nearby cycles commutes with pushforward and pullback along this cover, we get the desired isomorphisms \eqref{RepCommuteMonad}.
Moreover Lemma \ref{LocalSystemsGaitsgory}\ref{LocalSystemsGaitsgory2} gives commutativity of \eqref{MonodromySquare}.

Recall that the monoidal structure on Gaitsgory's functor $Z$ is defined using that nearby cycles commutes with pushforward along the convolution map $m: (\Fl^{\ad} \widetilde{\times} \Fl^{\ad})_{\bC} \rightarrow \Fl_{\bC}^{\ad}$ defined as in the proof of Proposition \ref{ZProperties}.

Consider the Cartesian square
\[\begin{tikzcd}
(\Fl^{\ad} \widetilde{\times} \Fl^{\ad})_{\bC} \times_{\Fl_{\bC}^{\ad}} \widetilde{\Fl}_{\bC}^{\ad} \arrow[r, "m"] \arrow[d, "c"'] & \widetilde{\Fl}_{\bC}^{\ad} \arrow[d, "c"] \\
(\Fl^{\ad} \widetilde{\times} \Fl^{\ad})_{\bC} \arrow[r, "m"'] & \Fl_{\bC}^{\ad}.
\end{tikzcd}\]
After applying the functor $\oblv : \cH^{\ad}_{\sfZ} \rightarrow \cH^{\ad}$ that forgets the $\cO(\sfZ)$-linear structure, \eqref{ZEquivariantMonoidal} identifies with the commutative hexagon
\[\begin{tikzcd}
c_!c^! \psi m_* \arrow[r, "\sim"] \arrow[d, "\sim"'] & c_!c^! m_* \psi \arrow[r, "\sim"] & m_* c_!c^! \psi \arrow[d, "\sim"] \\
\psi c_!c^! m_* \arrow[r, "\sim"'] & \psi m_* c_!c^! \arrow[r, "\sim"'] & m_* \psi c_!c^!.
\end{tikzcd}\]
Because this forgetful functor $\oblv$ is faithful on perverse sheaves, \eqref{ZEquivariantMonoidal} commutes as desired.
\end{proof}

\subsection{The Arkhipov--Bezrukavnikov functor is equivariant for the center}
Here we will prove that the Arkhipov--Bezrukavnikov functor for $G^{\ad}$ admits a $\sfZ$-linear structure. This is not needed for any other results in this paper, but it will be used in \cite{DT2}.

Let $\cA^{\ad} \subset \cH^{\ad}$ be the full additive subcategory of Wakimoto filtered perverse sheaves.

\begin{proposition}
The Arkhipov--Bezrukavnikov functor fits into a commutative diagram 
\beq\label{ZStructureSquareAB} \begin{tikzcd}
\Free_{\sfG^{\sct} \times \sfT^{\sct}}(\sfV^{\sct}) \arrow[r] \arrow[d, "\coact"']& \cA^{\ad} \arrow[d, "\coact"] \\
\Free_{\sfG^{\sct}_{\sfZ} \times \sfT^{\sct}_{\sfZ}}(\sfV^{\sct}_{\sfZ}) \arrow[r] & \cA^{\ad}_{\sfZ},
\end{tikzcd}\eeq
such that this natural isomorphism intertwines the monoidal structure.
\end{proposition}
\begin{proof}
Let $(W_{\lambda})_{\sfZ} \coloneqq W_{\lambda} \otimes_\bZ \cO(\sfZ) \in \cH^{\ad}_\sfZ$ denote the Wakimoto sheaf with $\cO(\sfZ)$-coefficients.
Then there are functorial isomorphisms 
\beq\label{WakiCommuteMonad} \coact W_{\lambda} \simeq (W_{\lambda})_{\sfZ},\eeq
fitting into a commutative diagram similar to \eqref{ZEquivariantMonoidal}. 

Moreover if $z \in Z$ is the image of $\lambda \in \Lambda$, then
\[\begin{tikzcd}
\coact W_{\lambda} \arrow[r, "\eqref{WakiCommuteMonad}"] \arrow[dd, "\coact (m_{\bZ_{\lambda}})"'] & (W_{\lambda})_{\sfZ} \arrow[d, "z"]\\
& (W_{\lambda})_{\sfZ} \arrow[d, "m_{\cO(\sfZ)_{\lambda}}"] \\
\coact W_{\lambda, S} \arrow[r, "\eqref{WakiCommuteMonad}"']  & (W_{\lambda})_{\sfZ} 
\end{tikzcd}\]
and
\[\begin{tikzcd}
\coact Z(M_{\lambda}) \arrow[r, "\eqref{RepCommuteMonad}"] \arrow[d, "\coact (b_{\lambda})"'] & Z(\coact M_{\lambda}) \arrow[d, "b_{\lambda}"] \\
\coact W_{\lambda} \arrow[r, "\eqref{WakiCommuteMonad}"'] & (W_{\lambda})_{\sfZ}
\end{tikzcd}\]
commute.
Combined with Lemma \ref{MonodromyTwistCenter}, we get the desired \eqref{ZStructureSquareAB}, providing a $\sfZ$-equivariant structure on the Arkhipov--Bezrukavnikov functor.
\end{proof}

\subsection{Centrality}
The following lemma says that the $\sfZ$-equivariance isomorphisms are compatible with the central structure.

Let $\cA^{\ad} \subset \cH^{\ad}$ be a full additive subcategory of perverse sheaves that is closed under convolution and contains all central sheaves.

\begin{lemma}\label{CentralityMonad}
For $A \in \Rep(\sfG^{\sct})$ and $B \in \cA^{\ad}$, the diagram
\beq\label{CentralityMonadRectangle} \begin{tikzcd}
(Z(\coact A)) * (\coact B) \arrow[r, "\sim"] \arrow[d, "\sim"'] & (\coact Z(A)) * (\coact B) \arrow[r] & \coact (Z(A) * B) \arrow[d, "\sim"] \\
(\coact B) * (Z(\coact A)) \arrow[r, "\sim"'] & (\coact B) * (\coact Z(A)) \arrow[r] & \coact (B * Z(A))
\end{tikzcd}\eeq
commutes.
\end{lemma}
\begin{proof}
Recall from Proposition \ref{Centrality} that the centrality isomorphism \[Z(A) * B \simeq \psi(C) \simeq B * Z(A)\]
is constructed by relating both sides to nearby cycles of $C \coloneqq S(A) \boxtimes B \boxtimes \bZ_{\bC^{\times}}$ in the family $\Fl_\bC'$ that degenerates $\Gr \times \Fl \times \bC^{\times}$ towards $\Fl$.

Let $D \coloneqq (S(A) \boxtimes R) \widetilde{\boxtimes} B \boxtimes \bZ_{\bC^{\times}}$. The first isomorphism \[Z(A) * B \simeq m_* \psi(D)\simeq \psi(m_*D) \simeq \psi(C)\] is constructed using that nearby cycles commutes with pushforward along $m: \Fl_\bC \widetilde{\times} \Fl \rightarrow \Fl_\bC'$.

Then we have a commutative diagram
\[\begin{tikzcd}
\oblv((Z(\coact A)) * (\coact B)) \arrow[r, "\sim"] \arrow[d, "\sim"'] & \oblv((\coact Z(A)) * (\coact B)) \arrow[r] \arrow[d, "\sim"] & \oblv (\coact (Z(A) * B)) \arrow[d, "\sim"] \\
m_* \psi c_! c^! D \arrow[r, "\sim"'] \arrow[d, "\sim"'] & m_* c_! c^! \psi D \arrow[r] & c_!c^! m_* \psi D \arrow[d, "\sim"] \\
\psi m_* c_! c^! D \arrow[r] & \psi c_! c^! m_* D \arrow[r, "\sim"'] & c_! c^! \psi m_* D.
\end{tikzcd}\]
By a similar argument there is also a commutative diagram
\[\begin{tikzcd}
\oblv((\coact B) * Z(A))  \arrow[r, "\sim"] \arrow[d, "\sim"'] & \oblv((\coact B) * (\coact Z(A))) \arrow[r] & \oblv (\coact (B * Z(A))) \arrow[d, "\sim"] \\
\psi m_* c_! c^! D \arrow[r] & \psi c_! c^! m_* D \arrow[r, "\sim"'] & c_! c^! \psi m_* D.
\end{tikzcd}\]
Therefore \eqref{CentralityMonadRectangle} commutes after applying the functor $\oblv: \cA^{\ad}_{\sfZ} \rightarrow \cA^{\ad}$ that forgets the $\cO(\sfZ)$-linear structure. Since this forgetful functor $\oblv$ is faithful on perverse sheaves, \eqref{CentralityMonadRectangle} commutes as desired.
\end{proof}

Using this we will obtain that Gaitsgory's functor to the Drinfeld center $\cA^{\ad}$ admits a $\sfZ$-equivariant structure.

\begin{proposition}\label{CentralZEquiv}
There is a monoidal functor $\Free_{\sfG^{\sct}}(\sfG^{\sct}) \rightarrow \Z(\cA^{\ad})$ fitting into a commutative triangle 
\beq\label{ForgetCentralTriangle} \begin{tikzcd}
\Free_{\sfG^{\sct}}(\sfG^{\sct}) \arrow[r] \arrow[dr] & \Z(\cA^{\ad}) \arrow[d] \\
 & \cA^{\ad}.
\end{tikzcd}\eeq
Moreover it fits into a commutative diagram
\beq\label{ZStructureSquare} \begin{tikzcd}
\Free_{\sfG^{\sct}}(\sfG^{\sct}) \arrow[r] \arrow[d, "\coact"']& \Z(\cA^{\ad}) \arrow[d, "\coact"] \\
\Free_{\sfG^{\sct}_{\sfZ}}(\sfG^{\sct}_{\sfZ}) \arrow[r] & \Z(\cA^{\ad}_{\sfZ}),
\end{tikzcd}\eeq
such that this natural isomorphism intertwines the monoidal structures.
\end{proposition}
\begin{proof}
Commutativity of \eqref{CentralityData} gives a monoidal functor \beq\label{MonodromyCentralFunctor} \Free_{\sfG^{\sct}}(\sfG^{\sct}) \rightarrow \Z(\cA^{\ad})\eeq
fitting into a commutative triangle \eqref{ForgetCentralTriangle}.

Lemma \ref{MonodromyTwistCenter} gives a commutative diagram
\[\begin{tikzcd}
\Free_{\sfG^{\sct}}(\sfG^{\sct}) \arrow[r] \arrow[d, "\coact"'] & \Z(\cA^{\ad}) \arrow[r] & \cA^{\ad} \arrow[d, "\coact"] \\
\Free_{\sfG^{\sct}_{\sfZ}}(\sfG^{\sct}_{\sfZ}) \arrow[r] & \Z(\cA^{\ad}_{\sfZ}) \arrow[r] & \cA^{\ad}_{\sfZ},
\end{tikzcd}\]
such that this natural isomorphism intertwines the monoidal structures.

Moreover Lemma \ref{CentralityMonad} gives a commutative diagram
\[\begin{tikzcd}
\Rep(\sfG^{\sct}) \arrow[d, "\coact"'] \arrow[r] &\Free_{\sfG^{\sct}}(\sfG^{\sct}) \arrow[r] & \Z(\cA^{\ad}) \arrow[d, "\coact"]\\
\Rep(\sfG^{\sct}_{\sfZ}) \arrow[r] & \Free_{\sfG^{\sct}_{\sfZ}}(\sfG^{\sct}_{\sfZ}) \arrow[r] & \Z(\cA^{\ad}_{\sfZ}).
\end{tikzcd}\]
Since $\Z(\cA^{\ad}) \rightarrow \cA^{\ad}$ is faithful and $\Rep(\sfG^{\sct}) \rightarrow \Free_{\sfG^{\sct}}(\sfG^{\sct})$ induces a bijection between isomorphisms classes of objects, we get the desired \eqref{ZStructureSquare}, providing a $\sfZ$-equivariant structure on \eqref{MonodromyCentralFunctor}.
\end{proof}

\subsection{The bi-Whittaker functor}
Now we construct the bi-Whittaker functor, using that the center of an algebra maps to the endomorphisms of any module for that algebra.

The commuting actions of $\Hfs$ and $\QCoh(\sfC)$ on $\Hwfs \simeq \QCoh(\sfT)$ endow $\wHws \simeq \wHs \otimes_{\Hfs} \Hwfs$ with a $\QCoh(\sfC)$-linear structure.

\begin{proposition}\label{BiWhitFunctor}
There exists monoidal $\QCoh(\sfC)$-linear functor \beq\label{bW} {}_{\chi}F_{\chi}: \IndPerf_\sfG(\sfG) \rightarrow \wHws,\eeq that preserves compactness and fits into a commutative square
\beq\label{BiWhitSquare} \begin{tikzcd}
\IndPerf_\sfG(\sfG) \arrow[d] \arrow[r, "{}_{\chi}F_{\chi}"] & \wHws \arrow[d, " - * {}^{\chi}\nabla_1"]    \\
\IndPerf_\sfG(\widetilde{\sfG}) \arrow[r, "{}_\chi F"']  & \wHs.
\end{tikzcd}\eeq
\end{proposition}
\begin{proof}
Let $\cA^{\ad} \subset \cH^{\ad}$ be the full monoidal subcategory of perverse sheaves that was considered in Lemma \ref{CentralityDerived}.

Proposition \ref{CentralZEquiv} gives $\sfZ$-equivariant functors fitting into a commutative diagram
\[\begin{tikzcd}
& \Z(\cA^{\ad}) \arrow[r] \arrow[d] \arrow[dr, phantom, "\eqref{CentralityDerivedSquare}"]& \Z(\cH^{\ad}) \arrow[r] \arrow[d] & {}_\chi \cH_\chi^{\ad} \arrow[d] \\
\Free_{\sfG^{\sct}}(\sfG^{\sct}) \arrow[r] \arrow[ur] \arrow[dr] & \cA^{\ad} \arrow[r] & \cH^{\ad} \arrow[r] & {}_\chi \cH^{\ad} \\
& \Free_{\sfG^{\sct} \times \sfT^{\sct}}(\sfV^{\sct}). \arrow[ur] & & 
\end{tikzcd}\]
Passing to $\sfZ$-equivariant objects and restricting tilting representations on which $\sfZ$ acts by the trivial character, we obtain $Z$-equivariant monoidal functors fitting into a commutative square
\[\begin{tikzcd}
\Free_{\sfG \times \sfZ}^{\tilt}(\sfG^{\sct}) \arrow[r] \arrow[d] & {}_\chi \cH_\chi \arrow[d] \\
\Free_{\sfG \times \sfT}^{\tilt}(\sfV) \arrow[r] & {}_\chi \cH.
\end{tikzcd}\]
Using Proposition \ref{IndPerfGsc}, passing to bounded homotopy categories and ind-extending gives the desired \eqref{BiWhitSquare}.

It remains to prove $\QCoh(\sfC)$-linearity. Proposition \ref{FunctorOpen} and commutativity of \eqref{BiWhitSquare} imply that ${}_{\chi}F_{\chi}(-) * {}^{\chi}\nabla_1$ is $\QCoh(\sfC)$-linear. Thus ${}_{\chi}F_{\chi}(-) * {}^{\chi}\nabla_1 * \Delta_1^{\chi}$ is also $\QCoh(\sfC)$-linear. But Lemma \ref{PittieSteinberg} implies that ${}_{\chi}F_{\chi}$ is a direct summand of ${}_{\chi}F_{\chi}(-) * {}^{\chi}\nabla_1 * \Delta_1^{\chi}$. Hence ${}_{\chi}F_{\chi}$ is $\QCoh(\sfC)$-linear as desired.
\end{proof}

We used the following lemma.

\begin{lemma}\label{PittieSteinberg}
Let $R^{\sct}$ be the ring of functions on $\sfT^{\sct}$. Then $(R^{\sct})^W$ is a direct summand of $R^{\sct}$, regarded as a $\sfZ$-equivariant $(R^{\sct})^W$-module.
\end{lemma}
\begin{proof}
Theorem 2.2 of \cite{S75} implies that $R^{\sct}$ is free over $(R^{\sct})^W$, and moreover that it admits a basis whose elements are of the form $e^{\lambda}$ for $\lambda$ a weight of $\sfT^{\sct}$. The $(R^{\sct})^W$-module spanned by each such basis element is stable under the $\sfZ$-action. Since $1 \in R^{\sct}$ is one of the basis elements, $(R^{\sct})^W$ is a $\sfZ$-equivariant direct summand.
\end{proof}

\section{The bi-Whittaker equivalence}\label{BiWhitEquivSection}
Here we prove the bi-Whittaker equivalence. We prove essential surjectivity by exhibiting a collection of generators of the bi-Whittaker category that is indexed by dominant weights. Then we prove fully faithfulness by pulling back along $\sfT \rightarrow \sfC$ to reduce to fully faithfulness of the Arkhipov--Bezrukavnikov functor.

\subsection{Compact generators}
Here we show that the bi-Whittaker category is compactly generated as a $\QCoh(\sfC)$-module by the images of vector bundles on $\sfG/\sfG$.

\begin{proposition}\label{BiwhitGenerators}
The bi-Whittaker category $\wHws$ is compactly generated as a $\QCoh(\sfC)$-module by ${}_{\chi}F_{\chi}(T_{\lambda} \otimes \cO)$ for $\lambda \in \Lambda^+$.
\end{proposition}
\begin{proof}
For support reasons, $\Hs$ is compactly generated by $\Delta_w * \Delta_{\lambda} * \Delta_v$ for $w, v \in W^{\fnt}$ and $\lambda \in \Lambda^+$. Therefore Lemma \ref{WhitFiniteLem} implies that $\wHws \simeq \wHs^{\fnt} \otimes_{\Hfs} \Hs \otimes_{\Hfs} \Hwfs$ is compactly generated by ${}^{\chi}\nabla_1 * \Delta_{\lambda} * \Delta_1^{\chi}$ for $\lambda \in \Lambda^+$.

Let $\wHws^{\leq \lambda}$ and $\wHws^{< \lambda}$ be the full subcategories generated by ${}^{\chi}\nabla_1 * \Delta_{\mu} * \Delta_1^{\chi}$ for $\mu \leq \lambda$ and $\mu < \lambda$. Then the image of ${}^{\chi}\nabla_1 * Z(T_{\lambda}) * \Delta_1^{\chi}$ in the quotient $\wHws^{\leq \lambda}/\wHws^{< \lambda}$ splits as a direct sum of several copies of the generator ${}^{\chi}\nabla_1 * \Delta_{\lambda} * \Delta_1^{\chi}$. Therefore ${}^{\chi}\nabla_1 * Z(T_{\lambda}) * \Delta_1^{\chi}$ generate ${}_\chi \cH_\chi$.

Moreover ${}^{\chi}\nabla_1 * Z(T_{\lambda}) * \Delta_1^{\chi}$ is obtained by tensoring ${}_{\chi}F_{\chi}(T_{\lambda} \otimes \cO)$ by $R \in \QCoh(\sfC)$. Therefore ${}_{\chi}F_{\chi}(T_{\lambda} \otimes \cO)$ generate ${}_\chi \cH_\chi$ as a $\QCoh(\sfC)$-module.
\end{proof}

\begin{remark}
The bi-Whittaker category $\wHws$ is always generated by ${}^{\chi}\nabla_1 * Z(T_{\lambda}) * \Delta_1^{\chi}$. If $G$ has disconnected center, then $\wHws$ is not generated as a category by ${}_{\chi}F_{\chi}(T_{\lambda} \otimes \cO)$, only as a $\QCoh(\sfC)$-module.
\end{remark}

\subsection{Proof of the bi-Whittaker equivalence}
Here we prove that the bi-Whittaker functor is an equivalence by pulling back along $\sfT \rightarrow \sfC$ and embedding both sides fully faithfully inside the Arkhipov--Bezrukavnikov category.

\begin{theorem}\label{BiWhitThm}
The monoidal functor ${}_{\chi}F_{\chi}: \IndPerf_\sfG(\sfG) \rightarrow \wHws$ is an equivalence.
\end{theorem}
\begin{proof}
Proposition \ref{BiwhitGenerators} implies essential surjectivity, so it remains to check fully faithfulness. Let \[H: \wHws \otimes_{\QCoh(\sfC)} \QCoh(\sfT) \simeq \wHs \otimes_{\Hfs} \QCoh(\sfT) \otimes_{\QCoh(\sfC)} \QCoh(\sfT) \rightarrow \wHs \simeq \wHs \otimes_{\Hfs} \Hfs\] be induced by $\Xi \otimes -: \QCoh(\sfT \times_{\sfC} \sfT) \rightarrow \Hfs$, the left adjoint to Soergel's functor.

By Proposition \ref{BiWhitFunctor}, there is a commutative diagram
\[\begin{tikzcd}
\IndPerf_\sfG(\sfG) \arrow[d, "{}_{\chi}F_{\chi}"'] \arrow[r] & \IndPerf_\sfG(\sfG \times_{\sfC} \sfT) \arrow[d, "{}_{\chi}F_{\chi} \otimes \id_{\sfT}"'] \arrow[r, "p^*"]  &  \IndPerf_\sfG(\widetilde{\sfG}) \arrow[d, "{}_{\chi}F"] \\
\wHws \arrow[r] & \wHws \otimes_{\QCoh(\sfC)} \QCoh(\sfT) \arrow[r, "H"'] & \wHs,
\end{tikzcd}\]
in which
\begin{enumerate}
\item[-] $p^*$ is fully faithful on $\Free^{\tilt}_{\sfG}(\sfG)$ by Proposition \ref{PushStructureSheaf},
\item[-] ${}_{\chi}F$ is fully faithful by Theorem \ref{Main},
\item[-] and $H$ is fully faithful by Equation \eqref{Endomorphismensatz}.
\end{enumerate}
Therefore ${}_{\chi}F_{\chi} \otimes \id_{\sfT}$ is fully faithful. Hence ${}_{\chi}F_{\chi}$ is fully faithful, because ${}_{\chi}F_{\chi}$ preserves compactness and $\sfT \rightarrow \sfC$ is faithfully flat.
\end{proof}

\begin{remark} As written, our proof of the bi-Whittaker equivalence uses the Whittaker equivalence as an input. However, the bi-Whittaker equivalence should be viewed as easier because $\sfG$ is affine unlike $\widetilde{\sfG}$. For this reason, it should be possible to prove Theorem \ref{BiWhitThm} directly using only the central functor and the order of vanishing calculations, but without the Drinfeld-Pl\"ucker arguments used to construct the Arkhipov--Bezrukavnikov functor. \end{remark}

\end{document}